\documentclass[letterpaper,12pt,]{amsart}
\usepackage{preambleFL}

\usepackage{enumerate}
 \usepackage[paper=letterpaper,right=3cm,left=3cm,top=2.5cm,bottom=2.5cm]{geometry}

\def\Day{\boxtimes}

\newcommand{\universal}{\zeta_n}
\newcommand{\overconf}[2][M]{\mathscr{F}_{#2 \downarrow #1}(k)}
\newcommand{\conf}[1]{\mathscr{F}_{#1}(k)}

\newif\ifcomment
\commentfalse
\newif\ifall
\allfalse

\newcommand{\myM}{\ensuremath{M}}
\newcommand{\moore}{\ensuremath{\mathbf{\Lambda}}}
\newcommand{\Gtop}{\ensuremath{G \mathrm{Top}}}
\newcommand{\topG}{\ensuremath{\mathrm{Top}_G}}
\newcommand{\mfdG}{\ensuremath{\mathrm{Mfld}^{\mathrm{fr}_V}_{G,n}}}
\newcommand{\ele}{\ensuremath{\nu}}

\newcommand{\subgroup}{\ensuremath{\subset}}
\newcommand{\overB}[1][B]{\ensuremath{^{\mathrm{h}}/_{#1}}}
\newcommand{\dimension}{cell dimension }

\title[Equivariant factorization homology]{A geometric approach to equivariant factorization homology and nonabelian Poincar\'e duality}
\author{Foling Zou}
\address{Department of Mathematics, University of Michigan, Ann Arbor, MI 48109}
\email{folingz@umich.edu}


\begin{document}

\maketitle

\begin{abstract}
Fix a finite group $G$ and an $n$-dimensional orthogonal $G$-representation $V$.
We define the equivariant factorization homology of a $V$-framed
smooth $G$-manifold with coefficients in an $\mathrm{E}_V$-algebra using a two-sided bar construction, generalizing \cite{Andrade,MK}.
This construction uses minimal categorical background and aims for maximal
concreteness, allowing convenient proofs of key properties, including invariance
of equivariant factorization homology under change of tangential structures. 
Using a geometrically-seen scanning map, we prove an equivariant version (eNPD) of the
nonabelian Poincar\'e duality theorem due to several authors. The eNPD states
that the scanning map gives a $G$-equivalence 
from the equivariant factorization homology 
to mapping spaces out the one-point compactification of
the $G$-manifolds, when the coefficients are $G$-connected.
For non-$G$-connected coefficients, when the $G$-manifolds have suitable copies
of $\mathbb{R}$ in them, the scanning map gives group completions.
This generalizes the recognition principle for $V$-fold loops
spaces in \cite{GMPermG}.
\end{abstract}

\section*{Acknowledgement}
This paper is mostly based on my thesis.
I would like to express my deepest gratitude to my advisor Peter May, who raises me up from the
kindergarten of mathematics.
I am indebted to Inbar Klang, Alexander Kupers and Jeremy Miller,
whose work motivates my research.
I would like to thank Asaf Horev and Haynes Miller for helpful conversations, Shmuel Weinberger for being my
committee member,
and the referee for making useful comments.

\tableofcontents

\section{Introduction}
\subsection{Factorization homology: history}
Factorization homology is a theory of invariants on manifolds with coefficients in
suitable $\mathrm{E}_n$-algebras.
The language has been used to formulate and solve questions
in many areas of mathematics. For example, there are homological stability results in
\cite{MK,Knudsen}, a reconstruction of the cyclotomic trace in \cite{AMGR}
and the study of quantum field theory in \cite{BZBJ,CG}.

Non-equivariantly, factorization homology has multiple origins.
The most well-known approach started in
Beilinson--Drinfeld's study of an algebraic geometric approach to conformal field theory \cite{BD}
under the name of chiral homology.
Lurie \cite[5.5]{HA} and Ayala--Francis \cite{AF15} introduced and extensively studied the algebraic
topology analogue, named as factorization homology.
This route relies heavily on $\infty$-categorical foundations.
An alternative geometric model is Salvatore's configuration spaces with summable
labels~\cite{Salvatore}.
This construction is close to the geometric intuition,
but is not homotopical.
Yet another model, using the bar construction and developed by Andrade~\cite{Andrade},
Miller~\cite{Miller} and Kupers--Miller~\cite{MK},
is homotopically well-behaved while staying close to the geometric intuition of
configuration spaces. It is this last approach that we will generalize
equivariantly.

To give context, we first give an introduction to this approach to non-equivariant factorization homology.
It is a classical theorem by Dold--Thom \cite{DT} that the ordinary integral
homology groups of a connected space $M$ are exactly the homotopy
groups the
configuration space on $M$ with summable labels in $\mathbb{N}$,
the commutative monoid of natural numbers.
Salvatore \cite{Salvatore} observed that one can form the configuration space on $M$ with summable labels
in an $\mathrm{E}_n$-algebra $A$, which has
less structure than a commutative monoid, if the space $M$ has the structure of a framed smooth manifold of dimension
$n$, because the local Euclidean chart of $M$ offers the way to sum the
labels in the $\mathrm{E}_n$-algebra $A$.
In \cite{Andrade, MK}, the authors used this idea and defined the factorization
homology of a framed smooth manifold $M$ with coefficients in an
$\mathrm{E}_n$-algebra $A$ to be the two sided bar construction
\begin{equation}
  \label{eq:intro-def1}
\int_M A = \mathrm{B}(\mathrm{D}_M, \mathrm{D}_n, A),
\end{equation}
where $\mathrm{D}_n$ is the monad associated to the little $n$-disks operad and
$\mathrm{D}_M$ is a certain functor associated to embeddings of disks in $M$.

This bar construction definition \autoref{eq:intro-def1} is a concrete point-set level model of the
$\infty$-categorical definition of \cite{HA,AF15}.
One can construct a topological category $\mathrm{Mfld}^{\mathrm{fr}}_n$
of framed smooth $n$-dimensional manifolds and framed embeddings,
which is a
common ground for both definitions. It is a symmetric monoidal
category under disjoint unions.
Let $\mathrm{Disk}^{\mathrm{fr}}_n$ be the full subcategory spanned by objects
equivalent to $\amalg_k \bR^n$ for some $k \geq 0$.
An $\mathrm{E}_n$-algebra $A$ can be viewed as a symmetric monoidal topological functor out of
$\mathrm{Disk}^{\mathrm{fr}}_n$.
The  $\infty$-categorical factorization homology \cite[definition 3.2]{AF15} is the derived symmetric monoidal topological left Kan extension of $A$
along the inclusion:  
\begin{equation}
  \label{eq:intro-def2}
  \begin{tikzcd}
    \mathrm{Disk}_n^{\mathrm{fr}} \ar[r,"A"] \ar[d,hook] & (\mathrm{Top},\times) \\
    \mathrm{Mfld}_{n}^{\mathrm{fr}} \ar[ur, dotted, "{\int_-A}"'] &
  \end{tikzcd}
\end{equation}
Horel \cite[7.7]{Horel} showed the equivalence of \autoref{eq:intro-def1} and \autoref{eq:intro-def2}.

\subsection{The definition of equivariant factorization homology}
We fix an integer $n$ and a finite group $G$ throughout.
An equivariant version of an $\mathrm{E}_n$-algebra is an $\mathrm{E}_V$-algebra,
where $\mathrm{E}_V$ is a monad associated to a $G$-operad that is equivalent to the little $V$-disks operad
$\mathscr{D}_V$ (see \autoref{sec:operad}).
The $\mathrm{E}_V$-algebras give the correct concrete coefficient input of equivariant
factorization homology on $V$-framed smooth $G$-manifolds.
Here, a smooth $G$-manifold  $M$ is $V$-framed if
there is a trivialization 
\begin{equation}
\label{eq:intro-frV}
\phi_M: \mathrm{T}M \cong M \times V
\end{equation} of its tangent bundle.

In line with \autoref{eq:intro-def1},
we define the equivariant factorization homology of a
$V$-framed smooth $G$-manifold $M$ with coefficients in an $\mathrm{E}_V$-algebra
$A$ to be (\autoref{defn:factoriaztion}):
\begin{equation}
  \label{eq:intro-defn-frV}
\int^{\mathrm{fr}_V}_M A = \mathrm{B}(\mathrm{D}_{\myM}^{\mathrm{fr}_V}, \mathrm{D}^{\mathrm{fr}_V}_V, A).
\end{equation}

\begin{rem}
 As will be made clear in \cite{KMZ},
$$ \mathrm{B}(\mathrm{D}_{\myM}^{\mathrm{fr}_V},
\mathrm{D}^{\mathrm{fr}_V}_V, A) \simeq  \mathrm{D}_{\myM}^{\mathrm{fr}_V}
\otimes_{\mathrm{D}^{\mathrm{fr}_V}_V} \int_V^{\mathrm{fr}_V} A,$$
where $\displaystyle\int_V^{\mathrm{fr}_V} A = \mathrm{B}(\mathrm{D}_{V}^{\mathrm{fr}_V},
\mathrm{D}^{\mathrm{fr}_V}_V, A)$. This bar construction is a cofibrant replacing of $A$ in
$\mathrm{D}^{\mathrm{fr}_V}_V$-algebra, and thus the equivariant factorization
homology could be understood as
first taking a cofibrant replacement,
and then extending from local to global by tensoring with
$\mathrm{D}_{\myM}^{\mathrm{fr}_V}$ over $\mathrm{D}^{\mathrm{fr}_V}_V$.
\end{rem}

We explain the definition \autoref{eq:intro-defn-frV} in a conveniently
generalized context.
A tangential structure is a $G$-map $\theta: B \to B_G
O(n)$ for some well-chosen $G$-space $B$\footnote{Non-equivariantly, $\theta$ is usually taken to be $B\Pi \to BO(n)$ for a subgroup $\Pi \subset O(n)$}. 
A morphism of two tangential structures is a $G$-map over $B_GO(n)$.
All tangential structures form a category $\mathcal{TS}$, which is simply the over category
$\Gtop/_{B_GO(n)}$.

Denote by $\universal$ the universal $G$-$n$-vector bundle over $B_GO(n)$. Pulling back along $\theta$ gives a bundle $\theta^{*}\universal$ over  $B$.  
  A $\theta$-framing on a smooth $G$-manifold $M$ is an equivariant bundle map $\phi_M:\mathrm{T}M \to \theta^{*}\universal$.
The $G$-manifold $M$ has a $\theta$-framing if and only if the classifying map
of its tangent bundle $\tau: M \to B_GO(n)$
 factors up to $G$-homotopy through $\theta: B \to B_GO(n)$.
Indeed, a $\theta$-framing on $M$ is the same
data as a map $\tau_B: M \to B$ plus a homotopy between the two classifying maps
$\tau$ and $\theta\circ\tau_B$ from
$M$ to $B_GO(n)$ (see \autoref{cor:HomVSMapoverB} with
\autoref{defn:mapOverB}).
  The $V$-framing \autoref{eq:intro-frV} is a special case:
  it is $\mathrm{fr}_V$-framing for a particular
  tangential structure $\mathrm{fr}_V: * \to B_GO(n)$.

In \autoref{sec:tangential}, we construct a $\Gtop$-enriched category
$\mathrm{Mfld}^{\theta}_{G,n}$, the category of smooth $n$-dimensional
$\theta$-framed $G$-manifolds and $\theta$-framed embeddings.
In particular, there is the category of $V$-framed smooth
$G$-manifold $\mfdG$.
It takes some effort to define the morphisms in the category.
For example, $V$-framed embeddings between little $V$-disks should be just the linear embeddings in the definition of the
little $V$-disks operad. However, we do
not have the notion of linear embeddings between general $V$-framed manifolds.
The solution is to allow all embeddings and to add in path data to correct the homotopy
type, so that we do not see the unwanted rotations. This idea goes back to Steiner \cite{Steiner}
and was used non-equivariantly by Andrade~\cite{Andrade} and Kupers--Miller \cite{MK}.
Using paths in the framing space, we define the $\theta$-framed embedding space
of $\theta$-framed manifolds (\autoref{def:embedding}).
This construction is covariant as a functor of $\theta$.

In \autoref{sec:eFH}, we use the $\Gtop$-enriched category
$\mfdG$ to build the $V$-framed factorization homology
by the bar construction \autoref{eq:intro-defn-frV}.
The representation $V$ can be viewed as a $G$-manifold with a
canonical $V$-framing, so each $\amalg_k V$ also has a canonical $V$-framing.
Let $\Lambda$ be the category of based finite sets $\mathbf{k} =
\{0,1,2,\cdots,k\}$ with base point $0$ and based injections.
For any $M$ in
$\mfdG$,
$\mathscr{D}_M^{\mathrm{fr}_V} ({k}) = \mathrm{Emb}^{\mathrm{fr}_V}(\amalg_k V, M)$
gives a functor $\Lambda^{op} \to \Gtop$.
Such functors 
$\mathscr{E}:\Lambda^{op} \to \Gtop$ and their associated functors
$\mathrm{E}: \Gtop_{*} \to \Gtop_{*} $
(\autoref{defn:functor}) give a convenient context for reduced operads and
monads, which we explain in \autoref{sec:Lambda}.

Taking $M = V$, compositions in $\mfdG$ equip the sequence $\mathscr{D}^{\mathrm{fr}_V}_V$
with the structure of a reduced $G$-operad. It is the endomorphism operad of the object $V$.
Moreover, it is equivalent to the little $V$-disks operad $\mathscr{D}_V$ (\autoref{cor:compareDV}), so it is an
$\mathrm{E}_V$-operad.
The functors associated to $\mathscr{D}^{\mathrm{fr}_V}_V$ and
$\mathscr{D}^{\mathrm{fr}_V}_{\myM}$ give a monad
$\mathrm{D}^{\mathrm{fr}_V}_V$ and a right $\mathrm{D}^{\mathrm{fr}_V}_V$-module
functor $\mathrm{D}_{\myM}^{\mathrm{fr}_V}$, and thus
\autoref{eq:intro-defn-frV} makes sense for a $\mathrm{D}^{\mathrm{fr}_V}_V$-algebra $A$.

For a tangential structure $\theta$ so that $V$ is $\theta$-framed 
(possible under the conditions on $\theta$ prescribed in \autoref{prop:22}), one
can define the
$\theta$-framed equivariant factorization homology with coefficient in a
$\mathrm{D}^{\theta}_V$-algebra $A$ as
\begin{equation}
  \label{eq:intro-theta-defn}
\int^{\theta}_M A = \mathrm{B}(\mathrm{D}_{\myM}^{\theta},
\mathrm{D}^{\theta}_V, A).
\end{equation}
Specializing to $\theta = \mathrm{fr}_V$, \autoref{eq:intro-theta-defn} gives \autoref{eq:intro-defn-frV}.
This construction is homotopically well-behaved.
\begin{prop}
  (\autoref{prop:homotopical}).
  The functor $\displaystyle\int_M^{\theta} -:
  \mathrm{D}_V^{\theta}[\Gtop_{*}]  \to \Gtop_{*}$ preserves weak equivalences.
\end{prop}

\subsection{Main results}
In \autoref{sec:embeddingspace}, we prove that the embedding space in  $\mfdG$
has a close connection to the configuration space. 
\begin{prop} (\autoref{cor:conf})
  \label{thm:intro1}
  Evaluating at $0$ of the embedding gives a $(G \times \Sigma_k)$-homotopy equivalence:
\begin{equation*}
  ev_0: \mathscr{D}^{\mathrm{fr}_{V}}_{\myM}(k) =
  \mathrm{Emb}^{\mathrm{fr}_V}(\amalg_k
  V, M) \overset{\simeq}{ \to } \conf{M}. 
\end{equation*}
\end{prop}
\noindent Here,  $\conf{M}$ is the ordered configuration space of
$k$ points in $M$.
This is used to justify that $\mathscr{D}^{\mathrm{fr}_V}_V$ is an $\mathrm{E}_V$-operad.

\medskip
We also prove an invariance result in the equivariant setting.
Such a result is known non-equivariantly \cite[Proposition 3.9]{AF15} and
expected equivariantly.
\begin{thm}
   (\autoref{prop:change-of-tangential})
  Let $q: \theta_1 \to \theta_2$ be a morphism of tangential structures
  and $V$ be $\theta_1$-framed. We also write $V$ for the $\theta_2$-framed
  $G$-manifold $q_{*}V$.
 Then for a $\theta_1$-framed $G$-manifold $M$ and a $\mathrm{D}_V^{\theta_2}$-algebra
  $A$, there is a $G$-equivalence
\begin{equation*}
\int^{\theta_1}_M q^{*}A \simeq \int^{\theta_2}_{q_{*}M} A.
\end{equation*}
\end{thm}
\noindent Due to the invariance, we may drop the $\theta$ from the notation $\displaystyle\int^{\theta}$
when the context is clear.

\medskip
The bar construction definition \autoref{eq:intro-theta-defn} stays close to the geometric origin, which readily leads to proofs of the following
results using classical techniques.
\begin{prop}
Equivariant factorization homology satisfies the following properties:  
\begin{enumerate}
\item (\autoref{prop:FHonV})
  \begin{align*}
     \int^{\theta}_VA & \simeq A. \\
\int^{\theta}_M \mathrm{D}^{\theta}_VA & \simeq \mathrm{D}^{\theta}_MA.
  \end{align*}
\item (\autoref{prop:FHonUnion})
\begin{equation*}
\int^{\theta}_{M \sqcup N}A \cong \int^{\theta}_MA \times
  \int^{\theta}_NA.
\end{equation*}
\end{enumerate}
\end{prop}

\medskip
In \autoref{chap:NPD}, we prove that our definition satisfies the following theorem.
\begin{thm}(\autoref{thm:NPDV} and \autoref{thm:grp})
  \label{thm:intro2}
 Let $M$ be a $V$-framed manifold and $A$ be a $\mathrm{D}^{\mathrm{fr}_V}_V$-algebra in $\Gtop$.
 There is a $G$-map:
\begin{equation*}
p_M: \int_M A \to \mathrm{Map}_*(M^+, \mathrm{B}^VA).
\end{equation*}
\begin{enumerate}
\item \label{item:eNPD}(eNPD) If $A$ is $G$-connected, $p_M$ is a weak  $G$-equivalence.
\item  If $V = W \oplus \bR$ and $M \cong N \times \mathbb{R}$ for a $W$-framed
  manifold $N$, then $p_M$ is a weak group completion (in the sense of
  \autoref{defn:weak-group-cplt}). 
\item If $V = U \oplus \bR^{2}$
  and $M \cong N \times \mathbb{R}^2$ for a $U$-framed
  manifold $N$, then $p_M$ is a group completion (in the sense of \autoref{defn:grp-complete}).
\end{enumerate}
\end{thm}
\noindent Here, $M^+$ is the one-point compactification of $M$; $\mathrm{B}^VA$
is a model for the $V$-fold deloop of $A$ defined in 
\autoref{sec:NPD}.

 In \autoref{thm:intro2}, part~\autoref{item:eNPD} is an equivariant version of the nonabelian Poincar\'e duality theorem
 due to several authors, including \cite[Theorem 6.6]{Salvatore}  and
 \cite[5.5.6.6]{HA};
 specializing to $M=V$ in \autoref{thm:intro2}, it recovers the equivariant
 recognition principle of \cite[Theorem 1.14]{GMPermG}.
In particular, if the $\mathrm{E}_V$-algebra $A$ is grouplike, then  $A \simeq \Omega^V\mathrm{B}^VA$.
  This justifies the definition of $\mathrm{B}^{V}A$.
  
\begin{cor}
  Let $M$ and $A$ be as in \autoref{thm:intro2} and $A$ be $G$-connected.
 \newline  Then we have
  $\displaystyle    \int_{G/H \times V} A \simeq
\mathrm{Map}_{*}(G/H_+, A).$
Therefore, $(\displaystyle\int_{G/H \times V} A)^G \simeq A^H.$
\end{cor}


\medskip
The map $p_M$ in the eNPD theorem is induced by a scanning map, a natural transformation of
 right $\mathrm{D}^{\mathrm{fr}_{V}}_{V}$-functors:
 \begin{equation}
   \label{eq:intro-scanning} 
\mathrm{D}^{\mathrm{fr}_V}_{\myM}(-) \to \mathrm{Map}_*(M^+,\Sigma^V-).
\end{equation}
The scanning map has been studied in various forms
in \cite{McDuff75, BM88, MT14}. In particular, Rourke--Sanderson \cite{RS00} proved that McDuff's scanning
map is a weak $G$-equivalence on $G$-connected objects.
Classically, given a configuration of $k$ points  in $M$, regarded as an
embedding of $\mathbf{k}$ to $M$,
the Pontryagin-Thom collapse gives an element of $\mathrm{Map}_{*}(M^+,
\wg_kS^n)$.
Note that the $i$-th wedge component $S^n$ is in fact the fiber at the image of
$i \in \mathbf{k}$ of the sphere bundle $\mathrm{Sph}(\mathrm{T}M)$.
The scanning map pushes the target further to the codomian
$\mathrm{Section}_c(M,\mathrm{Sph}(\mathrm{T}M))$ independent of $k$,
so that the individual Pontryagin-Thom maps vary continuously for the configurations.
To do this, one needs an identification of the normal bundle of the embedded points with the
tangent bundle of the manifold.
There are conceptually two ways to do this: to use geodesics to generate a canonical local
vector field (\cite{McDuff75}), or to fatten the configuration space to include the data of a tubular
neighborhood (\cite{MT14}).

In the $V$-framed case, we can give an easy definition of the
scanning map  \autoref{eq:defnsX}. 
In \autoref{chap:appendix-scanning}, we compare our scanning map to the scanning
maps in the literature. In particular, we prove in
\autoref{thm:htpyOfScanning} that equivariant versions of the scanning maps in \cite{McDuff75} and
\cite{MT14} are homotopic, which is claimed without proof in \cite[Remark
3.2]{MT14}.

\medskip
Our proof of eNPD has two steps. We sketch it out when $A$ is $G$-connected. The first step is to use the scanning map
\autoref{eq:intro-scanning}. It assembles to a simplicial map
  \begin{equation*}
\mathrm{B}_{\bullet}(\mathrm{D}^{\mathrm{fr}_V}_{\myM} , \mathrm{D}^{\mathrm{fr}_V}_V, A) \to  \mathrm{Map}_* (M^+, K_{\bullet} )
\end{equation*}
for a simplicial $G$-space $K_{\bullet}$ that realizes to $\mathrm{B}^VA$.
Using the Rourke--Sanderson result, the induced map on the geometric realization is a  weak $G$-equivalence
\begin{equation*}
  \int_MA = |\mathrm{B}_{\bullet}(\mathrm{D}^{\mathrm{fr}_V}_{\myM} , \mathrm{D}^{\mathrm{fr}_V}_V, A)|
  \to |\mathrm{Map}_* (M^+, K_{\bullet} )|.
\end{equation*}

The second step is to pull the $M^+$ out of the geometric realization.
The map
\begin{equation}
\label{eq:intro-map}
|\mathrm{Map}_{*}(M^+, K_{\bullet})| \to \mathrm{Map}_{*}(M^+, |K_{\bullet}|)
\end{equation}
 is a $G$-equivalence only when
$K_{\bullet}$ satisfies some connectivity conditions.
Non-equivariantly, for $M=\bR$ so that $M^+=S^1$ ,
a sufficient connectivity condition is given in \cite[Theorem
12.3]{MayGILS}.
Let $\nu$ be a function from the conjugacy classes of subgroups of $G$ to
$\mathbb{Z}_{\geq -1}$. We say a finite-dimensional based $G$-CW complex $X$ has \dimension $ \nu$ if
 its cells in the form of $G/H \times D^n$ have highest dimension $\nu(H)$. We define the function $\mathrm{dim}(X)$ to be
\begin{equation*}
\mathrm{dim}(X)(H) = \max\limits_{H \subgroup L} \nu(L).
\end{equation*}
Combining the non-equivariant result with induction shows: 
\begin{thm}(\autoref{thm:HM})
  If $X$ is a finite-dimensional based $G$-CW complex and
  $K_{\bullet}$ is a simplicial $G$-space such that for all $n$ and $H \subset G$, $K_n^H$ is
  $\mathrm{dim}(X)(H)$-connected, then $|\mathrm{Map}_{*}(X, K_{\bullet})| \to \mathrm{Map}_{*}(X, |K_{\bullet}|)$ is a weak $G$-equivalence.
\end{thm}
\noindent When $A$ is $G$-connected, the $K_{\bullet}$ constructed out of it satisfies
this connectivity condition, so the eNPD theorem follows.

\subsection{Comparison to other work}
In this paper we give a homotopical point set definition of equivariant
factorization homology generalizing \cite{Andrade}.\footnote{
Note that \cite{Andrade} is non-equivariant: their $G$ in $\mathrm{Emb}^G$ is a subgroup of $GL_n(\mathbb{R})$ and therefore refers to
a tangential structure $\theta: BG \to BGL_n(\mathbb{R})$.}
There are axiomatic approaches to $\infty$-categorical equivariant factorization homology
\cite{Horev,Weelinck} using $G$-$\infty$-categories and
$\infty$-$G$-categories respectively.
Our definition and \cite{Weelinck},
being generalizations of \autoref{eq:intro-def1} and \autoref{eq:intro-def2}
respectively, are equivalent.
\footnote{In \cite{Weelinck}, their $G$ is our $\theta: BG \to BO(n)$; their $\Gamma$  is our $G$; their $\rho$ is
our $V$; their $\Gamma^{\rho}\mathrm{Orb}^G_n$ is our $\mfdG$ with the adjustment that the morphisms
are replaced by the $G$-fixed points of the morphisms; their $\Gamma^{\rho}\mathrm{Disk}^G_n$-algebra is defined in
a symmetric monoidal category $\mathcal{C}$ whose objects do not necessarily
have $G$-actions, and a $\mathrm{D}^{\mathrm{fr}_V}_V$ algebra $A$ in $\Gtop$ in our
sense gives a  $\Gamma^{\rho}\mathrm{Disk}^G_n$-algebra in
$\mathcal{C} = \Gtop$ in their sense
by sending $G \times_H V \in \Gamma^{\rho}\mathrm{Disk}^G_n$ to $\mathrm{Map}(G/H, A)$.}
The definition of equivariant
factorization homology in \cite{Horev} is called ``genuine'', meaning that it considers $H$-manifolds for all subgroups
$H \subset G$.
Restricted  to $G$-manifolds, a theory of \cite{Horev} gives a theory of \cite{Weelinck}.

In joint work with Horev and Klang \cite{HKZ}, the author studies equivariant factorization homology of Thom
$G$-spectra in the context of \cite{Horev}. There,
a very different proof of the eNPD theorem adapted to the $\infty$-categorical context is given,
generalizing Corollary 4.6 of \cite{AF15}.
The alternative proof is an axiomatic one, based on equivariant
handle-body decompositions of the $G$-manifold $M$.
In contrast, we provide a geometrically-seen scanning map that gives the
equivalence in this paper.
The scanning map was used to prove homological
stability properties of non-equivariant configuration spaces and factorization homology in
\cite{McDuff75, Miller, MK}.
The approach in our paper should lead to equivariant stability
results.

Another advantage of our approach to the equivariant factorization
homology and the eNPD theorem is that it
gives a simplicial filtration on the mapping space
$\mathrm{Map}_{*}(M^+, Y)$ (taking $A = \Omega^VY$),
thus offering a spectral sequence.
It could be useful for obtaining equivariant generalizations of 
\cite{CTHomologyFunction}. 
However, as computations of equivariant homology of the free $\mathrm{E}_V$-algebra on $A$,
$\mathrm{H}_{\star}^G (\mathrm{D}_V^{\mathrm{fr}_V} A)$,
 and in general,
 $\mathrm{H}_{\star}^G (\mathrm{D}_M^{\mathrm{fr}_V} A)$,
 remains open for any coefficients,
 this computational tool has not yet been explored.

Our definition of 
$\mathrm{Mfld}^{\theta}_{G,n}$  in \autoref{sec:tangential} is closely related
to Ayala--Francis \cite{AF15}, which we compare in \autoref{chap:homspace}. For the
trivial tangential structure $\mathrm{id}: B_GO(n) \to B_GO(n)$,
we have $\mathrm{Mfld}^{\mathrm{id}}_{G,n} \simeq \mathrm{Mfld}_{G,n}$. The category
$\mathrm{Mfld}^{\theta}_{G,n}$ is a pullback of
$\mathrm{Mfld}^{\mathrm{id}}_{G,n}$ induced by the map tangential structure
$\theta \to \mathrm{id}$. We also identify the automorphism
$G$-space $\mathrm{Emb}^{\theta}(V,V)$ in \autoref{thm:autoV}.

\subsection{Notations}
\label{sec:notations}
\begin{itemize}
\item  $\Gtop$ is the $\mathrm{Top}$-enriched category of $G$-spaces and $G$-equivariant maps.
\item  $\topG$ is the $\Gtop$-enriched category of $G$-spaces
  and non-equivariant maps where $G$ acts by conjugation on the mapping space.
\end{itemize}

For a space $M$ and a fiber bundle $E \to M$, 
\begin{itemize}
\item $\conf{M}$ is the ordered configuration space of $k$ points in $M$.
\item $\overconf{E}$ is the ordered configuration space of $k$ points in $E$ whose images are $k$
  distinct points in $M$.
\end{itemize}


\section{Preliminaries on operads and equivariant bundles}
\label{part:preliminary}
\subsection{$\Lambda$-sequences and operads}
 \label{sec:Lambda}
 To streamline the monadic bar construction in the main body,
we use an elementary categorical framework of
$\Lambda$-objects. This framework is studied in more detail in a paper with May
and Zhang \cite{MZZ}.
  This subsection is a summary of the relevant content towards
  \autoref{exmp:DMDV} and \autoref{prop:circledistribute}, which are used in later sections.

Let $\Lambda$ be the category of based finite sets $\mathbf{k} = \{0,1,2,\cdots,k\}$ with base point
$0$ and based injections.
The morphisms of $\Lambda$ are generated by permutations and the ordered injections
 $s_i^{k} : \mathbf{k-1} \to \mathbf{k}$ that skip $i$ for $1 \leq i \leq k$.
 It is a symmetric monoidal category with wedge 
 sum as the symmetric monoidal product.

 For a symmetric monoidal category $(\mathscr{V},\otimes,\mathcal{I})$,
 let $ \mathscr{V}_{\mathcal{I}}$ be the category under the unit.
In \cite{MZZ}, $\mathscr{V}$ is more general, but here we will work only
with the  Cartesian monoidal category $(\Gtop, \times, *)$.
The empty $G$-space $\varnothing$ is an initial object.

\begin{defn}
 A $\Lambda$-sequence in $\Gtop$ is a functor $\mathscr{E}: \Lambda^{op} \to \Gtop$.
 We write $\mathscr{E}(k)$ for $\mathscr{E}(\mathbf{k})$.
 It is called unital if $\mathscr{E}(0) = *$.
 The category of all $\Lambda$-sequences in $\Gtop$ is denoted
 $\Lambda^{op}[\Gtop]$, where morphisms are natural transformations of functors.
 The category of all unital $\Lambda$-sequences in $\Gtop$ is denoted
 $\Lambda^{op}_{*}[\Gtop]$,
 where morphisms are natural transformations of functors that are identity at level zero.
\end{defn}

 The category $\Lambda^{op}[ \Gtop]$ admits a symmetric monoidal structure
 $(\Lambda^{op}[\Gtop], \Day, \mathscr{I}_0)$. Here, $\Day$ is the Day
 convolution of functors on the closed symmetric monoidal category $\Lambda^{op}$. The unit
 is given by
\begin{equation*}
\mathscr{I}_{0}(n) =
\begin{cases}
  *, & n=0;\\
  \varnothing, & n>0;
\end{cases}
\end{equation*}
The symmetric monoidal product $\Day$ on $\Lambda^{op}[\Gtop]$
induces a symmetric monoidal product on
$\Lambda^{op}[\Gtop]_{\mathscr{I}_0}$ and $\Lambda^{op}_{*}[\Gtop]$, which we still denote by $\Day$.

The categories
$\Lambda^{op}[\Gtop]_{\mathscr{I}_0}$ and  $\Lambda^{op}_{*}[\Gtop]$ admit a second
(nonsymmetric) monoidal product
 $\odot$ in addition to $\Day$, called the circle product.  
 It is analogous to Kelly's circle product on symmetric
 sequences \cite{Kelly05}. The unit for $\odot$ is given by
  $$
\mathscr{I}_{1}(n) =
\begin{cases}
  *, & n=0,1;\\
  \varnothing, & n>1;
\end{cases} $$
where the only non-trivial morphism $\mathscr{I}_1(1) \to \mathscr{I}_1(0)$ is the identity. For a
brief definition of $\odot$, see \autoref{defn:LambdaObjectCircle}~\autoref{item:associate2}.

An operad in $\Gtop$, as defined in \cite{May84}, gives an example of a symmetric sequence in
$\Gtop$. If the operad is unital, meaning the $0$-space of the operad is the unit, it has the
structure of a $\Lambda$-sequence in $\Gtop$.
A unital operad in $\mathrm{Top}$ or $\Gtop$, is also called a reduced
operad in \cite{May84}.
In fact,
we have the unital variant of Kelly's observation \cite{Kelly05}:
\begin{thm}(\cite[Theorem 0.10]{MZZ})
  \label{thm:operad}
  A unital operad in $\Gtop$ is a monoid in the monoidal category
   $(\Lambda^{op}_{*}[\Gtop], \odot, \mathscr{I}_1)$. 
\end{thm}

\medskip
We give a construction which will be used in the definition of equivariant factorization homology:
the associated functor of a unital $\Lambda$-sequence. This construction specializes to the 
monad associated to a reduced operad of \cite{May84}; it
also appears in the definition of the circle product $\odot$.
Assume that $(\mathscr{W},\otimes,\mathcal{J})$ is a cocomplete symmetric monoidal
category tensored over  $\Gtop$.
\begin{con} \label{cons:covariantFunctorXdot}
  Let $X \in
  \mathscr{W}_{\mathcal{J}}$ be an object under the unit. Define $X^{*}: \Lambda
\to \mathscr{W}$ to be the covariant functor that 
sends $\mathbf{n}$ to $X^{\otimes n}$. 
On morphisms, it sends the permutations to permutations of the $X$'s
and sends the injection $s_i^k: \mathbf{k-1} \to \mathbf{k}$ for $1 \leq i \leq k$ to the map
\begin{equation*}
\begin{tikzcd}
(s_i^k)_{*}:  X^{\otimes {k-1}} \cong X^{\otimes {i-1}} \otimes \mathcal{J} \otimes
X^{\otimes {k-i}}
\arrow{rr}{\mathrm{id}^{i-1} \otimes \eta \otimes \mathrm{id}^{k-i}} & &
X^{\otimes k}, 
\end{tikzcd}
\end{equation*}
where $\eta:
\mathcal{J} \to X$ is the unit map of $X$.
By convention, $X^{\otimes 0} = \mathcal{J}$.
\end{con}

 This defines a functor $(-)^{*}: \mathscr{W}_{\mathcal{J} } \to
  \mathrm{Fun}(\Lambda, \mathscr{W})$.
  Then one can form the categorical tensor product over $\Lambda$ of
  the contravariant functor $\mathscr{E}$ and the covariant functor $X^{*}$.

\begin{con}\label{defn:functor}
  Let $\mathscr{E} \in \Lambda^{op}_{*}[\Gtop]$ 
   be a unital $\Lambda$-sequence.  The functor 
\begin{equation*}
\mathrm{E}: \mathscr{W}_{\mathcal{J} } \to \mathscr{W}_{\mathcal{J}}
\end{equation*}
associated to $\mathscr{E}$ is defined to be
\begin{equation*}
\mathrm{E}(X) =  \mathscr{E} \otimes_{\Lambda} X^{*} = \coprod_{k \geq 0} \mathscr{E}(k)
\otimes  X^{\otimes k}/\approx,
\end{equation*}
where $(\alpha^{*} f, \mathbf{x}) \approx (f, \alpha_{*} \mathbf{x})$
for all $f \in \mathscr{E}(m)$, $ \mathbf{x} \in X^{\otimes n}$ and $\alpha \in \Lambda(\mathbf{n}
, \mathbf{m})$. The unit map of $\mathrm{E}(X)$ is given by 
$\mathcal{J} \cong * \otimes \mathcal{J} \cong \mathscr{E}(0) \otimes X^{\otimes 0} \to \mathrm{E}(X)$.
\end{con}

\begin{rem}
  \label{rmk:equi-relation}
It is sometimes useful to take the quotient in two steps and use the following alternative formula
for $\mathrm{E}$:
  \begin{equation*}
\mathrm{E}(X) =  \coprod_{k \geq 0} \mathscr{E}(k) \otimes_{\Sigma_k}  X^{\otimes k}/ \sim,
\end{equation*}
where $[(s_i^{k})^{*} f, \mathbf{x}] \sim [f, (s^{k}_{i})_{*} \mathbf{x}]$
for all $f \in \mathscr{E}(k)$, $ \mathbf{x} \in X^{\otimes k-1}$.
We will use $\approx$ or $\sim$ for the equivalence relation to be clear which
formula we are using and refer to $\sim$ as the base point identification.
\end{rem}

\begin{con} We focus on the following context of \autoref{defn:functor}.
\label{defn:LambdaObjectCircle}
  \begin{enumerate}
\item \label{item:associate1} Letting $\mathscr{W} = \Gtop$, one gets from
  $\mathscr{C} \in \Lambda^{op}_{*}[\Gtop]$ an associated functor:
\begin{equation*}
\mathrm{C}: \Gtop_{*}  \to \Gtop_{*}.
\end{equation*}

\item \label{item:associate2} Let $\mathscr{W} = (\Lambda^{op}[\Gtop], \Day, \mathscr{I}_{0})$ with the Day monoidal structure.
  Then $\mathscr{W}$ is tensored over $\Gtop$ in the obvious way by levelwise tensoring.
  One gets the circle product for
  $\mathscr{E} \in \Lambda^{op}_{*}[\Gtop]$ and $ \mathscr{F} \in \Lambda^{op}[\Gtop]_{\mathscr{I}_0}$:
\begin{equation*}
\mathscr{E} \odot \mathscr{F} := \mathscr{E} \otimes_{\Lambda} \mathscr{F}^{*} \in  \Lambda^{op}[\Gtop]_{\mathscr{I}_0}.
\end{equation*}
\end{enumerate}
\end{con}

These two cases are further related:  the $0$-th level functor
\begin{equation*}
\imath_0:\Gtop_{*} \to\Lambda^{op}[\Gtop]_{\mathscr{I}_{0}} , \ (\imath_0X)(n) =
\begin{cases}
  X, & n=0;\\
  \varnothing, & n>0;
\end{cases}
\end{equation*}
gives an inclusion of a full symmetric monoidal subcategory, so we have
\begin{equation}
  \label{eq:monad-and-circle}
\imath_0 (\mathrm{E}X) = \imath_0( \mathscr{E} \otimes_{\Lambda} X^{*}) \cong \mathscr{E}
\otimes_{\Lambda} (\imath_0(X)^{*}) = \mathscr{E} \odot \imath_0 X.
\end{equation}
In words, the reduced monad construction is what happens at the 0-space of the circle product.
Using this, one can show:
\begin{prop}(\cite[Proposition 6.2]{MZZ}) \label{prop:functorOfCircle}
  Let $\mathrm{E},\mathrm{F}: \Gtop_{*} \to \Gtop_{*}$
  be the functors associated to $\mathscr{E}$ and
  $\mathscr{F}$. Then the functor associated to $\mathscr{E} \odot \mathscr{F}$ is $\mathrm{E} \circ
  \mathrm{F}$. 
\end{prop}

A monad is a monoid in the functor category.
Using the associativity of the circle product and \autoref{eq:monad-and-circle}, it is easy
to prove that
when $\mathscr{C}$ is an operad, the associated functor $\mathrm{C}$ is a
monad; and that when $\mathscr{F}$ is a left/right module over the monoid
$\mathscr{C}$ in $(\Lambda^{op}_{*}[\Gtop], \odot)$, the
associated functor $\mathrm{F}$ is a left/right module over $\mathrm{C}$.
The following construction gives examples.
\begin{con}(\cite[Section 8]{MZZ})
  \label{prop:endoperad}
Suppose that we have a $\Gtop$-enriched symmetric monoidal category $(\mathscr{W}, \otimes,
\mathcal{J})$ such that
$\ul{\mathscr{W}}(\mathcal{J}, Y) \cong *$ for all objects $Y$
of $\mathscr{W}$. 
Then we can construct a $\Lambda^{op}_{*}[\Gtop]$-enriched category
$\mathcal{H}_{\mathscr{W}}$. The objects are the same as those of
$\mathscr{W}$, while the enrichment is given by 
\begin{equation*}
\ul{\mathcal{H}_{\mathscr{W}}}(X,Y) =  \ul{\mathscr{W}}(X^{\otimes *},Y).
\end{equation*}
The definition of the composition in $\mathcal{H}_{\mathscr{W}}$ is similar to the structure maps of
an endomorphism operad. So, for any objects $X, Y, Z$ of $\mathscr{W}$,
$\ul{\mathcal{H}_{\mathscr{W}}}(Y,Y)$ is a monoid in $(\Lambda^{op}_{*}[\Gtop],
\odot)$, $\ul{\mathcal{H}_{\mathscr{W}}}(X,Y)$ is a left module over it, and
$\ul{\mathcal{H}_{\mathscr{W}}}(Y,Z)$ is a right module.
By \autoref{thm:operad}, $\ul{\mathcal{H}_{\mathscr{W}}}(Y,Y)$ is a unital operad,
and it is called the endomorphism operad of $Y$.
The assumption $\ul{\mathscr{W}}(\mathcal{J}, Y) \cong *$ is
automatically satisfied if $\mathscr{W}$ is coCartesian monoidal.
\end{con}

\begin{exmp} \label{exmp:DMDV}
 In \autoref{sec:tangential}, we construct a $\Gtop$-enriched category
 $(\mathrm{Mfld}^{\theta}_{G,n},\amalg, \varnothing)$ with a designated element $V \in \mathrm{Mfld}^{\theta}_{G,n}$. Applying
 \autoref{prop:endoperad} to $\mathscr{W} = \mathrm{Mfld}^{\theta}_{G,n}$,
 we obtain for any $M \in \mathrm{Mfld}^{\theta}_{G,n}$ a $\Lambda$-sequence
\begin{equation*}
\mathscr{D}_{\myM}^{\theta} = \ul{\mathcal{H}_{\mathscr{W}}}(V, M).
\end{equation*}
Then, $\mathscr{D}^{\theta}_V = \ul{\mathcal{H}_{\mathscr{W}}}(V, V)$
is a monoid in $(\Lambda^{op}_{*}[\Gtop], \odot)$
and $\mathscr{D}_{\myM}^{\theta}$ is a right module over it.
Translating by \autoref{thm:operad}, $\mathscr{D}^{\theta}_V$ is a reduced
operad in $(\Gtop , \times)$.
By \autoref{prop:functorOfCircle},  $\mathrm{D}^{\theta}_V$ is a monad and $\mathrm{D}^{\theta}_{\myM}$ is a right
module over $\mathrm{D}^{\theta}_V$.
\end{exmp}

We will use that the circle product is strong symmetric monoidal in the first variable:
\begin{prop}(\cite[Proposition 4.7]{MZZ})
  \label{prop:circledistribute}
  For any $\mathscr{E} \in \Lambda^{op}[\Gtop]_{\mathscr{I}_0}$,
  the functor $ - \odot \mathscr{E}$ on $(\Lambda^{op}(\Gtop)_{\mathscr{I}_{0}},
  \Day ,\mathscr{I}_{0})$ is strong symmetric monoidal. That is, the circle product
  distributes over the Day convolution:
  for any $\mathscr{D}, \mathscr{D}' \in \Lambda^{op}(\Gtop)_{\mathscr{I}_{0}}$, we have 
\begin{equation*}
(\mathscr{D} \Day  \mathscr{D}') \odot \mathscr{E} \cong (\mathscr{D} \odot
\mathscr{E}) \Day  (\mathscr{D}' \odot \mathscr{E}).  
\end{equation*}
\end{prop}

\subsection{Equivariant bundles}
\label{sec:equiBundle}
As pointed out in the introduction, we define $\theta$-framed embeddings using
maps between
equivariant bundles.
In this subsection, we list some preliminary results on equivariant vector
bundles for the reader's reference. The proofs of the results as well as a clarification of different notions
of equivariant fiber bundles can be found in  \cite{ZouBundle}.

Let $G$ and $\Pi$ be compact Lie groups, where $G$ is the ambient
action group and $\Pi$ is the structure group.
\begin{defn} 
  \label{defn:Gvector}
  A $G$-$n$-vector bundle a map $p:E \to B$ such that the following statements hold:
  \begin{enumerate}
    \item The map $p$ is a non-equivariant $n$-dimensional vector bundle;
    \item Both $E$ and $B$ are $G$-spaces and $p$ is $G$-equivariant;
    \item \label{item:Gvector3} The $G$-action is linear on fibers.
    \end{enumerate}
\end{defn}

\begin{defn}
  \label{defn:Gprincipal}
  A principal $G$-$\Pi$-bundle is a map ${p:P \to B}$ such that the following statements hold:
    \begin{enumerate}
      \item The map $p$ is a non-equivariant principal $\Pi$-bundle;
      \item Both $P$ and $B$ are $G$-spaces and $p$ is $G$-equivariant;
      \item The actions of $G$ and $\Pi$ commute on $P$.
    \end{enumerate}
\end{defn}

\begin{rem}
  This is called a principal $(G,\Pi)$-bundle in \cite[IV1]{LMS86}.
\end{rem}

  \begin{thm}
    \label{thm:G-structure-1}
    There is an equivalence of categories between
    \{$G$-$n$-vector bundles over $B$\}
    and \{principal $G$-$O(n)$-bundles over $B$\}. 
  \end{thm}
The classical procedure of passing from $n$-vector bundles to principal $O(n)$-bundles is
called taking the space of admissible maps. The equivariant bundles mentioned
are both just non-equivariant bundles with $G$-actions, and the classical
procedure is compatible with the $G$-actions.

\medskip
A $G$-vector bundle $E \to B$ is $V$-trivial for some $n$-dimensional $G$-representation
$V$ if there is a $G$-vector bundle isomorphism $E \cong B \times V$.
Such an isomorphism is called a $V$-trivialization or $V$-framing of the bundle. This is analogous to the
case of non-equivariant vector bundles, except that equivariance adds in the representation $V$ that's part of the data.
However, the representation $V$ in the equivariant trivialization of a fixed vector bundle may not be unique. 

  \begin{exmp}(\cite[Examples 3.4 and 3.5]{ZouBundle})
 
\begin{enumerate}
\item Let $G=C_2$, $\sigma$ be the sign representation. The unit sphere, $S(2\sigma)$, is $S^1$ with the
  180 degree rotation action. As $C_2$-vector bundles,  
\begin{equation*}
S(2\sigma) \times \mathbb{R}^2 \cong S(2\sigma) \times 2\sigma.
\end{equation*}
\item  Take $V$ and $W$ to be any two representation
  of $G$ that are of the same dimension and take $B$ to have free $G$-action. Then $B \times V  \cong B \times W$.
\end{enumerate}  
  \end{exmp}
  
We do have the uniqueness of $V$ in the following case (\cite[Corollary 3.2]{ZouBundle}).
  \begin{prop}
\label{cor:trivialbundle}
    If $B$ has a $G$-fixed point, then $B \times V \cong B \times W$ only when $V \cong W$.
  \end{prop}

Equivariantly, $G$-representations serve the role of $\bR^n$. So it is natural
to consider the $V$-framing bundle $\mathrm{Fr}_V(E)$ for an orthogonal
  $n$-dimensonal representation $V$.
  
\begin{defn}
    \label{defn:frv}
  Let $p:E \to B$ be a $G$-$n$-vector bundle.
  Let $\mathrm{Fr}_V(E)$ be the space of admissible maps with the $G$-action
$  g(\psi) = g \psi \rho(g)^{-1}.$
\end{defn}
\noindent In other words, $\mathrm{Fr}_V(E)$ has the same underlying space as
$\mathrm{Fr}_{\bR^n}(E)$, but we think of admissible maps as mapping out of $V$ instead of
$\bR^n$. 

\medskip
Let  $H \subgroup G$ be a subgroup and $\mathrm{Rep}(H,\Pi)$ be the set:
\begin{equation*}
\mathrm{Rep}(H,\Pi)  = \{ \text{group homomorphism }\rho: H \to \Pi\}/ \Pi
 \text{-conjugation}.
\end{equation*}

A group homomorphism $\rho: H \to \Pi$ gives a subgroup $\Lambda_{\rho} \subgroup (\Pi \times G)$
via its graph:
$$\Lambda_{\rho} = \{(\rho(h) ,h)| h \in H\}.$$
Denote  the centralizer  of the image of $\rho$ in $\Pi$ by
$\mathrm{Z}_{\Pi}(\rho)$. It is a closed subgroup of $\Pi$, and we define
\begin{equation*}
  \mathrm{Z}_{\Pi}(\rho)  = \Pi \cap \mathrm{Z}_{\Pi \times G}(\Lambda_{\rho})
  = \{\ele \in \Pi| \ele \rho(h) = \rho(h) \ele  \text{ for all } h \in H\}.
\end{equation*}

Take $p: P \to B$ to be a principal $G$-$\Pi$-bundle. Then each
component $B_0 \subset B^H$ is associated to a homomorphism $[\rho] \in \mathrm{Rep}(H,\Pi)$: 
\begin{thm}
  \label{thm:LM} There is a well-defined map
  $\pi_0^H(B) \to \mathrm{Rep}(H,\Pi)$ by
  $$B_0 \mapsto \{\rho: H \to \Pi | \big(p^{-1}(B_0)\big)^{\Lambda_{\rho}} \not= \varnothing \}.$$
Furthermore, for any fixed representative $\rho$,  $ \big(p^{-1}(B_0)\big)^{\Lambda_{\rho}}  \to
B_0$ is a principal $Z_{\Pi}(\rho)$-bundle and $p^{-1}(B_0) \cong
\Pi \times_{Z_{\Pi}(\rho)} \big(p^{-1}(B_0)\big)^{\Lambda_{\rho}}$.
\end{thm}
This is essentially \cite[Theorem 12]{LM86} and is explained in \cite[Section
2.6]{ZouBundle}. Note that a principal $G$-$\Pi$-bundle morphism preserves the associated homomorphism $[\rho]$.

\medskip
There is a notion of the universal $G$-$\Pi$-bundle $E_G\Pi \to B_G\Pi$, so that principal
$G$-$\Pi$-bundles over a base $G$-space $B$ are classified by $G$-homotopy
classes of maps from $B$ to $B_G\Pi$.
 We denote the universal $G$-$n$-vector bundle by $\universal \to B_GO(n)$,
where $$\universal = E_GO(n) \times_{O(n)} \bR^n.$$






The $G$-homotopy type of the universal base can be obtained from information
about the fixed-point spaces of total space. We have
\begin{thm}(\cite[Theorem 2.17]{L82})
\label{thm:BGO(n)}
\begin{align*}
(B_GO(n))^{G} & \simeq \coprod_{[\rho] \in
                \mathrm{Rep}(G,O(n))}B\mathrm{Z}_{O(n)}(\rho); \\
    & \simeq \coprod_{[V] \in
                \mathrm{Rep}(G,O(n))}B (O(V)^G).
\end{align*}
 \end{thm}
Here, $O(V)$ is the space of isometric self maps of $V$ with $G$ acting by conjugation.

\begin{exmp}
  Take $H=G=C_2$ and $\Pi=O(2)$. Then
  $$\mathrm{Rep}(C_2,O(2)) = \{\mathrm{id}, \text{ rotation}, \text{ reflection}\}.$$
  For $\rho=\mathrm{id}$ or $\rho=\text{rotation}$, $Z_{\Pi}(\rho)=O(2)$. For
  $\rho=\text{reflection}$, $Z_{\Pi}(\rho) \cong \bZ/2 \times \bZ/2$. So 
\begin{equation*}
(B_{C_2}O(2))^{C_2} \simeq BO(2) \amalg BO(2) \amalg B(\bZ/2 \times \bZ/2).
\end{equation*}
\end{exmp}

  One can make explicit the classifying maps of $V$-trivial bundles as
  follows.  
  A $G$-map $\theta: * \to B_GO(n)$ gives the following data:
  it lands in one of the $G$-fixed components of $B_GO(n)$, indexed by a representation class $[V]$;
  its image is a $G$-fixed point $b \in B_GO(n)$.
  
 \begin{prop}
\label{rmk:classifyingV}
   The pullback of the universal bundle along this map is exactly
   $\theta^{*}\universal \cong V$ as a $G$-vector bundle over $*$.
\end{prop}
   
  The loop space of $B_GO(n)$ at the base point $b$, $\Omega_b B_GO(n)$, is a $G$-space
  with the pointwise $G$-action on the loops.
  Via concatenation of loops, it is an $A_{\infty}$-algebra in $G$-spaces.
  Using the Moore loop space
 \begin{equation*}
\moore_b B_GO(n) = \{(l,\alpha) \in \bR_{\geq 0} \times \mathrm{Map}(\bR_{\geq
    0}, B_GO(n)) | \alpha(0) = b, \ \
 \alpha(t)=b \text{ for }t \geq l\},
\end{equation*}
we may strictify  $\Omega_bB_GO(n)$ to a $G$-monoid.
\begin{defn}
    A $G$-monoid is a monoid in $G$-spaces, that is, an underlying monoid such that the
  multiplication is $G$-equivariant. A morphism of $G$-monoids is an equivalence if it is a weak
  $G$-equivalence.
\end{defn}
\begin{thm}(\cite[Theorem 3.12]{ZouBundle})
  \label{cor:monoidmap}
  Let $b$ be a fixed point in the $V$-indexed component of $(B_GO(n))^G$. 
\begin{enumerate}
\item 
  There is a $G$-homotopy equivalence $\Omega_b B_GO(n) \simeq O(V)$;
\item There is an equivalence of $G$-monoids $\moore_b B_GO(n) \simeq O(V)$.
\end{enumerate}
\end{thm}
The equivalence of $G$-monoids is explicitly given by a zigzag (see \autoref{rem:zigzag}).
 \autoref{cor:monoidmap} is used in \autoref{thm:autoV} to understand
 the automorphism space of a framed disk $V$.

\section{Tangential structures and factorization homology}
\label{chap:FH}

\subsection{Equivariant tangential structures}
\label{sec:tangential}

In this subsection we fix a tangential structure $\theta$ and construct two categories.
The first one is $\mathrm{Vec}^{\theta}_{G,n}$,
the category of $n$-dimensional $\theta$-framed equivariant bundles and $\theta$-framed
bundle maps.
The second one is $\mathrm{Mfld}^{\theta}_{G,n}$, the category of smooth $n$-dimensional
$\theta$-framed $G$-manifolds and $\theta$-framed embeddings. The category $\mathrm{Mfld}^{\theta}_{G,n}$ is a
subcategory of $\mathrm{Vec}^{\theta}_{G,n}$; both  $\mathrm{Mfld}^{\theta}_{G,n}$ and
$\mathrm{Vec}^{\theta}_{G,n}$ are enriched over $\Gtop$.
If we let $\theta$ vary, both constructions define covariant functors from
$\mathcal{TS}$ to categories.

Recall that $\universal$ is the universal $G$-$n$-vector bundle over $B_GO(n)$. Pulling back along
  the tangential structure $\theta: B \to B_G O(n)$ gives a bundle $\theta^{*}\universal$ over
  $B$. This is meant to be the universal $\theta$-framed vector bundle.
For an $n$-dimensional smooth $G$-manifold $M$,
the tangent bundle of $M$ is a $G$-$n$-vector bundle.
It is classified by a $G$-map up to $G$-homotopy:
\begin{equation*}
 \tau: M \to B_G O(n).
\end{equation*}

\begin{defn}
  A $\theta$-framing on a $G$-$n$-vector bundle $E \to M$ is a $G$-$n$-vector bundle map
  $\phi_E: E \to \theta^{*}\universal$.
  A $\theta$-framing on a smooth $G$-manifold $M$ is a $\theta$-framing $\phi_M$ on its tangent
  bundle. We abuse notations and refer to the map 
  on the base spaces as $\phi_M$ as well.
\end{defn}

Note that for a manifold $M$ to be $\theta$-framed, it must be of dimension
$n$. We consider the empty set to be uniquely $\theta$-framed for any $n$ and
any $\theta: B \to B_GO(n)$.

A bundle has a $\theta$-framing if and only if its classifying map $\tau: M \to B_GO(n)$
has a factorization up to $G$-homotopy through $B$; see diagram~\autoref{eq:theta-framing}.
However, a factorization $\tau_B: M \to B$ does not uniquely determine a $\theta$-framing
$\phi_E: E \to \theta^{*}(\universal)$. Indeed, a bundle map $\phi_E: E \to \theta^{*}(\universal)$ is the same
data as a map $\tau_B: M \to B$ on the base plus a homotopy between the two classifying maps from
$M$ to $B_GO(n)$. For a detailed proof, see \autoref{cor:HomVSMapoverB} with
\autoref{defn:mapOverB}.
\begin{equation}
  \label{eq:theta-framing}
  \begin{tikzcd}
    & B \ar[d,"\theta"] \\
  M \ar[r,"\tau"'] \ar[ur,dotted, "\tau_B","h^{\curvearrowright}"'] & B_GO(n)
  \end{tikzcd}
\end{equation}

\begin{exmp}
  As seen in \autoref{rmk:classifyingV}, the tangential structure $\mathrm{fr}_V: * \to
   B_GO(n)$ characterizes $V$-trivializations.
  We call it the $V$-framing tangential structure, and emphasize that is not only a space $B= *$ but
  also a map $\mathrm{fr}_V$.
\end{exmp}


\begin{defn}
  \label{defn:theta-bundle-map}
Given two $\theta$-framed bundles $E_1,E_2$ with framings $\phi_1, \phi_2$, the space of
$\theta$-framed bundle maps between them is
defined as:
\begin{equation}
  \label{eq:theta-hom}
\mathrm{Hom}^{\theta}(E_1, E_2) := \mathrm{hofib}\big( \mathrm{Hom}(E_1,E_2) \overset{\phi_2 \circ -}{
  \longrightarrow } \mathrm{Hom}(E_1, \theta^{*}\universal)\big),
\end{equation}
where $\mathrm{Hom}(E_1, \theta^{*}\universal)$ is based at $\phi_1$.
\end{defn}

  We use the following model for the homotopy fiber in \autoref{eq:theta-hom}:
\begin{align*}
  \mathrm{Hom}^{\theta}(E_1,E_2)
  =\{(f, \alpha, l)| & f \in \mathrm{Hom}(E_1,E_2), \alpha \in \mathrm{Map}({\bR_{\geq 0}}, \mathrm{Hom}(E_1,\theta^{*}\universal)), \\
                         &l \in \mathrm{Map}({\mathrm{Hom}(E_1,E_2)}, {\bR_{\geq0}}) \text{ such that }  \\
    & l \text{ is locally constant}, \\
    & \alpha(0) = \phi_1, \alpha(t) = \phi_2 \circ f \text{ for } t \geq l(f)\}.
\end{align*}
Here, the function $l$ is the length of the Moore paths and locally constant means being constant on path components.
The $\theta$-framed bundle maps have unital and associative composition,
with the unit in $\mathrm{Hom}^{\theta}(E,E)$ given by $(\mathrm{id}_E,
\phi_{\mathrm{const}},0_{\mathrm{const}})$.
Treating the path data $l$ as $1_{\mathrm{const}}$, the composition is defined up to homotopy as: 
\begin{equation*}
  \begin{array}{ccc}
    \mathrm{Hom}^{\theta}(E_2,E_3) \times \mathrm{Hom}^{\theta}(E_1, E_2) &  \to &
                                                                            \mathrm{Hom}^{\theta}(E_1,E_3);\\
    \big((g,\beta),(f,\alpha)\big) & \mapsto & (g \circ f, \lambda),
  \end{array}
\end{equation*}
\begin{equation*}
\text{ where }\lambda(t) = \left\{ 
\begin{array}{l@{\quad \text{ when}\quad}l}
 \alpha(2t),& 0 \le t \le 1/2; \\
 \beta(2t-1) \circ f,& 1/2 < t \le 1.
\end{array}\right.
\end{equation*}

\medskip
  Note that in the definition of $\mathrm{Hom}^{\theta}(E_1,E_2)$,
  everything is taken non-equivariantly. The spaces $\mathrm{Hom}(E_1,E_2)$ and $\mathrm{Hom}(E_1,
  \theta^{*}\universal)$  have $G$-actions by conjugation.
  Since
  $\phi_1$ and $\phi_2$ are $G$-maps, the homotopy fiber
  $\mathrm{Hom}^{\theta}(E_1,E_2)$ inherits the conjugation $G$-action.
So we have built a $G \mathrm{Top}$-enriched category $\mathrm{Vec}^{\theta}_{G,n}$ of $\theta$-framed bundles and $\theta$-framed bundle maps.
  
\begin{defn}
\label{def:embedding}
The space of $\theta$-framed embeddings between two $\theta$-framed manifolds is defined as the
pullback displayed in the following diagram of $G$-spaces:
\begin{equation}
  \label{eq:emb-space1}
  \begin{tikzcd}
    \mathrm{Emb}^{\theta}(M,N) \ar[r] \ar[d] & \mathrm{Hom}^{\theta}(\mathrm{T}M,\mathrm{T}N) \ar[d] \\
    \mathrm{Emb}(M,N) \ar[r, "d"] & \mathrm{Hom}(\mathrm{T}M,\mathrm{T}N)
  \end{tikzcd}
\end{equation}
Here, $\mathrm{Emb}(M,N)$ is the space of smooth embeddings and the map $d$ takes an embedding to its
derivative. For the empty manifold, we define
$\mathrm{Emb}^{\theta}(\varnothing,N) = *$ and
$\mathrm{Emb}^{\theta}(M,\varnothing) = \varnothing$
unless $M = \varnothing$.
The category $\mathrm{Mfld}^{\theta}_{G,n}$ has objects $\theta$-framed
manifolds (including the empty set) and morphism spaces $\mathrm{Emb}^{\theta}$.
\end{defn}

\begin{rem}
  \label{rem:emb-data} Most of the time, we drop the Moore-path-length data and write an
  element of $\mathrm{Emb}^{\theta}(M,N)$ as a package of a map $f$ and a homotopy
  $\bar{f}=(f,\alpha)$, with $f \in \mathrm{Emb}(M,N)$ and $\alpha: [0,1] \to
  \mathrm{Hom}(\mathrm{T}M, \mathrm{T}N)$ satisfying $\alpha(0) = \phi_M$ and
  $\alpha(1)= \phi_N \circ  df$.
  There is a functor $\mathrm{Mfld}^\theta_{G,n} \to \mathrm{Mfld}_{G,n}$ by forgetting the tangential structure.
  It sends $\bar{f} \in \mathrm{Emb}^{\theta}(M, N)$ to $f \in \mathrm{Emb}(M,N)$.
\end{rem}

Let $\amalg$ be the disjoint union of $\theta$-framed vector bundles or manifolds and $\varnothing$
be the empty bundle or manifold. 
Both $(\mathrm{Vec}^{\theta}_{G,n}, \amalg, \varnothing)$ and $(\mathrm{Mfld}^{\theta}_{G,n}, \amalg, \varnothing)$  are
$\Gtop$-enriched symmetric monoidal categories. In both categories, $\varnothing$ is the initial object.
In $ \mathrm{Vec}^{\theta}_{G,n}$, $\amalg$ is the coproduct, but it is not so in $\mathrm{Mfld}^{\theta}_{G,n}$.
\begin{rem}
  We need the length of the Moore path to be locally constant as opposed to constant for the enrichment to work. Namely, the map
\begin{equation*}
\mathrm{Hom}^{\theta}(E_1, E'_1) \times \mathrm{Hom}^{\theta}(E_2, E'_2) \to
\mathrm{Hom}^{\theta}(E_1 \amalg E_2 , E'_1\amalg E'_2)
\end{equation*}
is given by first post-composing with the obvious $\theta$-framed map $E'_i \to E'_1 \amalg E'_2$ for
$i=1,2$, then using a homeomorphism, as follows:
\begin{align*}
  \mathrm{Hom}^{\theta}(E_1, E'_1) \times \mathrm{Hom}^{\theta}(E_2, E'_2)
  & \to \mathrm{Hom}^{\theta}(E_1, E'_1 \amalg E'_2) \times \mathrm{Hom}^{\theta}(E_2, E'_1 \amalg
    E'_2)\\
  & \cong \mathrm{Hom}^{\theta}(E_1 \amalg E_2, E'_1 \amalg E'_2)
\end{align*}
If the length of the Moore path were constant, the displayed homeomorphism would only be a homotopy
equivalence, as the length of a Moore path can be different on the two parts.
\end{rem}

\medskip

To set up factorization homology in \autoref{sec:eFH},
we fix an $n$-dimensional orthogonal $G$-representation $V$;
in addition, we suppose that $V$ is $\theta$-framed and fix a $\theta$-framing
on $V$
$$\phi: \mathrm{T}V \to \theta^{*}\universal.$$
From the proof of the next proposition, we may assume without loss of generality
that the base of $\phi: V \to B$ is the constant map to
$\phi(0) \in B^G$ (which is a $V$-indexed component in the sense of
\autoref{thm:LM}).

\begin{prop}
     \label{prop:22}
 Write $\rho: G \to O(n)$ for a matrix representation of $V$ and
   $\Lambda_{\rho} = \{(\rho(g),g) \in O(n) \times G| g \in G\}$.
  For a tangential structure $\theta: B \to B_GO(n)$, the space
  of $\theta$-framings on the $G$-manifold $V$ is equivalent to $
  (\theta^{*}E_GO(n))^{\Lambda_\rho} \cong \theta^{*}(E_GO(n))^{\Lambda_\rho}$.
  So a $\theta$-framing on $V$ exists, if and only if the intersection of $\theta(B)$ and the $V$-indexed
component of $(B_GO(n))^G$ as introduced in \autoref{thm:BGO(n)} is non-empty.
\end{prop}
\begin{proof}
   Since $\mathrm{T}V \cong V$ as $G$-vector bundles, the space of
   $\theta$-framings on $V$ is
\begin{equation*}
  \mathrm{Hom}(\mathrm{T}V, \theta^{*}\universal)^G
  \simeq \mathrm{Hom}(V,\theta^{*}\universal)^G
  = \mathrm{Hom}(\bR^n,\theta^{*}\universal)^{\Lambda_{\rho}} \cong (\theta^{*}E_GO(n))^{\Lambda_\rho},
\end{equation*}
We have
$$(\theta^{*}E_GO(n))^{\Lambda_\rho} \cong \theta^{*}(E_GO(n))^{\Lambda_\rho}$$
by applying \autoref{thm:LM} to the principal $G$-$O(n)$-bundles
\begin{equation*}
\theta^{*}E_GO(n) \to B \text{ and } E_GO(n) \to B_GO(n). \qedhere
\end{equation*}

\end{proof}
\autoref{prop:22} and \autoref{thm:BGO(n)} give:
\begin{cor}
   Let $V, W$ be $n$-dimensional $G$-representations. 
\begin{enumerate}
\item The $G$-manifold $W$ can be $\mathrm{fr}_V$-framed if and only if $W \cong V$ as $G$-representations.
\item For a tangential structure $\theta$ so that $V$ and $W$ are both
  $\theta$-framed and $H \subset G$, $(\mathrm{Emb}^{\theta}(V,W))^H \neq
  \varnothing$ if and only if $\mathrm{Res}_H^G W \cong \mathrm{Res}_H^G V$ as $H$-representations.
\end{enumerate}
\end{cor}

We also describe the change of tangential structures.
Let $q$ be a morphism from $\theta_1: B_1 \to B_GO(n)$ to $\theta_2: B_2 \to B_GO(n)$, equivalently, a
$G$-map $q: B_1 \to B_2$ satisfying $\theta_2q=\theta_1$. Then a $\theta_1$-framed vector bundle
$E \to B$ with $\phi_E: E \to \theta_1^{*}\universal $ is $\theta_2$-framed by
$$ E \to \theta_1^{*}\universal = q^{*}\theta_2^{*}\universal \to \theta_2^{*}\universal.$$
The morphism $q$ also induces a map on framed-morphisms. So we have a functor
\begin{equation*}
  q_{*}:\mathrm{Vec}^{\theta_1}_{G,n} \to \mathrm{Vec}^{\theta_2}_{G,n}, \text{ and similarly }
  q_{*}: \mathrm{Mfld}^{\theta_1}_{G,n} \to \mathrm{Mfld}^{\theta_2}_{G,n}.
\end{equation*}

\subsection{Equivariant factorization homology}
\label{sec:eFH}
In this subsection, we use the
$\Gtop$-enriched category $\mathrm{Mfld}^{\theta}_{G,n}$ developed in \autoref{sec:tangential}
to define the equivariant factorization homology as a monadic bar construction. We have fixed an $n$-dimensional orthogonal
$G$-representation $V$ and a $\theta$-framing $\phi: \mathrm{T}V \to \theta^{*}\universal$ on the
$G$-manifold $V$.

Recall that $\Lambda$ is the category of finite based sets $\mathbf{k}$ and based injections.
From \autoref{exmp:DMDV}, we have a $\Lambda$-sequence $\mathscr{D}_{\myM}^{\theta}$ for a $\theta$-framed manifold $M$.
Explicitly, on objects, we have
\begin{align}
  \label{defn:Dm}
    \mathscr{D}_{\myM}^{\theta}(k) & = \mathrm{Emb}^{\theta}(\amalg_k V, M);
  \end{align}
  On morphisms, $\Sigma_k $ acts by permuting the copies of $V$,
  and $s_i^k: \mathbf{k-1} \to  \mathbf{k}$ induces
  $(s_i^k)^{*}: \mathscr{D}_{\myM}^{\theta}(k) \to \mathscr{D}_{\myM}^{\theta}(k-1)$
  by forgeting the $i$-th $V$ in the embeddings for $1 \leq i \leq k$.

 Taking $M = V$,  $\mathscr{D}_V^{\theta}$ is a reduced $G$-operad.
Using \autoref{defn:LambdaObjectCircle}, we get associated functors of $ \mathscr{D}^{\theta}_{\myM}$ and $\mathscr{D}^{\theta}_{V}$, which we denote by
\begin{align*}
 \mathrm{D}^{\theta}_{\myM},\mathrm{D}^{\theta}_{V}&:   \Gtop_{*} \to \Gtop_{*}; \\
 \mathrm{D}^{\theta}_{\myM}(X) & =  \coprod_{k \geq 0} \mathscr{D}_{\myM}^{\theta}(k) \times_{\Sigma_k}  X^{\times k}/ \sim.
\end{align*}
The associated functor $
\mathrm{D}^{\theta}_{V}$ is a monad (with natural transformations
${\eta: \mathrm{id} \to \mathrm{D}^{\theta}_{V}}$ and
$m: \mathrm{D}^{\theta}_{V} \circ \mathrm{D}^{\theta}_{V} \to
\mathrm{D}^{\theta}_{V}$) and 
$\mathrm{D}^{\theta}_{\myM}$ is a right $
\mathrm{D}^{\theta}_{V}$-module
(with a natural transformation $m_{L}: \mathrm{D}^{\theta}_{\myM} \circ \mathrm{D}^{\theta}_{V} \to
\mathrm{D}^{\theta}_{\myM}$).
The following is a standard definition:
\begin{defn}
  Let $\mathscr{C}$ be a reduced operad in $(\Gtop,\times)$ and $\mathrm{C}$ be the associated
 monad. An object $A \in \Gtop_{*}$ 
   is a $\mathscr{C}$-algebra, or equivalently a $\mathrm{C}$-algebra, if there is a map $\gamma: \mathrm{C}A \to A$ such that the following diagrams
   commute, where the unlabeled maps are the unit and multiplication map of the monad $\mathrm{C}$:
   \begin{equation*}
     \begin{tikzcd}
     CCA \ar[r,"C\gamma"] \ar[d] & CA \ar[d,"\gamma"] \\
     CA \ar[r,"\gamma"'] & A
   \end{tikzcd}; \quad
   \begin{tikzcd}
     A \ar[r] \ar[rd,equal] & CA \ar[d,"\gamma"]\\
     & A
   \end{tikzcd}.
   \end{equation*}
\end{defn}

In what follows, let $A$ be a $\mathscr{D}_V^{\theta}$-algebra in
$\Gtop_{*}$. It is conceptually a left $\mathrm{D}^{\theta}_{\myM} $-module.
We have a simplicial $G$-space, whose $q$-th level is
\begin{equation*}
  \mathrm{B}_q(\mathrm{D}_{\myM}^{\theta},\mathrm{D}_V^{\theta},A) :=
  \mathrm{D}_{\myM}^{\theta}(\mathrm{D}_V^{\theta})^qA.
\end{equation*}
The face maps are induced by the above-given structure maps
\begin{equation*}
  m_L: \mathrm{D}^{\theta}_M \mathrm{D}^{\theta}_V \to \mathrm{D}^{\theta}_M,
  \ \ m: \mathrm{D}^{\theta}_V \mathrm{D}^{\theta}_V \to \mathrm{D}^{\theta}_V \text{ and }
\gamma: \mathrm{D}^{\theta}_VA \to A.  
\end{equation*}
The degeneracy maps are induced by $\eta: \mathrm{id} \to
\mathrm{D}^{\theta}_V$.

We have the following definition after the non-equivariant
work of \cite[IX.1.5]{Andrade}:
\begin{defn}
\label{defn:factoriaztion}
The factorization homology of $M$ with coefficients $A$ is
\begin{equation*}
\int_M^{\theta}A : = \mathrm{B}(\mathrm{D}_{\myM}^{\theta},\mathrm{D}_V^{\theta} ,A).
\end{equation*}
\end{defn}

\medskip
The category of algebras $\mathscr{D}_V^{\theta}[\Gtop_{*}]$ has a transfer model structure via
the forgetful functor $\mathscr{D}_V^{\theta}[\Gtop_{*}] \to \Gtop_{*}$ (\cite[3.2, 4.1]{BergMoer}),
so that weak equivalences of maps between algebras are just underlying weak equivalences.
\begin{prop}
  \label{prop:homotopical}
  The functor $\displaystyle\int_M^{\theta} -: \mathscr{D}_V^{\theta}[\Gtop_{*}] \to \Gtop_{*}$ is homotopical.
\end{prop}
\begin{proof}
  The proof is a formal argument assembling the literature.
  We show that the bar construction is Reedy cofibrant in the deferred \autoref{lem:Reedy}.
  The claim then follows since geometric realization preserves levelwise weak equivalences between
  Reedy cofibrant simplicial $G$-spaces, as quoted in the deferred \autoref{thm:Reedy}.
\end{proof}

We have the following properties of the factorization homology.
\begin{prop}
\label{prop:FHonV}
\begin{align*}
  \int^{\theta}_VA & \simeq A. \\
\int^{\theta}_M \mathrm{D}^{\theta}_VA & \simeq \mathrm{D}^{\theta}_MA.
\end{align*}
\end{prop}
\begin{proof}
Both follow from the extra degeneracy argument of \cite[Propositions
9.8 and 9.9]{MayGILS}. For the first equivalence, the extra
degeneracy 
coming from the unit map of the first $\mathrm{D}_V^{\theta}$ establishes $A$ as a retract of 
$\mathbf{B}(\mathrm{D}_V^{\theta},\mathrm{D}_V^{\theta},A)$, which is just
$\displaystyle\int_VA.$
For the second equivalence, the unit map $A \to \mathrm{D}_V^{\theta}A$
establishes $\mathrm{D}_M^{\theta}A$ as a retract of $\mathbf{B}(\mathrm{D}_M^{\theta},\mathrm{D}_V^{\theta},\mathrm{D}_V^{\theta}A)$.
\end{proof}

\begin{prop}
  \label{prop:FHonUnion}
   For $\theta$-framed manifolds $M$ and $N$,
\begin{equation*}
\int^{\theta}_{M \amalg N}A \cong \int^{\theta}_MA \times \int^{\theta}_NA.
\end{equation*}
\end{prop}
\begin{proof}
  Without loss of generality, we may assume that both $M$ and $N$ are connected. Then
\begin{align*}
  \mathscr{D}_{M \amalg N}^{\theta}(k) & \cong \mathrm{Emb}^{\theta}(\amalg_k V, M \amalg N) \\
  &\cong \coprod_{i=0}^k\big(\mathrm{Emb}^{\theta}(\amalg_i V, M ) \times \mathrm{Emb}^{\theta}(\amalg_{k-i}
    V, N)\big) \times_{\Sigma_i  \times \Sigma_{k-i}} \Sigma_k \\
  & \cong \coprod_{i=0}^k \big(\mathscr{D}_M^{\theta}(i) \times \mathscr{D}_{N}^{\theta}(k-i)\big) \times_{\Sigma_i
    \times \Sigma_{k-i}} \Sigma_k
\end{align*}

This is the formula of the Day convolution of $\mathscr{D}_{M}^{\theta}$ and
$\mathscr{D}_{N}^{\theta}$. So we have
\begin{equation}
\label{eq:5}
\mathscr{D}_{M \amalg N}^{\theta} \cong \mathscr{D}_{M}^{\theta} \boxtimes
  \mathscr{D}_{N}^{\theta}. 
\end{equation}
  
  We drop the $\theta$ in the rest of the proof. 
By \autoref{eq:5} and iterated use of \autoref{prop:circledistribute}, there is an
isomorphism in $\Lambda^{op}_*(\Gtop)$ for each $q$:
\begin{equation}
  \label{eq:1}
\mathbf{B}_q(\mathscr{D}_{M \amalg N},\mathscr{D}_{V},\imath_0(A)) \cong \mathbf{B}_q(\mathscr{D}_M,
\mathscr{D}_V, \imath_0(A)) \boxtimes \mathbf{B}_q(\mathscr{D}_{N}, \mathscr{D}_V, \imath_0(A)).
\end{equation}

Iterated use of \autoref{eq:monad-and-circle} identifies
\begin{equation*}
\imath_0(\mathbf{B}_{q}(\mathrm{D}_{\myM},\mathrm{D}_{V},A)) \cong \mathbf{B}_q(\mathscr{D}_{\myM}, \mathscr{D}_V, \imath_0(A)),
\end{equation*}
so evaluating on the 0-th level of \autoref{eq:1} gives an equivalence of simplical $G$-spaces:
\begin{equation*}
\mathbf{B}_{*}(\mathrm{D}_{M \amalg N},\mathrm{D}_{V}, A) \cong \mathbf{B}_*(\mathrm{D}_{M}, \mathrm{D}_V, A) \times \mathbf{B}_*(\mathrm{D}_{N}, \mathrm{D}_V, A).
\end{equation*}

 The claim follows from passing to geometric realization and commuting the geometric realization with the product.
\end{proof}

\begin{thm}
  \label{prop:change-of-tangential}
  Let $q: \theta_1 \to \theta_2$ be a morphism of tangential structures
  and $V = (V, \phi_1)$ be $\theta_1$-framed.
  We also write $V$ for the $\theta_2$-framed $G$-manifold $q_{*}V = (V, q\phi_1)$.
  For a $\theta_1$-framed $G$-manifold $M$ and a $\mathrm{D}_V^{\theta_2}$-algebra
  $A$, there is a $G$-equivalence
\begin{equation*}
\int^{\theta_1}_M q^{*}A \simeq \int^{\theta_2}_{q_{*}M} A.
\end{equation*}
\end{thm}
The proof is deferred to the end of \autoref{sec:operad}.

\begin{notn}
  From now on, we consider $\theta$ implicit and write
  $\displaystyle\int_M^{\theta} A $ as $\displaystyle\int_M A$.
\end{notn}

\subsection{Relation to configuration spaces}
\label{sec:embeddingspace}
Now we restrict our attention to the $V$-framed case for an orthogonal $n$-dimensional
$G$-representation $V$.
We give $V$ the canonical $V$-framing $\mathrm{T}V \cong V \times V$
and let $M$ be a $G$-manifold of dimension $n$.
When $M$ is $V$-framed, we denote the $V$-framing by $\phi_M: \mathrm{T}M  \to V$.

In this subsection,
we first prove that a smooth embedding of $\amalg_kV$ into $M$ is determined by its images and
derivatives at the origin up to a contractible choice of homotopy (\autoref{lem:derivative}).
Then we proceed to prove that a $V$-framed
embedding space of $\amalg_kV$ into $M$ as defined in \autoref{eq:emb-space1}
is homotopically the same as choosing the center points (\autoref{cor:conf}).

To formulate the result, we
first define the suitable equivariant configuration space related to a
manifold, which will be ``the space of points and derivatives''.

We use $\conf{E}$ to denote the ordered configuration space of $k$ distinct points in $E$,
topologized as a subspace of $E^k$.
  When $E$ is a $G$-space, $\conf{E}$ has a $G$-action by pointwise acting. It commutes with the
  $\Sigma_k$-action that permutes the points.
\begin{defn}
  For a fiber bundle $p: E \to M$, define $\overconf{E}$ to be configurations of $k$-ordered distinct
  points in $E$ with distinct images in $M$.
  $\overconf{E}$ is a subspace of $\conf{E}$ and inherits a free $\Sigma_k$-action.
  When $p$ is a $G$-fiber bundle, $\overconf{E}$ is a $G$-space.
\end{defn}

\begin{exmp}
  When $k=1$, $\mathscr{F}_{E \downarrow M}(1) \cong \mathscr{F}_E(1)$.
\end{exmp}
\begin{exmp}
  \label{exmp:frame-manifold-PConf}
  When $E = M \times F$ is a trivial bundle over $M$ with fiber $F$,
\begin{equation*}
\overconf{E} \cong \conf{M} \times F^k .
\end{equation*}
\end{exmp}
In general, we have the following pullback diagram:
\begin{equation*}
  \begin{tikzcd}
    \overconf{E} \ar[d] \ar[r,hook] & E^k \ar[d,"p^k"] \\
    \conf{M} \ar[r,hook] & M^k.
  \end{tikzcd}
\end{equation*}

Now, we take $E = \mathrm{Fr}_V(\mathrm{T}M)$.
Recall that $\mathrm{Fr}_V(\mathrm{T}M) = \mathrm{Hom}(V, \mathrm{T}M)$ is a $G$-bundle over $M$.
For an embedding $\amalg_kV \to M$, we take its
derivative and evaluate at $0 \in V$. We will get $k$-points in $\mathrm{Fr}_V(\mathrm{T}M)$ with
different images projecting to $M$. In other words, the composition
\begin{equation*}
\mathrm{Emb}(\amalg_k V, M) \overset{d}{ \to }
 \mathrm{Hom}(\amalg_k \mathrm{T}V, \mathrm{T} M) \overset{ev_0}{ \to }
 \mathrm{Hom}(\amalg_k V, \mathrm{T}M) = \mathrm{Fr}_V(\mathrm{T}M)^{k} 
\end{equation*}
factors as
\begin{equation}
  \label{equ:derivative}
\mathrm{Emb}(\amalg_k V, M) \overset{d_0}{ \to }
 \overconf{\mathrm{Fr}_V(\mathrm{T}M)} \hookrightarrow \mathrm{Fr}_V(\mathrm{T}M)^{k} .
\end{equation}

\begin{prop}
\label{lem:derivative}
  The map $d_0$ in \autoref{equ:derivative} is a $G$-Hurewicz fibration and $(G \times \Sigma_k)$-homotopy equivalence.
\end{prop}
  One can find an equivariant local trivialization. The proof is tedious and can
  be found in \cite[Prop 5.5.5]{mythesis}.

  A section and homotopy inverse exists uniquely up to homotopy:
\begin{equation}
\label{equ:sectionexp}
\sigma:  \overconf{\mathrm{Fr}_V(\mathrm{T}M)} \to \mathrm{Emb}(\amalg_k V, M).
\end{equation}
For $k=1$,  it is given by the exponential map:
\begin{equation*}
\sigma: \mathrm{Fr}_V(\mathrm{T}M) \to \mathrm{Emb}(V,M).
\end{equation*}
Since there is a (contractible) choice of the radius at each point for the exponential map to be
homeomorphism, $\sigma$ is unique only up to homotopy.

\begin{lem}  \label{item:corconf0}
For a $V$-framed manifold $M$, the projection 
\begin{equation*}
  \overconf{\mathrm{Fr}_V(\mathrm{T}M)} \to  \conf{M}
\end{equation*}
  is a trivial bundle with fiber $(\mathrm{Hom}(V,V))^k$.
  We call the section that selects $(\mathrm{id}_{V})^k$ in each fiber the zero section $z$.  
\end{lem}
\begin{proof}
 Regarding $V$ as a bundle over a point, we may identify $\mathrm{Fr}_V(V) = \mathrm{Hom}(V,V)$.
 Since $M$ is $V$-framed, $\mathrm{Fr}_V(\mathrm{T}M) \cong \mathrm{Fr}_V(M \times V) \cong M \times
 \mathrm{Fr}_V(V)$ as equivariant bundles. The claim follows from \autoref{exmp:frame-manifold-PConf}.  
\end{proof}

We can restrict the exponential map \autoref{equ:sectionexp} to the zero section in
\autoref{item:corconf0} to get
\begin{equation}
\label{equ:sectionexp-at0}
\sigma_0:  \conf{M} \to \mathrm{Emb}(\amalg_k V, M).  
\end{equation}
Now we are ready to justify the equivalence of
$\mathrm{Emb}^{\mathrm{fr}_V}(\amalg_{k}V, M)$
and the configuration spaces of $M$.
Moreover, we show that this equivalence is compatible over
$\mathrm{Emb}(\amalg_k V, M)$. This will be used in later sections
to compare different scanning maps.

\begin{prop}
\label{cor:conf}
For a $V$-framed manifold $M$, we have:
\begin{enumerate}
\item \label{item:corconf1}
  Evaluating at $0$ of the embedding gives a $(G \times \Sigma_k)$-homotopy equivalence:
\begin{equation*}
 ev_0:\mathscr{D}^{\mathrm{fr}_{V}}_{\myM}(k) \equiv  \mathrm{Emb}^{\mathrm{fr}_V}(\amalg_k
  V, M) \to \conf{M}.
\end{equation*}
\item \label{item:corconf2}
The forgetful map $\mathrm{Emb}^{\mathrm{fr}_V}(\amalg_k
  V, M) \to \mathrm{Emb}(\amalg_k V, M) $ is homotopic to \autoref{equ:sectionexp-at0} in the
  sense that the following diagram is $(G
\times \Sigma_k)$-homotopy commutative: 
\begin{equation*}
  \begin{tikzcd}
   \mathrm{Emb}^{\mathrm{fr}_V}(\amalg_k
  V, M) \ar[r] \ar[d,"ev_0"'] &  \mathrm{Emb}(\amalg_k V, M) \ar[d,"ev_0"] \\
  \conf{M} \ar[ur,"\sigma_0"'] \ar[r,equal]&  \conf{M}
  \end{tikzcd}
\end{equation*}
\end{enumerate}
\end{prop}

\begin{proof} 
\autoref{item:corconf1}
  By \autoref{def:embedding} and \autoref{defn:Dm}, $\mathrm{Emb}^{\mathrm{fr}_V}(\amalg_k V, M)$ is the homotopy
  fiber of the composite:
\begin{equation*}
D: \mathrm{Emb}(\amalg_k V, M) \overset{d}{ \to }  \mathrm{Hom}(\amalg_k
\mathrm{T}V, \mathrm{T} M)  \overset{(\phi_M)_{*}}{ \to } \mathrm{Hom}(\amalg_k
\mathrm{T}V, V).
\end{equation*}

We would like to restrict the composite at
$\{0\} \amalg \cdots \amalg \{0\} \subset V\amalg \cdots \amalg V$. Since 
\begin{equation*}
\mathrm{Hom}(\amalg_k \mathrm{T}V, \mathrm{T} M)  \cong \prod_k  \mathrm{Hom}(\mathrm{T}V,
\mathrm{T} M) 
\end{equation*}
and $i_0: V \to \mathrm{T}V$ is a $G$-homotopy equivalence of $G$-vector bundles,
\begin{equation*}
ev_0: \mathrm{Hom}(\amalg_k \mathrm{T}V, \mathrm{T} M) \overset{(i_0)^{*}}{ \to }  \prod_k
\mathrm{Hom}(V,\mathrm{T} M) \cong (\mathrm{Fr}_V(\mathrm{T}M))^k
\end{equation*}
is a $(G \times \Sigma_k)$-homotopy equivalence.
So in the following commutative diagram,
the vertical maps are all $(G \times \Sigma_k)$-homotopy equivalences:
\begin{equation}
  \label{eq:big}
    \begin{tikzcd}      
 \mathrm{Emb}(\amalg_k V, M)\ar[r,"d"]  \ar[d,"d_0"',"{\simeq \text{ by \autoref{lem:derivative}}}"] &
 \mathrm{Hom}(\amalg_k \mathrm{T}V, \mathrm{T} M) \ar[d,"ev_0"',"\simeq"]  \ar[r,"(\phi_M)_{*}"]
 & \mathrm{Hom}(\amalg_k \mathrm{T}V, V) \ar[d,"ev_0"',"\simeq"]  \\
 \overconf{\mathrm{Fr}_V(\mathrm{T}M)} \ar[r,hook] \ar[d,"{\cong \text{ by \autoref{item:corconf0}}}"]
 & \mathrm{Fr}_V(\mathrm{T}M)^{k}  \ar[r,"(\phi_M)_{*}"] & \mathrm{Fr}_V(V)^{k} \ar[d,equal]\\
   \conf{M} \times \mathrm{Fr}_V(V)^{k} \ar[rr,"proj_2"] &  & \mathrm{Fr}_V(V)^{k}.
    \end{tikzcd}
  \end{equation}

We focus on the top composition $D$ and the bottom map $proj_2$.
The map $ev_0$ between their codomains is a based map.
Indeed, the base point of ${\mathrm{Hom}(\amalg_k \mathrm{T}V, V)}$ is
from the $V$-framing of $\amalg_kV$ and is $(G \times \Sigma_k)$-fixed.
It is mapped to $\mathrm{id}^k$, the base point of $\mathrm{Fr}_V(V)^{k}$.
Consequently, there is a $(G \times \Sigma_k)$-homotopy equivalence between the homotopy fibers of
these two maps.
\begin{equation}
  \label{eq:14}
  \mathrm{Emb}^{\mathrm{fr}_V}(\amalg_k V, M) = \mathrm{hofib}(D) \overset{\simeq}{ \to }
   \mathrm{hofib}(proj_2).
\end{equation}

Our desired $ev_0$ in question is the composite of \autoref{eq:14} and the following map:
\begin{equation*}
  X: \mathrm{hofib}(proj_2) \to \conf{M} \times \mathrm{Fr}_V(V)^{k} \overset{proj_1}{\to} \conf{M}.
\end{equation*}
It suffices to show that $X$ is a $(G \times \Sigma_k)$-equivalence.
Indeed, $X$ is the comparison of the homotopy fiber and the actual fiber of $proj_2$.
Write temporarily  $F = \conf{M}$ and $B =  \mathrm{Fr}_V(V)^{k}$ with the $(G
\times \Sigma_k)$-fixed base point $b$.
Then the map $X$ is projection to $F$:
\begin{equation*}
\mathrm{hofib}(proj_2) \cong P_bB \times F \to F.
\end{equation*}
  The claim follows from the fact that $P_bB$ is $(G \times \Sigma_k)$-contractible.

\autoref{item:corconf2}
We examine the following diagram, where $z$ is the zero section in \autoref{item:corconf0}:
  \begin{equation*}
  \begin{tikzcd}
   \mathrm{Emb}^{\mathrm{fr}_V}(\amalg_k
  V, M) \ar[r] \ar[d,"ev_0"'] &  \mathrm{Emb}(\amalg_k V, M) \ar[d,"d_0"']\\
  \conf{M} \ar[r, "z"] \ar[ur,"\sigma_0", dotted]& 
  \overconf{\mathrm{Fr}_V(\mathrm{T}M)} \ar[u,shift right,"\sigma"',dotted].
  \end{tikzcd}
\end{equation*}
The left column is given by the (homotopy) fibers of the
first and second rows of \autoref{eq:big}, so the solid  diagram is $(G \times \Sigma_k)$-homotopy
commutative.
As  $\sigma_0 = \sigma \circ z$ and  $\sigma$ is a $(G \times \Sigma_k)$-homotopy inverse
of $d_0$ by \autoref{lem:derivative}, the upper triangle with the dotted arrow is homotopy commutative.
\end{proof}

\subsection{Comparison of operads and the invariance theorem}
\label{sec:operad}
In this subsection, we study the $\theta$-framed little $V$-disk operad
$\mathscr{D}^{\theta}_V$.

For $\theta = \mathrm{fr}_V$, 
 $\mathscr{D}_V^{\mathrm{fr}_V}$ is equivalent to the little $V$-disks operad $\mathscr{D}_V$.
 For background,
$\mathscr{D}_V$ is a well-studied notion introduced for recognizing
$V$-fold loop spaces; see \cite[1.1]{GMPermG}.
Roughly speaking, $\mathscr{D}_V (k)$ is the space of non-equivariant embeddings of $k$ copies of the open unit
disks $\mathrm{D}(V)$ to $\mathrm{D}(V)$, each of which takes only the form $\mathbf{v} \mapsto a\mathbf{v}+\mathbf{b}$
for some $0<a \leq 1$ and $\mathbf{b} \in \mathrm{D}(V)$, called linear. In
particular, the spaces are the same as those of the non-equivariant little $n$-disks operad, and so are the structure
maps. The $G$-action on $\mathscr{D}_{V}(k)$ is by conjugation. It is well-defined, commutes with
the $\Sigma_k$-action and the structure maps are $G$-equivariant.

\begin{prop}
  \label{cor:compareDV}
  There is an equivalence of $G$-operads $\beta: \mathscr{D}_V \to \mathscr{D}_V^{\mathrm{fr}_V}$.
\end{prop}
\begin{proof}
  To construct the map of operads $\beta$, we first define $\beta(1): \mathscr{D}_V(1) \to
  \mathscr{D}_V^{\mathrm{fr}_V}(1)$. Take $e \in \mathscr{D}_V(1)$; we must give 
  $\beta(1)(e) = (f , l, \alpha) \in  \mathscr{D}_V^{\mathrm{fr}_V}(1)$.
  Explicitly, 
\begin{equation*}
e: \mathrm{D}(V) \to \mathrm{D}(V) \text{ is } e( \mathbf{v}) =  a \mathbf{v} + \mathbf{b} \text{
  for some } 0< a \leq 1 \text{ and } \mathbf{b} \in \mathrm{D}(V).
\end{equation*}
Define
\begin{equation*}
  \begin{array}[h]{rcl}
    f: V \to V & \text{ to be } & f( \mathbf{v}) =  a \mathbf{v} + \mathbf{b}; \\
    l \in \bR_{\geq 0} & \text{ to be } & l = -\ln(a); \\
    \alpha: \bR_{\geq 0} \to \mathrm{Hom}(\mathrm{T}V, V) & \text{ to be } & \alpha(t) =
  \begin{cases}
  \mathfrak{c}_{\exp(-t)\mathrm{I}} & \text{ for } t \leq l; \\
  \mathfrak{c}_{a \mathrm{I}} & \text{ for } t > l.
\end{cases}
  \end{array}
\end{equation*}
 For $\alpha$, $\mathrm{Hom}(\mathrm{T}V, V) \cong \mathrm{Map}(V, O(V))$, $\mathrm{I}$ is the unit element
 of $O(V)$ and $\mathfrak{c}$ is the constant map to the indicated element.
 It can be checked that $\beta(1)$ as defined is a map of $G$-monoids.
 
 Restricting $\beta(1)^k: \mathscr{D}_V(1)^k \to
 \mathscr{D}_V^{\mathrm{fr}_V}(1)^k$ to the subspace $\mathscr{D}_V(k) \subset \mathscr{D}_V(1)^k$, we get
 $\beta(k): \mathscr{D}_V(k) \to \mathscr{D}_V^{\mathrm{fr}_V}(k)$. Then $\beta$ is automatically a map
 of $G$-operads because $\mathscr{D}_V$ and $\mathscr{D}_V^{\mathrm{fr}_V}$ are suboperads of
 $\mathscr{D}_V(1)^*$ and $(\mathscr{D}_V^{\mathrm{fr}_V}(1))^*$.

 The composite $\mathrm{ev}_0 \circ \beta: \mathscr{D}_V \to \mathscr{D}_V^{\mathrm{fr}_V} \to
 \mathscr{F}_V$ is a levelwise homotopy equivalence by \cite[Lemma 1.2]{GMPermG}. We have shown
 that  $\mathrm{ev}_0$ is a levelwise equivalence (\autoref{cor:conf}~\autoref{item:corconf1}). So $\beta$ is
 also a levelwise homotopy equivalence.
\end{proof}

For general $\theta$, $\mathscr{D}_V^{\theta}$ also allows
$\theta$-framed automorphisms of the embedded $V$-disks.
By \autoref{thm:autoV}, the $\theta$-framed automorphism space of $V$ is
equivalent to $\moore_{\phi} B$, the Moore loop space of $B$ based at $\phi(0)$.
\begin{prop}(\cite[B.2.8]{mythesis})
  \label{prop:operad-semi}
  There is a $G$-monoid
$\widetilde{\Lambda}B$ equivalent to
$\moore_{\phi} B$ which  acts on  $\mathscr{D}_V$. Furthermore, there is an equivalence of $G$-operads $\mathscr{D}_V \rtimes
\widetilde{\Lambda}B \to \mathscr{D}_V^{\theta}$.
\end{prop}
\begin{proof}[Explanation]
  
Without loss of generality we assume $V$ is
$\theta$-framed by a constant map.
Recall $\mathscr{D}_M^{\theta}(k) = \mathrm{Emb}^{\theta} (\amalg_k V , M )$.
Note that $\mathrm{fr}_V$ is initial for such tangantial structures, so we have
$$\mathscr{D}_M^{\mathrm{fr}_V}(k) \to \mathscr{D}_M^{\theta}(k) .$$
Let $\mathrm{Emb}_0^{\theta}(V,V) \subset
\mathrm{Emb}^{\theta}(V,V)$ be the sub-$G$-monoid of embeddings that preserves
the origin $0 \in V$. We claim that the composition map
\begin{equation}
\label{eq:20} \mathscr{D}_M^{\mathrm{fr}_V}(k)  \times (\mathrm{Emb}_0^{\theta}(V,V))^k
\to \mathscr{D}_M^{\theta}(k)
\end{equation}
is a $G \times \Sigma_k$-equivalence.
In fact, the composite  
\begin{equation*}
  \begin{tikzcd}
    \mathrm{Emb}^{\mathrm{fr}_V} (\amalg_k V , M ) \ar[r] 
    &    \mathrm{Emb}^{\theta} (\amalg_k V , M ) \ar[r, "ev_0"]
    & \conf{M}
  \end{tikzcd}
\end{equation*}
is an equivalence by \autoref{cor:conf}, where the map $ev_0$ is evaluation at
$0$ and is a $G \times \Sigma_k$ fibration. Its fiber is
$(\mathrm{Emb}_0^{\theta}(V,V))^k$. So it follows that \autoref{eq:20} is an equivalence.

Combining \autoref{cor:compareDV} with \autoref{eq:20}, there
is a $G \times \Sigma_k$-equivalence
\begin{equation*}
\mathscr{D}_V(k)  \times (\moore_{\phi} B)^k \simeq \mathscr{D}_V^{\mathrm{fr}_V}(k)  \times (\mathrm{Emb}_0^{\theta}(V,V))^k
\to \mathscr{D}_V^{\theta}(k)
\end{equation*}
for each $k$. In \cite[Appendix B]{mythesis}, this equivalence is upgraded to an equivalence of
$G$-operads $\mathscr{D}_V \rtimes
\widetilde{\Lambda}B \to \mathscr{D}_V^{\theta}$. Here, $\widetilde{\Lambda}B$
is a replacement of $\moore_{\phi} B$ that acts on $\mathscr{D}_V(k)$,
$\mathscr{D}_V \rtimes \widetilde{\Lambda}B$  is a $G$-operad whose $k$-th space
is $\mathscr{D}_V \times (\widetilde{\Lambda}B)^k$, and the semi-direct product
notation is introduced in  \cite{SW01} to
indicate a twisting in the structure maps.
\end{proof}


\begin{proof}[Proof of \autoref{prop:change-of-tangential}]
Without loss of generality we assume $V$ is
$\theta_{1}$-framed by a constant map.
We omit the $q_{*}$ and $q^{*}$ in the proof.
As $\mathrm{B}(\mathrm{D}_V^{\theta_2}, \mathrm{D}_V^{\theta_2}, A) \simeq A$ as
$\mathrm{D}_V^{\theta_2}$-algebra, we have
\begin{align*}
  \int^{\theta_2}_{q_{*}M} A &= \mathrm{B}(\mathrm{D}_M^{\theta_1},
  \mathrm{D}_V^{\theta_1}, A)\\
  &\simeq
    \mathrm{B}(\mathrm{D}_M^{\theta_1}, \mathrm{D}_V^{\theta_1},
    \mathrm{B}(\mathrm{D}_V^{\theta_2}, \mathrm{}D_V^{\theta_2},A)) \\
  &\simeq
  \mathrm{B}(\mathrm{B}(\mathrm{D}_M^{\theta_1}, \mathrm{D}_V^{\theta_1}, \mathrm{D}_V^{\theta_2}), \mathrm{D}_V^{\theta_2},A).
\end{align*}
It suffices to show that natural map of right $\mathrm{D}_V^{\theta_2}$-functors 
\begin{equation}
  \label{eq:22}
\epsilon: \mathrm{B}(\mathrm{D}_M^{\theta_1}, \mathrm{D}_V^{\theta_1},
\mathrm{D}_V^{\theta_2}) \to \mathrm{D}_M^{\theta_2}
\end{equation}
is an equivalence.

Using  \autoref{eq:20}, one can already construct a retract of \autoref{eq:22}.
To construct a deformation retract, we need the full strength of
\autoref{prop:operad-semi}. There are equivalences of $G$-operads
fitting in a commutative diagram
\begin{equation}
  \label{eq:32}
  \begin{tikzcd}
    \mathscr{D}_V \rtimes \widetilde{\Lambda}B_1 \ar[d] \ar[r,"\sim"]
    & \mathscr{D}_V^{\theta_1} \ar[d]\\
    \mathscr{D}_V \rtimes \widetilde{\Lambda}B_2 \ar[r,"\sim"]& \mathscr{D}_V^{\theta_2}
  \end{tikzcd}
\end{equation}

The monad associated to $\mathscr{D}_V \rtimes
\widetilde{\Lambda}B_i $ for $i=1,2$ is 
\begin{equation*}
\overline{\mathrm{D}}_V^{\theta_i}(A) = \mathrm{D}_V(\widetilde{\Lambda}B_i \times A).
\end{equation*}
And similarly the associated functors for $k \mapsto \mathscr{D}_M(k) \times
(\widetilde{\Lambda}B_i)^k $ are given by
\begin{equation*}
\overline{\mathrm{D}}_M^{\theta_i}(A) = \mathrm{D}_M(\widetilde{\Lambda}B_i \times A).
\end{equation*}
Note that $\widetilde{\Lambda}B_i$ is a $G$-monoid, 
so the functor $A \mapsto \widetilde{\Lambda}B_i \times A$ is a monad,
which we still write as $\widetilde{\Lambda}B_i$. We have
\begin{align*}
\overline{\epsilon}: \mathrm{B}(\overline{\mathrm{D}}_M^{\theta_1}, \overline{\mathrm{D}}_V^{\theta_1},
  \overline{\mathrm{D}}_V^{\theta_2})
  & =  \mathrm{B}(\mathrm{D}_M\circ \widetilde{\Lambda}B_1, \mathrm{D}_V\circ \widetilde{\Lambda}B_1,
    \mathrm{D}_V\circ \widetilde{\Lambda}B_2) \\
  &\cong
 \mathrm{B}(\mathrm{D}_M\circ \widetilde{\Lambda}B_1 \circ
 \mathrm{D}_V,\widetilde{\Lambda}B_1 \circ \mathrm{D}_V, \widetilde{\Lambda}B_2)
  \\
  & \simeq \mathrm{D}_M\circ  \widetilde{\Lambda}B_2 =
    \overline{\mathrm{D}}_M^{\theta_2}
\end{align*}
is an equivalence. Here, the last equivalence is given by a deformation retract using an extra degeneracy argument
\cite[Proposition 9.9]{MayGILS}. Now, in the following commutative
diagram whose vertical maps are equivalences induced by the approximation \autoref{eq:32},
\begin{equation*}
  \begin{tikzcd}
   \mathrm{B}(\overline{\mathrm{D}}_M^{\theta_1}, \overline{\mathrm{D}}_V^{\theta_1},
    \overline{\mathrm{D}}_V^{\theta_2})  \ar[r, "\overline{\epsilon}"] \ar[d]
    &  \overline{\mathrm{D}}_M^{\theta_2} \ar[d]\\
  \mathrm{B}({\mathrm{D}}_M^{\theta_1}, \mathrm{D}_V^{\theta_1},
\mathrm{D}_V^{\theta_2}) \ar[r, "\epsilon"]& \mathrm{D}_M^{\theta_2}
  \end{tikzcd}
\end{equation*}
we see that $\epsilon$ is an equivalence.

\end{proof}

\section{Nonabelian Poincar\'e Duality for $V$-framed manifolds}
\label{chap:NPD}
Configuration spaces have scanning maps out of them.
It turns out that equivariantly the scanning map is an equivalence in the case of $G$-connected labels $X$.
Since the factorization homology is built up simplicially by the configuration spaces,
we can upgrade the scanning equivalence to what is known as the nonabelian Poincar\'e duality theorem.

\subsection{Scanning map for $V$-framed manifolds}
\label{sec:scanning}

In this subsection we construct the scanning map, a natural transformation of
 right $\mathrm{D}^{\mathrm{fr}_{V}}_{V}$-functors:
\begin{equation}
\label{map:scanning}   
s: \mathrm{D}^{\mathrm{fr}_V}_{\myM}(-) \to \mathrm{Map}_{c}(M,\Sigma^V-).
\end{equation}
Here, $\mathrm{Map}_{c}(X,Y)$ for a based space $Y$ denotes the space of maps
$f$ so that the support $\overline{f^{-1}(Y \setminus *)}$ is compact.
In \autoref{chap:appendix-scanning}, we compare our scanning map to the existing different constructions in the
literature. This allows us to utilize known results about equivariant scanning maps to give \autoref{thm:scanning-equi},
a key input to the nonabelian Poincar\'e duality theorem in \autoref{sec:NPD}.

Assume that the scanning map  \autoref{map:scanning} has been constructed for a moment.
When we take $M=V$, \autoref{map:scanning} gives a map of monads
$s: \mathrm{D}^{\mathrm{fr}_{V}}_V \to \Omega^V\Sigma^{V}$.
The adjoint natural transformation
\begin{equation*}
\Sigma^V\mathrm{D}^{\mathrm{fr}_V}_V
\overset{\Sigma^Vs}{\longrightarrow } \Sigma^V\Omega^V\Sigma^V \overset{\text{ counit }}{
  \longrightarrow } \Sigma^V
\end{equation*}
induces the right $\mathrm{D}^{\mathrm{fr}_V}_V$-module structure for
the functor $\mathrm{Map}_c(M,\Sigma^V-)$.

Now we construct the scanning map.
For any $G$-space $X$, recall that 
\begin{equation*}
\mathrm{D}^{\mathrm{fr}_V}_{\myM} (X) = \coprod_{k \geq 0} \mathscr{D}_{\myM}^{\mathrm{fr}_{V}}(k) \times_{\Sigma_k} X^{k}/\sim,
\end{equation*}
where $\sim$ is the base point identification. Take an element
$$P = [\bar{f}_1, \cdots, \bar{f}_k, x_1, \cdots, x_k] \in
\mathscr{D}_{\myM}^{\mathrm{fr}_{V}}(k) \times_{\Sigma_k} X^{k}.$$
Here, each $\bar{f}_i =
(f_i, \alpha_{i})$ consists of an embedding $f_i: V \to M$ and a homotopy $\alpha_i $ of two bundle maps
$\mathrm{T}V \to V$, see \autoref{def:embedding}. We use only the embeddings $f_i$ to define an
element $s_{X}(P) \in \mathrm{Map}_c(M,\Sigma^VX)$:
\begin{equation}
\label{eq:defnsX}
s_{X}(P)(m) = 
  \begin{cases}
     f_{i}^{-1}(m) \sm x_{i} & \text{ when $m \in M$ is in the image of some $f_i$;} \\
     * & \text{ otherwise. }
  \end{cases}
\end{equation}
Notice that if $x_i$ is the base point, $f_i^{-1}(m) \sm x_i$ is the base point regardless of what
$f_i$ is. So passing to the quotient,  \autoref{eq:defnsX} yields a well-defined map
\begin{equation}
  \label{eq:sX}
  s_{X}: \mathrm{D}_{\myM}^{\mathrm{fr}_V}(X) \to \mathrm{Map}_{c}(M, \Sigma^VX).
\end{equation}
In particular, taking $X=S^0$, we get 
\begin{equation}
  \label{eq:9}
s_{S^0}: \coprod_{k \geq 0} \mathscr{D}_{\myM}^{\mathrm{fr}_{V}}(k)/\Sigma_k \to \mathrm{Map}_{c}(M, S^V),
\end{equation}
and $s_X$ is simply a labeled version of it.
A more categorical construction of  the scanning map $s_X$, as  the
composition of the  Pontryagin-Thom collapse map and a
``folding'' map $\vee_kS^V \times X^k \to \Sigma^VX$ is given in \cite[Section 9]{MZZ}.

We use the following results of Rourke--Sanderson \cite{RS00}, which are proved using equivariant
transversality. To translate from their context to ours, see
Theorems \ref{cor:scanning-equi} and \ref{thm:RS}.
\begin{thm}
  \label{thm:scanning-equi}
  The scanning map $s_{X}: \mathrm{D}^{\mathrm{fr}_V}_{\myM} X \to \mathrm{Map}_c(M, \Sigma^V X)$ is:
\begin{enumerate}
\item \label{item:scanning1} a weak $G$-equivalence if  $X$ is $G$-connected,
\item \label{item:scanning2} or a weak group completion if  $V \cong W \oplus \bR$ and $M \cong N \times
  \mathbb{R}$. Here, $W$ is a $(n-1)$-dimension $G$-representation and $N$ is a $W$-framed compact
  manifold, so that $N \times \bR$ is $V$-framed. 
\end{enumerate}
\end{thm}

\subsection{Equivariant nonabelian Poincar\'e duality (eNPD) theorem }
\label{sec:NPD}
We have seen that the scanning map is an equivalence for $G$-connected labels $X$.
Since the factorization homology is built up simplicially by the configuration spaces,
we can upgrade the scanning equivalence to the eNPD
theorem. The proof in this subsection is motivated by the non-equivariant treatment  \cite{Miller}.

  Let $A$ be a $\mathrm{D}^{\mathrm{fr}_V}_V$-algebra in $\Gtop$ throughout this
  subsection. Assume that $A$ is non-degenerately based, meaning that the structure map
  $\mathscr{D}^{\mathrm{fr}_V}_V(0) = *  \to A$ gives a non-degenerate base point of $A$.
  This is a mild assumption for homotopical purposes.
  We use the following $V$-fold delooping model of $A$.
\begin{defn}
  \label{defn:BV}
  The $V$-fold delooping of $A$, denoted as $\mathrm{B}^VA$,
  is the monadic two sided bar construction $\mathrm{B}(\Sigma^V,
  \mathrm{D}^{\mathrm{fr}_{V}}_V, A)$.  \footnote{
  A $ \mathrm{D}^{\mathrm{fr}_{V}}_V$-algebra $A$ has a $\mathrm{D}_V$-algebra
  structure by pulling back along 
 the equivalence of $G$-operads $\mathscr{D}_V \to
 \mathscr{D}^{\mathrm{fr}_V}_V $ (\autoref{cor:compareDV}), and there is an
  equivalence from the delooping $\mathrm{B}(\Sigma^V, \mathrm{D}_V, A)$ in  \cite{GMPermG}  to our delooping
$ \mathrm{B}(\Sigma^V, \mathrm{D}^{\mathrm{fr}_{V}}_V, A)$.}

\noindent Here, $\mathrm{B}_q(\Sigma^V, \mathrm{D}^{\mathrm{fr}_{V}}_V, A) =
\Sigma^V(\mathrm{D}^{\mathrm{fr}_{V}}_V)^qA$. The first face map
$\Sigma^V\mathrm{D}^{\mathrm{fr}_{V}}_V \to \Sigma^V$is induced by the scanning map of monads
$\mathrm{D}^{\mathrm{fr}_V}_V \to \Omega^V\Sigma^V$. The last face map
$\mathrm{D}^{\mathrm{fr}_{V}}_V A \to A$ is the structure map of the
algebra. The middle face maps and degeneracy maps are induced by the structure map of the monad
$\mathrm{D}^{\mathrm{fr}_{V}}_V \mathrm{D}^{\mathrm{fr}_{V}}_V \to \mathrm{D}^{\mathrm{fr}_{V}}_V$
and by its unit map $\mathrm{Id} \to \mathrm{D}^{\mathrm{fr}_{V}}_V$.
\end{defn}

\begin{thm}(eNPD)
\label{thm:NPDV}
  Let $M$ be a $V$-framed manifold and $A$ be a $\mathrm{D}^{\mathrm{fr}_V}_V$-algebra in $\Gtop$.
  Then there is a $G$-map, which is a weak $G$-equivalence if $A$ is $G$-connected:
\begin{equation*}
p_M: \int_M A = |B_{\bullet}(\mathrm{D}^{\mathrm{fr}_V}_{\myM} , \mathrm{D}^{\mathrm{fr}_V}_V, A)| \to \mathrm{Map}_*(M^+, \mathrm{B}^VA).
\end{equation*}
Here, $M^+$ is the one-point-compactification of $M$.
\end{thm}
\begin{proof}
  We give the proof assuming some lemmas that are proven in the remainder of this section.
  First, from \autoref{map:scanning},  we have a scanning map for each $q \geq 0$:
\begin{equation*}
\mathrm{D}^{\mathrm{fr}_V}_{\myM} (\mathrm{D}^{\mathrm{fr}_V}_V)^q A \to   \mathrm{Map}_c (M, \Sigma^V   (\mathrm{D}^{\mathrm{fr}_{V}}_V) ^{q} A).
\end{equation*}
They assemble to a simplicial scanning map, which is a levelwise weak $G$-equivalence as shown
in \autoref{cor:simplicialScanning}:
  \begin{equation}
\label{eq:18}
    \mathrm{B}(s, \mathrm{id}, \mathrm{id}): 
\mathrm{B}_{\bullet}(\mathrm{D}^{\mathrm{fr}_V}_{\myM} , \mathrm{D}^{\mathrm{fr}_V}_V, A) \to
  \mathrm{Map}_c (M, \Sigma^V   (\mathrm{D}^{\mathrm{fr}_{V}}_V) ^{\bullet} A).
\end{equation}
One can identify the space of compactly supported maps with the space of
based maps out of the one point compactification: 
\begin{equation*}
\mathrm{Map}_c (M, \Sigma^V   (\mathrm{D}^{\mathrm{fr}_{V}}_V) ^{\bullet} A)  \overset{\sim}{ \to } \mathrm{Map}_* (M^+, \Sigma^V   (\mathrm{D}^{\mathrm{fr}_{V}}_V) ^{\bullet} A).
\end{equation*}
With some cofibrancy argument in \autoref{thm:Reedy} and \autoref{lem:Reedy}, this map induces
  a weak $G$-equivalence on the geometric realization: 
\begin{equation*}
  \mathrm{B}(\mathrm{D}^{\mathrm{fr}_V}_{\myM} , \mathrm{D}^{\mathrm{fr}_V}_V, A)
  \to |\mathrm{Map}_* (M^+, \Sigma^V   (\mathrm{D}^{\mathrm{fr}_{V}}_V) ^{\bullet} A)|.
\end{equation*}

Next, we change the order of the mapping space and the geometric realization.
There is a natural map \autoref{eq:MapToSset}:
\begin{equation*}
\zeta: | \mathrm{Map}_{*} (M^{+}, \Sigma^V   (\mathrm{D}^{\mathrm{fr}_{V}}_V) ^{\bullet} A)| \to  \mathrm{Map}_{*}
 (M^{+}, |\Sigma^V   (\mathrm{D}^{\mathrm{fr}_{V}}_V) ^{\bullet} A|).
\end{equation*}
Taking $X = M^+$ and $K_{\bullet} = \Sigma^V (\mathrm{D}^{\mathrm{fr}_{V}}_V)
^{\bullet} A$, 
\autoref{thm:HM} gives a sufficient connectivity condition for it to be a weak $G$-equivalence.
This connectivity condition is then checked in \autoref{lem:NPDV}.

  Finally, $|\Sigma^V   (\mathrm{D}^{\mathrm{fr}_{V}}_V) ^{\bullet}
  A| = \mathrm{B}^VA$ by \autoref{defn:BV}. This finishes the proof of the theorem.
\end{proof}

When $A$ is not $G$-connected but $M \cong N \times \mathbb{R}$ or $M \cong N
\times \mathbb{R}^2$, there is also a group completion
version of \autoref{thm:NPDV} in \autoref{thm:grp}.
\begin{rem}
  If we take $M=V$ in the theorem and use \autoref{prop:FHonV}, we get that
  $A \simeq \Omega^V\mathrm{B}^VA$ for a $G$-connected $\mathrm{E}_V$-algebra $A$.
  This recovers \cite[Theorem 1.14]{GMPermG} and justifies the definition of $\mathrm{B}^{V}A$.
\end{rem}

\subsection{$G$-connectedness}
\begin{defn}
  A $G$-space $X$ is $G$-connected if $X^H$ is connected for all subgroups $H\subgroup G$.
\end{defn}
To show that the scanning map is an equivalence in each simplicial level, we need:
\begin{lem}
\label{lem:Gconnected}
  If $X$ is $G$-connected, then $\mathrm{D}^{\mathrm{fr}_V}_V X$ is also $G$-connected.
\end{lem}
\begin{proof}By \autoref{cor:conf}, $\mathrm{D}^{\mathrm{fr}_V}_V X$ is $G$-homotopy equivalent to
  $F_VX$. It suffices to show that $F_VX$ is $G$-connected. Fix any subgroup $H\subgroup G$; we must show that
  $(F_VX)^H$ is connected. This is the space of $H$-equivariant unordered configuration on $V$ with
  based labels in $X$. Intuitively, this is true because the
  space of labels $X$ is $G$-connected, so that one can always move the labels of a configuration to the
  base point. Nevertheless, we give a proof here by carefully writing down the fixed points of $F_VX$ in
  terms of the fixed points of $\mathscr{F}_V(k)$ and $X$.
 We have:
  \begin{align*}
    (F_VX)^{H} & = (\coprod_{k \geq 0} F_V (k) \times_{\Sigma_k} X^{k}
                 /\sim)^{H}
     = \coprod_{k \geq 0} (F_V (k) \times_{\Sigma_k}  X^{k})^{H}/\sim_{H}
\end{align*}
 Here, $\sim$ is the equivalence relation in \autoref{rmk:equi-relation} 
 and $\sim_{H}$ is $\sim$ restricted on $H$-fixed points. They are explicitly forgetting a point in
 the configuration if the corresponding label is the base point in $X$.
 Notice that taking $H$-fixed points  will not commute with $\approx$ in \autoref{defn:functor}, but
 commutes with $\sim$. This is because the $H$-action preserves the filtration and $\sim$ only
 identifies elements of different filtrations. 
The single point at filtration $k=0$, or equivalently the point at any $k$ with
all labels being the base point of $X$, is the base point of $(F_VX)^H$.

 Since the $\Sigma_k$-action is free on $F_V(k) \times X^k$ and commutes with the $G$-action,
 we have a principal $G$-$\Sigma_k$-bundle 
\begin{equation*}
F_V(k) \times X^k \to F_V(k) \times_{\Sigma_k} X^k.
\end{equation*}
 To get $H$-fixed points on the base space, we need to consider the $\Lambda_{\alpha}$-fixed points on the
 total space for all the subgroups
$\Lambda_{\alpha} \subgroup  G \times \Sigma_k$ that are the graphs of some group homomorphisms $\alpha: H
\to \Sigma_k$.  More precisely, by \autoref{thm:LM}, we have
\begin{equation*}
     (F_V (k) \times_{\Sigma_k}  X^{k})^{H} 
               = \coprod_{ [\alpha: H \to \Sigma_k]}
                \Big((F_V(k) \times  X^k)^{\Lambda_{\alpha}} /Z_{\Sigma_k}(\alpha)\Big).
\end{equation*}
Here, the coproduct is taken over $\Sigma_k$-conjugacy classes of group homomorphisms and
$Z_{\Sigma_k}(\alpha)$ is the centralizer of the image of $\alpha$ in $\Sigma_k$.

We would like to make the expression coordinate-free for $k$. A homomorphism $\alpha$ can be
identified with an $H$-action on the set $\{1, \cdots, k\}$.
For an $H$-set $S$, write $X^S = \mathrm{Map}(S,X)$ and $F_V(S) = \mathrm{Emb}(S, V)$.  Then
\begin{equation*}
  (F_V(k) \times  X^k)^{\Lambda_{\alpha}} = ( F_V(S) \times  X^{S})^H \text{ and }
  Z_{\Sigma_k}(\alpha)= \mathrm{Aut}_H(S).  
\end{equation*}
So we have:
 \begin{equation*}
 (F_V (k) \times_{\Sigma_k}  X^{k})^{H} 
   = \coprod_{[S]:  \text{iso classes of $H$-set} , |S|=k} \Big( (F_V(S) \times
    X^{S})^H/\mathrm{Aut}_H(S) \Big).
 \end{equation*}
If we take care of the base point identification, we end up with:
\begin{equation}
  \label{eq:10}
  (F_VX)^H = \bigg( \coprod_{[S]: \text{iso classes of finite $H$-set}} (F_V(S) \times
X^S)^H / \mathrm{Aut}_H(S) \bigg) /\sim_H. 
\end{equation}

Suppose that the $H$-set $S$ breaks into orbits as $S = \amalg_i r_{i}(H/K_i)$ for $i=1,\cdots, s$,
 where $K_{i}$'s are in distinct conjugacy classes of subgroups of $H$ and $r_i > 0$, then we know
 explicitly each coproduct component is:
\begin{align*}
     (F_V(S) \times X^{S})^H/\mathrm{Aut}_H S & = (\mathrm{Emb}_{H}(S, V) \times
 \mathrm{Map}_H(S, X))/\mathrm{Aut}_H S \\
& = (\mathrm{Emb}_{H}(\amalg_{i} r_i (H/K_{i}), V) \times \prod_{i} (X^{K_i})^{r_{i}})/ \prod_i
  (W_H(K_i) \wr \Sigma_{r_i}).
\end{align*}
Since $X^{K_i}$ are all connected, so are the spaces $\prod_{i} (X^{K_i})^{r_{i}}$.
They contain the base point of the labels $*=\prod_i\prod_{r_i} * \to \prod_i (X^{K_i})^{r_{i}}$. So after
the gluing $\sim_{H}$, each component in \autoref{eq:10} is in the same component as the base point of $F_VX$. Thus $(F_VX)^{H}$ is connected.
\end{proof}

\begin{cor}
  \label{cor:simplicialScanning}
   The map 
$\mathrm{B}_{\bullet}(\mathrm{D}^{\mathrm{fr}_V}_{\myM} , \mathrm{D}^{\mathrm{fr}_V}_V, A) \to
   \mathrm{Map}_c (M, \Sigma^V   (\mathrm{D}^{\mathrm{fr}_{V}}_V) ^{\bullet} A)$ in
   \autoref{eq:18} is a levelwise weak
   $G$-equivalence of simplicial $G$-spaces if $A$ is $G$-connected.
\end{cor}
\begin{proof}
  This is a consequence of \autoref{thm:scanning-equi} and \autoref{lem:Gconnected}.
\end{proof}

For geometric realization, we have:
\begin{thm}[Theorem 1.10 of \cite{MMO}]
  \label{thm:Reedy}
A levelwise weak $G$-equivalence between Reedy cofibrant simplicial objects realizes to a weak $G$-equivalence.
\end{thm}

\subsection{Cofibrancy}
We take care of the cofibrancy issues in this part, following details in \cite{MayGILS}. 
We first show that some functors preserve $G$-cofibrations. 
One who is willing to take it as a blackbox may skip directly to \autoref{lem:Reedysi}.
We uses NDR data, which give a hands-on way to handle cofibrations.
\begin{defn}[Definition A.1 of \cite{MayGILS}]
  \label{defn:NDR}
  A pair $(X,A)$ of $G$-spaces with $A \subset X$ is an NDR pair if there exists a $G$-invariant map
  $u : X \to I  = [0,1]$ such
  that $A = u^{-1}(0)$ and a homotopy given by a map $h: I \to \mathrm{Map}_G(X,X)$ satisfying
\begin{itemize}
\item $h_0(x) = x $ for all $x \in X$;
\item  $h_t(a) = a$ for all $t \in I$ and $a \in A$;
\item $h_1(x) \in A$ for all $x \in  u^{-1}[0,1)$. 
\end{itemize} The pair $(h,u)$ is said to a representation of $(X,A)$ as an NDR pair.
A pair $(X,A)$ of based $G$-spaces is an NDR pair if it is an NDR pair of $G$-spaces with the $h_t$
being based maps for all $t \in I$.
\end{defn}

An NDR pair gives a $G$-cofibration $A \to X$.
The function $u$ gives an open neighboorhood $U$ of $A$ by taking $U=u^{-1}[0,1)$.
The function $h$ restricts on $I \times U$ to a neighborhood deformation retract
of $A$ in $X$.

We have the following lemma by elaborating the NDR data.
Its proof is tedious and omitted here (See \cite[Section 6.4]{mythesis}).
\begin{lem}
  \label{lem:NDR}
  Any functor $F$ associated to  $\mathscr{F} \in \Lambda^{op}_{*}[\Gtop]$, in
  particular both $\mathrm{D}^{\mathrm{fr}_{V}}_V$ and
  $\mathrm{D}^{\mathrm{fr}_{V}}_{\myM}$, sends NDR pairs to NDR pairs.
  The functors  $\mathrm{Map}_c (M,-)$, $\mathrm{Map}_{*}(M^+,-)$ and 
 $\Sigma^V$ all send NDR pairs to NDR pairs.
\end{lem}

\begin{defn}[Lemma 1.9 of \cite{MMO}]
  \label{lem:Reedysi}
A simplicial $G$-space $X_{\bullet}$ is Reedy cofibrant if all degeneracy operators $s_i$ are $G$-cofibrations.
\end{defn}

The following lemma shows that monadic bar constructions are Reedy cofibrant.
\begin{lem}[adaptation of Proposition A.10 of \cite{MayGILS}]
  \label{lem:GILSA10}
  Let $\mathscr{C}$ be a reduced operad in $G$-spaces such that the unit map $\eta: * \to
  \mathscr{C}(1)$ gives a non-degenerate base point.
  Let $C$ be the reduced monad associated to $\mathscr{C}$. Let $A$ be a $C$-algebra in $\Gtop_{*}$ and
  $F: \Gtop_{*} \to \Gtop_{*}$ be a right-$C$-module functor.
  Suppose that $F$ sends NDR pairs to NDR pairs. Then
  $B_{\bullet}(F,C,A)$ is Reedy cofibrant.
\end{lem}

\begin{proof}
  We need to show that for any $n \geq 0$ and $0 \leq i \leq n$,
  the degeneracy map $${s^i_n = FC^i\eta_{C^{n-i}A}: FC^nA \to  FC^{n+1}A}$$ is a $G$-cofibration.
  Write $X = C^{n-i}A$. By \autoref{lem:NDR}, $C$ sends NDR pairs to NDR pairs.
  Starting from the NDR pair $(A, *)$ and applying this functor $(n-i)$ times, we get an NDR
  pair $(C^{n-i}A, *) = (X, *)$.
  Together with the assumption that $\mathscr{C}(1)$ is
  non-degenerately based, we can show $(CX, X)$ is an
  NDR pair where $X$ is identified with the image $\eta_X: X \to CX$ (see the
  proof of \cite[A.10]{MayGILS}). Applying $C$ another $i$ times
  and then $F$, we get the NDR pair  $\big(FC^{i+1}X, FC^iX\big)=\big(FC^{n+1}A, FC^nA\big)$.
  Thus $s^i_n = FC^i\eta_X$ is a $G$-cofibration. 
\end{proof}

\begin{cor}
  \label{lem:Reedy}
  Let $M,V,A$ be as in \autoref{thm:NPDV}.
  Then the following are Reedy cofibrant simplicial $G$-spaces:
\begin{equation*}
  \mathrm{B}_{\bullet}(\mathrm{D}^{\mathrm{fr}_V}_{\myM} ,{D}^{\mathrm{fr}_V}_V, A),
  \ \mathrm{Map}_c (M, \Sigma^V   (\mathrm{D}^{\mathrm{fr}_{V}}_V)^{\bullet} A)
  \text{ and }\mathrm{Map}_{*} (M^+, \Sigma^V   (\mathrm{D}^{\mathrm{fr}_{V}}_V)^{\bullet} A). 
\end{equation*}
\end{cor}
\begin{proof}  
  In \autoref{lem:GILSA10}, we take $C = \mathrm{D}^{\mathrm{fr}_{V}}_V $ and  respectively
  $F =  \mathrm{D}^{\mathrm{fr}_{V}}_{\myM}$, $ F = \mathrm{Map}_c (M, \Sigma^V -)$
  or $ F = \mathrm{Map}_* (M^+, \Sigma^V -)$.
  By \autoref{lem:NDR}, each $F$ does send NDR pairs to NDR pairs.
\end{proof}

\subsection{Dimension}
\label{sec:dimension}
We start by recalling some facts about $G$-CW complexes and equivariant dimensions following \cite[I.3]{MayAlaska}.
A $G$-CW complex $X$ is a union of $G$-spaces $X^n$,
where $X^0$ is a disjoint union of orbits,
and $X^n$ is obtained by inductively gluing cells
$G/K \times D^n$ for subgroups $K \subgroup G$ via $G$-maps  along their boundaries $G/K \times S^{n-1}$
to the previous skeleton $X^{n-1}$.

We shall look at functions from the conjugacy classes of subgroups of $G$ to $\bZ_{\geq -1}$ and
typically denote such a function by $\nu$. We say that a 
$G$-CW complex $X$ has dimension $\leq \nu$ if its cells of orbit type $G/H$ all have
dimensions $\leq \nu(H)$, and that a $G$-space $X$ is $\nu$-connected if $X^H$ is $\nu(H)$-connected
for all subgroups $H\subgroup G$, that is, $\pi_k(X^H) = 0$ for $k \leq \nu(H)$.
We allow $\nu(H)=-1$ for the case $X^H=\varnothing$.

It is worth pointing out that this notion of dimension should be more appropriately called
the cell dimension.
(It is \emph{not} the dimension of $X^H$, as we explain
shortly.) It gives information on which cells to
consider in an induction.  For the purpose of induction, we use the following \emph{ad hoc}
definition in this paper:
\begin{defn}
  A based $G$-CW complex is a union of $G$-spaces $X^n$ obtained by inductively
  gluing cells to $X^0$, a disjoint union of orbits plus a disjoint base
  point $*$. (The gluing maps are non-based maps.)
 In a based map out of $X$,  the base point $*$ has no freedom but to be sent to
 the base point. So we do NOT count it as a cell for a based $G$-CW complex,
 excluding it from counting the dimension as well.
 It then makes sense to write $X^{-1} = *$. 
  This is not the same as a based $G$-CW complex in \cite[Page 18]{MayAlaska}, where the base point
  is put in the 0-skeleton $X^0$.
\end{defn}

Fix a subgroup $H \subgroup G$. A function $\nu$ from the
conjugacy classes of subgroups of $G$ to $\bZ_{\geq -1}$ induces a function from the conjugacy
classes of subgroups of $H$ to $\bZ_{\geq -1}$, which we still call $\nu$.
We have the double coset formula
\begin{equation}
\label{eq:27}
  G/K \cong \coprod_{1 \leq i \leq |H\backslash G/K|} H/K_i \text{ as $H$-sets,}
\end{equation}
where each $K_i = H \cap g_iKg_i^{-1} $ for some element $g_i \in G$. So a (based) $G$-CW structure on $X$
restricts to a (based) $H$-CW structure on the $H$-space $\mathrm{Res}^G_HX$. 
However, for $X$ of \dimension $\leq \nu$, $\mathrm{Res}^G_H X$ may not be of \dimension $\leq \nu$,
as we see in \autoref{eq:27} that an $H/K_i$-cell can come from a $G/K$-cell for a larger group $K$.
For a function $\nu$, we define the function $d_{\nu}$ to be
\begin{equation}
\label{eq:dnu}
d_{\nu}(K) = \max\limits_{K \subgroup L} \nu(L).
\end{equation}
Then $\mathrm{Res}^G_H X$ is of \dimension
$\leq d_{\nu}$.

\begin{rem}
  \label{rem:dim}
  More specifically, we define the \dimension of a (based) $G$-CW complex $X$
  to be the minimum $\nu$ such that $X$ is of \dimension $\leq \nu$.
  Suppose that $X$ has \dimension $\nu$.
  From \autoref{eq:27}, we get:
\begin{enumerate}[(i)]
\item The (based) $H$-CW complex $\mathrm{Res}^G_H X$ has \dimension $\nu_H$, where 
\begin{equation*}
\nu_{H}(K) = \max\limits_{\substack{K \subgroup L \\ K = L \cap H}} \nu(L).
\end{equation*}
We have $\nu_H(K) \leq d_{\nu}(K)$, and it can be strictly less. (For a trivial example, take $H = G$.)
\item The (based) CW-complex $X^H$ has dimension $\nu_H(H) = d_{\nu}(H) \geq
  \nu(H)$.
  (In the based case, we also exclude the base point from counting the dimension
  of $X^H$, so that if $X^H=*$, the dimension of $X^H$ is -1.) 
\end{enumerate}
\end{rem}
\begin{defn}
\begin{enumerate}
\item For a (based) $G$-CW complex $X$ of \dimension $\nu$,  $\mathrm{dim}(X)$ is
  the function $d_{\nu}$. 
\item For a $G$-representation $V$, $\mathrm{dim}(V)$ is the function $\mathrm{dim}(V)(H)= \mathrm{dim}(V^H)$.
\end{enumerate}
\end{defn}
From \autoref{rem:dim}, we have two observations: First, $\mathrm{dim}(X)(H)$ is equal to
the dimension of the CW-complex $X^H$. So $\mathrm{dim}(X)$ is independent of
the $G$-CW decomposition of the underlying $G$-space of $X$. Second, for a unbased $G$-CW complex $X$,
the based $G$-CW complex $X_+ = X \amalg {*}$ satisfies
$\mathrm{dim}(X_+) = \mathrm{dim}(X)$ because  $*$ is excluded
from cells in the based case.

\medskip
We prepare the following results regarding dimension for the next subsection.
\begin{thm}[Theorem 3.6 of \cite{Illman}]
\label{lem:Illman}
For a smooth $G$-manifold $M$ and a closed smooth $G$-submanifold $N$, there exists a smooth
$G$-equivariant triangulation of $(M,N)$. 
\end{thm}

\begin{lem}
\label{lem:NPDV}
  Let $M$ be a $V$-framed manifold and $A$ be a $G$-space, then 
\begin{enumerate}
\item\label{item:NPDV1} $M^{+}$ has the homotopy type of a $G$-CW complex of \dimension $\leq \mathrm{dim}(V)$.
\item\label{item:NPDV2} $K_{n} = \Sigma^V  (\mathrm{D}^{\mathrm{fr}_{V}}_V)
  ^{n} A$ is $(\mathrm{dim}(V)-1)$-connected. If furthermore $A$ is
  $G$-connected, then $K_n$ is $\mathrm{dim}(V)$-connected.
\end{enumerate}
\end{lem}
\begin{proof}
  \autoref{item:NPDV1}
  Since $M$ is a $V$-framed, the exponential maps give local coordinate charts of $M^H$ as a
  (possibly empty) manifold of dimension $\mathrm{dim}(V^H)$.
  If $M$ is compact we take $W = M$, otherwise we take a manifold $W$ with boundary such that $M$
  is diffeomorphic to the interior of $W$.
  By \autoref{lem:Illman}, $(W,\partial W) $ has a
  $G$-equivariant triangulation. It gives a relative $G$-CW structure on $(W,\partial W)$ with relative
  cells of type $G/H$ of dimension $\leq \mathrm{dim}(V^{H})$. The quotient $W/\partial
  W$ gives the desired $G$-CW model for $M^+$.

  \autoref{item:NPDV2}
  For any subgroup $H\subgroup G$, we have
  $K_n^H = (\Sigma^V  (\mathrm{D}^{\mathrm{fr}_{V}}_V) ^{n} A)^H = \Sigma^{V^{H}} ((\mathrm{D}^{\mathrm{fr}_{V}}_V)
  ^{n} A)^{H}$.
  Then $(K_n)^H$ is obviously $(\mathrm{dim}(V^H)-1)$-connected.
  When $A$ is $G$-connected,
  by \autoref{lem:Gconnected}, $((\mathrm{D}^{\mathrm{fr}_{V}}_V)^{n} A)^{H}$ is
  connected, so that
  $K_n^H$ is $\mathrm{dim}(V^H)$-connected.
\end{proof}

\subsection{Commuting mapping space and geometric realization}
Let $X$ be a based $G$-CW complex and $K_{\bullet}$ be a simplicial $G$-space. Then the
levelwise evaluation is a $G$-map
\begin{equation*}
|\mathrm{Map}_{*}(X, K_{\bullet})| \sm X \cong |\mathrm{Map}_{*}(X, K_{\bullet}) \sm X| \to |K_{\bullet}|,
\end{equation*}
whose adjoint gives a $G$-map
\begin{equation}
  \label{eq:MapToSset}
\zeta: |\mathrm{Map}_{*}(X, K_{\bullet})| \to \mathrm{Map}_{*}(X, |K_{\bullet}|).
\end{equation}
Non-equivariantly, it is one of the key steps in May's recognition principal that
\autoref{eq:MapToSset} is a weak equivalence when each $K_{\bullet}$ is
$\mathrm{dim}(X)$-connected \cite[Theorem 12.3]{MayGILS}. The goal of this subsection is to give a sufficent condition for 
$\zeta$ to be a weak $G$-equivalence.


The strategy is to induce on cells. 
However, the geometric realization of a levelwise fibration is not necessarily a fibration. 
Dold--Thom came up with the notion of quasi-fibrations, which is good enough for handling the homotopy groups.
\begin{defn}
  A map $p: Y \to W$ of spaces is a quasi-fibration if $p$ is onto and it induces an isomorphism on
  homotopy groups  $\pi_{*}(Y,p^{-1}(w), y) \rightarrow \pi_{*}(W,w)$ for all $w \in W$ and $y \in
  p^{-1}(w)$. In other words, there is a long exact sequence on homotopy groups of the sequence
  $p^{-1}(w) \to Y \to W$ for any $w \in W$.
\end{defn}

\begin{thm}(\cite[Theorem 12.7]{MayGILS}) \label{thm:quasifib}
  Let $p: E_{\bullet} \to B_{\bullet}$ be a levelwise Hurewicz fibration of pointed simplicial
  spaces such that $B_{\bullet}$ is Reedy cofibrant and $B_n$ is connected for all $n$.
  Set  $F_{\bullet} = p^{-1}(*)$.
  Then the realization  $|E_{\bullet}| \to |B_{\bullet}|$ is a quasi-fibration with fiber
  $|F_{\bullet}|$. 
\end{thm}

\begin{thm}
  \label{thm:HM} Let $G$ be a finite group.
  If $X$ is a finite-dimensional based $G$-CW complex and $K_{\bullet}$ is a simplicial $G$-space such that for any $n$, $K_n$ is $\mathrm{dim}(X)$-connected, then the natural map \autoref{eq:MapToSset}
\begin{equation*}
\zeta: |\mathrm{Map}_{*}(X, K_{\bullet})| \to \mathrm{Map}_{*}(X, |K_{\bullet}|)
\end{equation*}
is a weak $G$-equivalence.
\end{thm}
\begin{proof} Suppose that $X$ is  of \dimension $\nu$, so $\mathrm{dim}(X) =
  d_{\nu}$. (See \autoref{eq:dnu} for $d_{\nu}$.)
Let  $* = X^{-1} \subset X^0 \subset X^1 \subset \cdots \subset X^{d_\nu(e)} = X$ be the $G$-CW
skeleton of $X$.
We use induction on $k$ to show that 
\begin{enumerate}[(i)]
\item \label{item:induct2}$\mathrm{Map}_{*}(X^k, K_n)^H$ is connected for all $n$ and $H \subgroup G$.
\item \label{item:induct1} $|\mathrm{Map}_{*}(X^k, K_{\bullet})|^H \to \mathrm{Map}_{*}(X^k, |K_{\bullet}|)^H$ is a
weak equivalence for all $H \subgroup G$;
\end{enumerate}

The base case $k=-1$ is obvious. Suppose that \autoref{item:induct2} and
\autoref{item:induct1} hold for $k$. Take the cofiber sequence
$$X^{k} \to X^{k+1} \to X^{k+1}/X^{k}$$
and map it into $K_{\bullet}$. We then apply 
\autoref{eq:MapToSset} and get the following commutative diagram:
\begin{equation}
  \label{equ:quasifib}
  \begin{tikzcd}
    \vert \mathrm{Map}_{*}(X^{k+1}/X^k, K_{\bullet})\vert ^H  \ar[d] \ar[r] & \vert \mathrm{Map}_{*}(X^{k+1}, K_{\bullet})\vert ^H  \ar[d] \ar[r] &
    \vert \mathrm{Map}_{*}(X^k, K_{\bullet})\vert ^H  \ar[d] \\
    \mathrm{Map}_{*}(X^{k+1}/X^k, |K_{\bullet}|)^H \ar[r] & \mathrm{Map}_{*}(X^{k+1}, |K_{\bullet}|)^H \ar[r]
    &\mathrm{Map}_{*}(X^k, |K_{\bullet}|)^H
  \end{tikzcd}
\end{equation}
Since maps out of a cofiber sequence form a fiber sequence, we have a fiber sequence in the second row
and a realization of the following levelwise fiber sequence in the first row:
\begin{equation}
  \label{eq:quasifib2}
  \begin{tikzcd}
     \mathrm{Map}_{*}(X^{k+1}/X^k, K_{\bullet}) ^H   \ar[r] &  \mathrm{Map}_{*}(X^{k+1}, K_{\bullet}) ^H   \ar[r] &
     \mathrm{Map}_{*}(X^k, K_{\bullet}) ^H 
  \end{tikzcd}
\end{equation}
By the inductive hypothesis~\autoref{item:induct2} and \autoref{thm:quasifib}, it realizes to a
quasi-fibration.

We first show the inductive case of \autoref{item:induct2}. We can write
  $$X^{k+1}/X^k =  \vee_i (G/K_i)_+ \sm S^{k+1},$$
where each $K_i$ is a subgroup of $G$. When $K_i$ is presented, $\nu(K_i) \geq k+1$.
From \autoref{eq:27}, we can further write $X^{k+1}/X^k \cong \vee_i \vee_j
(H/K_{i,j})_+ \sm S^{k+1}$ as spaces with $H$-action, where each
$K_{i,j}$ is $G$-conjugate to a subgroup of $K_i$. Then $d_{\nu}(K_{i,j}) \geq \nu(K_i) \geq k+1$, and the following space is connected by assumption:
\begin{equation*}
\mathrm{Map}_{*}(X^{k+1}/X^k, K_{n}) ^H  = \prod_i \mathrm{Map}_{*}(S^{k+1}, K_n^{K_{i,j}}).
\end{equation*}
This space is the fiber in \autoref{eq:quasifib2}. The connectedness of the base
space given by ~\autoref{item:induct2} then implies the connectedness of the total space.

We next show the inductive case of \autoref{item:induct1}.
Commuting geometric realization with finite products and with fixed points, the left vertical map  of
\autoref{equ:quasifib} is a product of maps
\begin{equation*}
|\mathrm{Map}_{*}(S^{k+1}, K_{\bullet}^{K_{i,j}})| \to \mathrm{Map}_{*}(S^{k+1}, |K_{\bullet}^{K_{i,j}}|).
\end{equation*}
Since we have $d_{\nu}(K_{i,j}) \geq k+1$, these maps are weak equivalences by \cite[Theorem 12.3]{MayGILS}.
By~\autoref{item:induct1}, the right vertical map
is a weak equivalence. Comparing the long exact sequences of homotopy groups, this implies that the
middle vertical map is also a weak equivalence.
\end{proof}

\begin{rem}
 Non-equivariantly, Miller  \cite[Cor 2.22]{Miller} observed that the theorem is also true
 if $K_n$ is only $(\mathrm{dim}(X)-1)$-connected for all $n$, since the only
 thing that fails in the proof is the claim \autoref{item:induct2} for $k=\mathrm{dim}(X^e)$.
 Equivariantly, one needs  \autoref{item:induct2} to
 hold for all inductive steps of $k<d_\nu(e)$. So we can only relax the
 assumption to the following extent:
 If $K_n^H$ is  $\min\{d_\nu(H), d_\nu(e)-1\}$-connected for all $n$ and $H$, then the natural map \autoref{eq:MapToSset}
is a weak $G$-equivalence. This is an improvement only when $d_{\nu}(H) =
d_{\nu}(e)$, that is $d_\nu(H) \geq \nu(K)$ for all $K \subset H$.
\end{rem}
Nevertheless, when $X = \Sigma Z$ and $Z$ is of \dimension $\nu$, so that $X$ is
of \dimension $\nu+1$, we can relax the assumption further.
\begin{cor}
  \label{cor:connective}
  If $Z$ is a finite-dimensional based $G$-CW complex and $K_{\bullet}$ is a simplicial $G$-space such that for any $n$, $K_n$ is $\mathrm{dim}(Z)$-connected, then the natural map \autoref{eq:MapToSset}
\begin{equation*}
\zeta: |\mathrm{Map}_{*}(\Sigma Z, K_{\bullet})| \to \mathrm{Map}_{*}(\Sigma Z, |K_{\bullet}|)
\end{equation*}
is a weak $G$-equivalence.
\end{cor}
\begin{proof}
  The cofiber sequence $S^0 \vee S^0 \to S^0 \to S^1 $ gives a levelwise fiber sequence
  \begin{equation}
    \label{eq:21}
  \begin{tikzcd}
     \mathrm{Map}_{*}(\Sigma Z, K_{\bullet})  \ar[r] &  \mathrm{Map}_{*}(Z, K_{\bullet}) \ar[r] &
    \mathrm{Map}_{*}(Z,K_{\bullet}) \times \mathrm{Map}_{*}(Z,K_{\bullet}) .
  \end{tikzcd}
\end{equation}
By \autoref{thm:HM} and its proof, \autoref{eq:21} has a $G$-connected base and
realizes to a quasi-fibration; the same method will show the claim.
\end{proof}

The unbased version of \autoref{thm:HM} is due to Hauschild and written down
by Costenoble--Waner \cite[Lemma 5.4]{CW91}, stated as:
\begin{thm} \label{thm:CW}
   Let $G$ be a finite group.
  If $Y$ is a finite unbased $G$-CW complex and $K_{\bullet}$ is a simplicial
  $G$-space such that for any $n$, $K_n$ is
  $\mathrm{dim}(Y)$-connected, then the natural map
\begin{equation*}
|\mathrm{Map}(Y, K_{\bullet})| \to \mathrm{Map}(Y, |K_{\bullet}|)
\end{equation*}
is a weak $G$-equivalence. 
\end{thm}
\noindent
\autoref{thm:HM} improves \autoref{thm:CW} slightly in the case when $X^G = *$.
On one hand, taking $X$ in \autoref{thm:CW} to be $Y \amalg \{*\}$ recovers \autoref{thm:HM}.
On the other hand, for a based $G$-CW complex $X$ we have the levelwise
fibration sequence
$$\mathrm{Map}_{*}(X, K_{\bullet}) \to \mathrm{Map}(X, K_{\bullet}) \to K_{\bullet}.$$
If the \dimension of $X$ satisfies $\nu(H) \geq 0$ for all $H$, then
$\mathrm{dim}(X)(H) = d_{\nu}(H) \geq 0$. The assumptions
imply that $K_n$ is $G$-connected, we can use the quasi-fibration
technique to deduce \autoref{thm:HM} from \autoref{thm:CW} (with $Y = X$). But there are also cases when the assumption in
 \autoref{thm:HM} is weaker, for example, when $X = (G/H)_+ \sm S^n$
 for some $H \neq G$. In this case, $d_{\nu}(G) = \mathrm{dim}(X^G) = -1$, so
 the $K_n^G$ are required to be connected in \autoref{thm:CW} but not in \autoref{thm:HM}.

 \subsection{Group completion}
 \label{sec:group-completion}
Recall that an $\mathrm{E}_n$-structure on a $G$-space is an algebra structure
over the little disk operad $\mathscr{D}_n$ for the trivial representation
$\mathbb{R}^n$. As pointed out in \cite[Section 1.2]{GMPermG}, there are two notions of group completion, one topological, one
computational, which we recall now.

\begin{defn} \label{defn:weak-group-cplt}
  Let $C$ and $D$ be $\mathrm{E}_1$-$G$-spaces.
  An $\mathrm{E}_1$-$G$-map $f: C \to D$ is called a weak group completion if for any subgroup $H\subgroup G$,
  there is a homotopy equivalence $\Omega\mathrm{B}(C^H) \simeq D^H$ and $f^H$ is homotopic to
  $C^H \to \Omega\mathrm{B}(C^H) \simeq D^H$.
\end{defn}
\noindent When $C$ is an $\mathrm{E}_1$-$G$-space and $H \subgroup G$, the fixed point space $C^H$ is an
$\mathrm{E}_1$-space; so $f^H$ is up to homotopy a weak group completion of $C^H$.
\begin{defn} \label{defn:grp-complete} Let $C$ and $D$ be $\mathrm{E}_2$-$G$-spaces\footnote{This definition makes sense for homotopy associative and commutative
$G$-monoids, for which $\mathrm{E}_2$-$G$-spaces are examples.}.
\begin{enumerate}
\item $D$ is called grouplike if 
 for any  subgroup $H\subgroup G$, $\pi_0^H(D)$ is a group.
\item A $\mathrm{E}_2$-$G$-map $f: C \to D$ is
  called a group completion if $D$ is grouplike and for any
  subgroup $H\subgroup G$,
  $f^H$ induces an isomorphism $H_{*}(C^H)[\pi_0^H(C)^{-1}] \cong
  H_{*}(D^H)$ for any field coefficients.
\end{enumerate}
\end{defn}

\begin{thm}(\cite[15.1]{MayClassify})
   Let $C$ and $D$ be $\mathrm{E}_2$-$G$-spaces.
  Then a weak group completion $f: C \to D$ is a group completion.
\end{thm}

\begin{lem}
  \label{lem:grp-cplt-realize}
  Let $C_{\bullet}$ and $D_{\bullet}$ be Reedy cofibrant simplicial $\mathrm{E}_1$-$G$-spaces.
  Suppose that $f: C_{\bullet} \to D_{\bullet}$ be a levelwise weak group
  completion. Then $f$ induces a weak group completion $|C_{\bullet}| \to
  |D_{\bullet}|$. If $C_{\bullet}$ and $D_{\bullet}$ are levelwise
  $\mathrm{E}_2$, then $f$ induces a group completion.
\end{lem}
\begin{proof} The $\mathrm{E}_n$-$G$-space structures are algebra structures over certain monads and
  thus preserved by geometric realization (\cite[Lemma 8.17]{cellular}).
  The functor $\mathrm{B}$ is the geometric realization of
  a simplicial construction   $\mathrm{B}_m(-)$. So
$\mathrm{B}|C_{\bullet}^H|$
and  $|\mathrm{B}C_{\bullet}^H|$, being two ways of realizing the bisimplicial space
$X_{m,n} = \mathrm{B}_m C_n^H $, are homeomorphic.
We have the following commutative diagram:
\begin{equation*}
  \begin{tikzcd}
    \mid C_{\bullet}^H \mid \ar[d] \ar[r]&
    \mid\Omega \mathrm{B}C_{\bullet}^H \mid \ar[d,"\sim"',"\zeta"] \ar[r, "\sim"]& \mid D_{\bullet}^H\mid\\
 \Omega \mathrm{B}|C_{\bullet}^H| \ar[r, "\cong"] &  \Omega
 |\mathrm{B}C_{\bullet}^H| & 
  \end{tikzcd}
\end{equation*}
The top right map is induced by $\Omega B C_{\bullet}^H \to
D_{\bullet}^H$.
From the assumptions, it is a levelwise equivalence between Reedy cofibrant simplicial
spaces, so the top right map is a weak equivalence. Each $\mathrm{B}C_n^H$ is connected, so the
vertical map $\zeta$ is a weak equivalence by \autoref{cor:connective}. This
proves that the top composite is homotopic to the left arrow up to equivalence.
\end{proof}
\begin{thm}
\label{thm:grp}
  Let $M$ be a $V$-framed manifold and $A$ be a
  $\mathrm{D}^{\mathrm{fr}_V}_V$-algebra in $\Gtop$. There is a $G$-map from \autoref{thm:NPDV}
$$p_M: \int_M A \to \mathrm{Map}_*(M^+, \mathrm{B}^VA).$$
\begin{enumerate}
\item \label{item:grp-1} If $V = W \oplus \bR$ and $M \cong N \times \mathbb{R}$ for a $W$-framed
  manifold $N$, then $p_M$ is a weak group completion.
\item \label{item:grp-2} If $V = U \oplus \bR^{2}$
  and $M \cong N \times \mathbb{R}^2$ for a $U$-framed
  manifold $N$, then $p_M$ is a group completion.
\end{enumerate}
\end{thm}
\begin{proof} From the proof of \autoref{thm:NPDV}, the map $p_M$ is a composite 
  \begin{equation*}
  \begin{tikzcd}
   \mid B_{\bullet}(\mathrm{D}^{\mathrm{fr}_V}_{\myM} ,
    \mathrm{D}^{\mathrm{fr}_V}_V, A)\mid \ar[r,"\alpha_M"]
    & \mid\mathrm{Map}_{*} (M^{+}, \Sigma^V (\mathrm{D}^{\mathrm{fr}_{V}}_V)
    ^{\bullet} A)\mid \ar[r, "\zeta"]
    & \mathrm{Map}_*(M^+, \mathrm{B}^VA)
  \end{tikzcd}
\end{equation*}

We first examine $\alpha_M$. 
By \autoref{thm:scanning-equi}, $\alpha_M$ is the realization of a levelwise
weak group completion between simplicial $\mathrm{E}_1$-$G$-spaces in case of
\autoref{item:grp-1} and $\mathrm{E}_2$-$G$-spaces in case of
\autoref{item:grp-2}. Then by \autoref{lem:grp-cplt-realize}, $\alpha_M$ is a
weak group completion in case of \autoref{item:grp-1} and a group completion in
case of \autoref{item:grp-2}.

Next we proof that $\zeta$ is a weak $G$-equivalence in case
\autoref{item:grp-1}, and case \autoref{item:grp-2} will follow.
 By \autoref{lem:NPDV}~\autoref{item:grp-2}, $ \Sigma^{V}(\mathrm{D}^{\mathrm{fr}_{V}}_V)
  ^{\bullet} A$ is $(\mathrm{dim}(V)-1) = \mathrm{dim}(W)$-connected.  Applying \autoref{lem:NPDV}~\autoref{item:grp-1} to $N$, it has
  a $G$-CW structure of \dimension $\leq \mathrm{dim}(W)$. By
  \autoref{cor:connective} and the fact that $M^+ \simeq
  \Sigma(N^+)$, $\zeta$ is a weak $G$-equivalence. This finishes the proof.
\end{proof}
\appendix
\section{A comparison of scanning maps}
\label{chap:appendix-scanning}
The scanning map studied in \autoref{sec:scanning} is a key input to the eNPD theorem. In this section we compare our scanning map \autoref{eq:sX} to other constructions.

\begin{notn}
  For a $G$-manifold $M$, $\mathrm{Sph}(\mathrm{T}M)$ is the $G$-space obtained by fiberwise one-point compatification of
  the tangent bundle of $M$.
  It is a fiber bundle over $M$ with based fiber $S^n$, where the base point in each fiber is
  the point at infinity.
\end{notn}

Non-equivariantly, people have used the name scanning map to refer to different but related constructions.
In slogan, it is a map from the (fattened) configuration spaces of a manifold $M$ to compactly
defined sections of $\mathrm{T}M$, or compactly supported sections of $\mathrm{Sph}(\mathrm{T}M)$.
McDuff \cite{McDuff75} was probably the first to study
the scanning map for general manifolds. She thought of it as the field of the point charges and
proved homological stability properties of this map.

When $\mathrm{T}M \cong M \times V$, the situation is simpler and we have defined a
scanning map in \autoref{eq:9}:
\begin{equation*}
s_{S^0}: \coprod_{k \geq 0} \mathscr{D}_{\myM}^{\mathrm{fr}_{V}}(k)/\Sigma_k \to \mathrm{Map}_{c}(M, S^V).
\end{equation*}
The left hand side is a model of the configuration space as justified in
\autoref{cor:conf}~\autoref{item:corconf1}; 
the right hand side is equivalent to the compactly supported
sections of $\mathrm{Sph}(\mathrm{T}M) \cong M \times S^V$.

We are interested in the scanning maps of Manthorpe--Tillman and McDuff, both of which can be made
equivariant without pain. The following table is a summary of the natural domains and codomains of
each construction:


\begin{figure}[h!]
  \centering
\begin{tabular}{lcc}
   scanning map & domain & codomain \\ \hline
  this paper, $s$ & framed embeddings $V$ to $M$ & maps $M^+ $ to $S^V$ \\ 
  Manthorpe--Tillman, $\tilde{s}^{\mathrm{MT}}$ & embeddings $V$ to $M$ & sections of
                                                                 $\mathrm{Sph}(\mathrm{T}M)$
  \\
  McDuff, $\tilde{s}^{\mathrm{MD}}$ & configuration of points of $M$ & sections of
                                                                 $\mathrm{Sph}(\mathrm{T}M)$
  \\
\end{tabular}
\end{figure}
In this section, we focus on the case of $V$-framed manifolds $M$.
Then these maps have equivalent domains and identical codomains. 
We will show in \autoref{prop:compareScanning} and \autoref{thm:htpyOfScanning} that:
\begin{thm}
  \label{cor:scanning-equi}
  The scanning maps  $s_X$, $s_X^{\mathrm{MD}}$ and $s_X^{\mathrm{MT}}$  are $G$-homotopic after
  the change of domain. 
\end{thm}

\begin{notn}
  In the above and subsequent paragraphs, 
\begin{itemize}
\item We use the letter $s$ for scanning maps without labels
  and $s_X$ for labels in $X$. 
\item A tilde is put on $s$ to denote when the codomain is the sections of
  $\mathrm{Sph}(\mathrm{T}M)$, that is, before composition with the framing. 
\item A superscript is put on $s$ to
  distinguish between the different authors in the literature. 
\end{itemize} 
\end{notn}

\subsection{Scanning map from tubular neighborhood}
\label{sec:Manthorpe-Tillman}
Manthorpe--Tillman \cite[Section 3.1]{MT14} gave a map
\begin{equation*}
  \gamma^{+}: 
  \big( \coprod_{k \geq 0} \mathrm{Emb}(\amalg_k \mathbb{R}^n, M) \times_{\Sigma_{k}} X^{k} \big) /\sim \ 
 \to \mathrm{Sect}_c(M, \mathrm{Sph}(\mathrm{T}M)\wedge_M
\tau_{X}).
\end{equation*}
Here,
$\mathrm{Sect}_c$ is the space of compactly supported sections;
$\tau_X$ is the constant parametrized base space $X \times M$ over $M$ 
and $\mathrm{Sph}(\mathrm{T}M)\wedge_M \tau_{X}$ is the fiberwise smashing of
$\mathrm{Sph}(\mathrm{T}M)$ with $X$.
(To translate, take their $M_0 = \varnothing$, $Y = W \times X$.
 Their $E_k(M,\pi)$ is the space $\mathrm{Emb}(\amalg_k \mathbb{R}^n, M) \times_{\Sigma_{k}}
 X^{k}$, and their $\Gamma(W \setminus M_0, W \setminus M, \pi)$ is $\mathrm{Sect}_c(M,
 \mathrm{Sph}(\mathrm{T}M)\wedge_M \tau_{X}) $.) 

 The key feature of their construction is to exploit the data of the tubular neighborhood, so
 a framing on $M$ is not needed. For example, when $k = 1$, we start
 with an embedding $f \in\mathrm{Emb}(\bR^n, M)$ and want to define $\gamma^+(f)$,
 a compactly supported section of $\mathrm{Sph}(\mathrm{T}M)$.
The image of $f$ is a tubular neighborhood of the image of $0 \in V$ in
$M$, and $f$ induces an inclusion of bundles $df: \mathrm{T}\bR^n \to \mathrm{T}M$. 
There is a canonical diagonal section $\bR^{n} \to \bR^n \times \bR^n \cong \mathrm{T}\bR^n$.
Pushing this section by $df$ gives $\gamma^+(f)$.

We can modify their $ \gamma^+$ by replacing $\mathbb{R}^n$ by the representation $V$ to get
\begin{equation*}
\gamma^+_V:\mathrm{Emb}_{\myM}(X) \equiv \big( \coprod_{k \geq 0} \mathrm{Emb}(\amalg_k V, M) \times_{\Sigma_{k}}
 X^{k} \big) /\sim \  \to \mathrm{Sect}_c(M, \mathrm{Sph}(\mathrm{T}M)\wedge_M
\tau_{X}).
\end{equation*}
We then precompose with the forgetting map 
$ \mathrm{D}_{\myM}^{\mathrm{fr}_V} (X) \to \mathrm{Emb}_{\myM}(X)$
in \autoref{rem:emb-data} to get
\begin{equation}
\label{equ:scanningX}
\tilde{s}_{X}^{\mathrm{MT}}: \mathrm{D}_{\myM}^{\mathrm{fr}_V} (X) \to \mathrm{Sect}_c(M, \mathrm{Sph}(\mathrm{T}M)\wedge_M
\tau_{X}).
\end{equation}
We describe how $\tilde{s}^{\mathrm{MT}}_X$ works on the subspace $k=1$ and it
is similar on the whole space. For the element $\bar{f}=(f, \alpha) \in
\mathrm{Emb}^{\mathrm{fr}_V}(V, M)$, we take the
embedding $f : V \to M$. The derivative map of $f$ is  $df : \mathrm{T}V \cong V \times V  \to \mathrm{T}M$.
For each $m \in \mathrm{image}(f)$, we need a vector $\tilde{s}^{\mathrm{MT}}(f) \in
\mathrm{T}_mM$ that is determined by $f$. Denote $v = f^{-1}(m) \in V$. We have
$df_v: V \cong \mathrm{T}_vV \to  \mathrm{T}_mM$. Then the explicit formulas 
without or with labels are given by
\begin{equation}
\label{eq:scanningMT}
\tilde{s}^{\mathrm{MT}}(\bar{f})(m) = df_{v}(v) \quad \text{and} \quad  \tilde{s}^{\mathrm{MT}}_X(\bar{f},x)(m) = df_v(v) \wedge
x.
\end{equation}
Both of them are $G$-maps.

The $V$-framing $\phi_M: \mathrm{T}M \to V$ induces
 $\mathrm{Sph}(\mathrm{T}M) \sm_M \tau_X \cong M \times \Sigma^VX$. So we obtain a map which we
 still call the scanning map:
\begin{equation}
\label{eq:defnsXMT}
s^{\mathrm{MT}}_{X} :\mathrm{D}_{\myM}^{\mathrm{fr}_V} (X) 
\to \mathrm{Map}_c(M, \Sigma^VX).
\end{equation}

A prior, this scanning map is different from the scanning map \autoref{eq:defnsX} in \autoref{sec:scanning}.
For an element $\bar{f}=(f,\alpha)$ where $f: V \to M$ with $f(v)=m$, we have
$s(\bar{f})(m) = v \in V$ in \autoref{eq:defnsX},
while $s^{\mathrm{MT}}(\bar{f})(m)=df_v(v) \in \mathrm{T}_mM$ in \autoref{eq:scanningMT}.
However, the data of a homotopy in defining the $V$-framed embedding ensure that the two
approaches give homotopic scanning maps:
\begin{prop}
  \label{prop:compareScanning}
  The map $s_{X}$ defined by \autoref{eq:defnsX} is $G$-homotopic to the map
  $s^{\mathrm{MT}}_{X}$ defined by \autoref{eq:scanningMT}.
\end{prop}

\begin{proof}
We show that $s \simeq s^{\mathrm{MT}}: \mathscr{D}_{\myM}^{\mathrm{fr}_{V}}(k) \to
\mathrm{Map}_{c}(M, S^{V})$.
We write the homotopy explicitly for $k=1$ and the case for general $k$ is similar.
To unravel the data, an element $\bar{f} =
(f, \alpha) \in \mathscr{D}_{\myM}^{\mathrm{fr}_{V}}(1)$
consists of an embedding $f: V \to M$ and a homotopy $\alpha$ of two maps
$\mathrm{T}V \to V$, where $\alpha(0)$ is the standard framing on $V$ and $\alpha(1)$ is $\phi_M \circ df$.

The two scanning maps use the two endpoints of this homotopy. Namely,
for $m$ in $\mathrm{Image}(f)$, write $v = f^{-1}(m) \in V\cong \mathrm{T}_vV$.
Then the first approach can be written as
\begin{equation*}
s(\bar{f})(m) = v = \alpha(0)_{v}(v)
\end{equation*}
and the $df$-shifted-approach can be written as
\begin{equation*}
s^{\mathrm{MT}}(\bar{f})(m) = \phi_M df_{v}(v) = \alpha(1)_{v}(v). 
\end{equation*}

Now it is clear that we can define a homotopy
\begin{equation*}
H: \mathscr{D}_{\myM}^{\mathrm{fr}_{V}}(1) \times I  \to \mathrm{Map}_{c}(M, S^{V});
\end{equation*}
\begin{equation*}
H(\bar{f},t)(m) = \alpha(t)_{f^{-1}(m)}(f^{-1}(m)).
\end{equation*}
It is $G$-equivariant and gives a homotopy
between $H(-,0) = s$ and $H(-,1) = s^{\mathrm{MT}}$.
The claim follows from observing that this homotopy is compatible with forgetting from $k$ to $k-1$ . 
\end{proof}

\subsection{Scanning map using geodesic}
\label{sec:mcduff}
 McDuff gave a geometric construction for
\begin{equation*}
\tilde{s}^{\mathrm{MD}}:  F_M(S^{0}) = \coprod_{k \geq 0} \mathscr{F}_M(k) \to
\mathrm{Sect}_c(M, \mathrm{Sph}(\mathrm{T}M)),
\end{equation*}
Recall that $\mathscr{F}_M(k)$ is the configuration space of $k$ points in $M$.
Note that the base point in each fiber of $\mathrm{Sph}(\mathrm{T}M)$ is the
point at infinity. A compactly supported section of $\mathrm{Sph}(\mathrm{T}M)$
is just a vector field defined in the interior of a compact set on $M$ that blows
up to infinity towards the boundary.

We first copy McDuff's construction and fit it into a neat comparison with the previously defined
scanning maps.
We focus on the case when $M$ is without boundary. Then we can translate her $M_{\epsilon}$ to our $M$;
her $E_M$ can be identified with our $\mathrm{Sph}(\mathrm{T}M)$;
her $\tilde{C}_M$ to our $F_M(S^0)$; her $\tilde{C}_{\epsilon}(M)$  to a subspace of our
$\mathrm{Emb}_M(S^0)$.

In summary, $\tilde{s}^{\mathrm{MD}}$ goes in two steps: fatten up the configurations (\cite[Lemma
2.3]{McDuff75}) and use geodesics to give compactly supported vector fields (\cite[p95]{McDuff75}).
\begin{equation}
\label{eq:defnsMD}
  \begin{tikzcd}
    \tilde{s}^{\mathrm{MD}}: F_M(S^0) \ar[r, " \mathrm{fatten}"] &
    \tilde{C}_{\epsilon}(M) \ar[r, "\phi_{\epsilon}"] \ar[d," \mathrm{include}"']
    & \mathrm{Sect}_c(M, E_M) \ar[d,"\eta_1","\cong"'] \\
    & \mathrm{Emb}_M(S^0) \ar[r,"\gamma^+"] & \mathrm{Sect}_c(M, \mathrm{Sph}(\mathrm{T}M))
  \end{tikzcd}
\end{equation}
The commutative diagram \autoref{eq:defnsMD} is central in this section. 
In the first row, $\mathrm{fatten}$ and $\phi_{\epsilon}$ are the two steps in McDuff's scanning
map. The map $\gamma^+$ is from \autoref{sec:Manthorpe-Tillman}.
We will define the undefined spaces and maps as we go along.

Define
\begin{align*}
\tilde{C}_{\epsilon}(M)_1 & \equiv \{\mathrm{exp}_{m_{0}}: \mathrm{T}_{m_{0}}M \to M
                            \text{ such that it is a diffeomorphism on the $\epsilon$-ball}\}; \\
\tilde{C}_{\epsilon}(M) & \equiv \{(\delta, e_1, \cdots, e_k) |0 < \delta \leq \epsilon, \ k \in \bN, e_i \in \tilde{C}_{\epsilon}(M)_1
                          \text{ for }1 \leq i \leq k, \\
  & \phantom{\equiv \{} \text{images of $e_i$ on the $\delta$-balls are disjoint in $M$}\}.
\end{align*}

For preparation, we write down an explicit homeomorphism 
\begin{equation*}
  \eta_{\epsilon}: D_{\epsilon}(\bR^n) \to \bR^n; \ v \mapsto \tan\big(\frac{\pi |v|}{2\epsilon}\big) \frac{v}{|v|}.
\end{equation*}
Here, $D_{\epsilon}(\bR^n)$ is the disk of radius $\epsilon$ in $\bR^n$. Then, abusively we also have
$$\eta_1: D_1(\mathrm{T}_mM)/\partial D_1(\mathrm{T}_mM) \cong \mathrm{T}_mM \cup \{\infty\}
\equiv \mathrm{Sph}(\mathrm{T}_mM).$$
Define $E_M$ to be the bundle over $M$ whose fiber over $m$ is $ D_1(\mathrm{T}_mM)/\partial
D_1(\mathrm{T}_mM)$, which is identified with $\mathrm{Sph}(\mathrm{T}_mM)$ through $\eta_1$.
This is the right vertical map in \autoref{eq:defnsMD}.

We give the vertical map in the middle of \autoref{eq:defnsMD}. For an element
$\mathrm{exp}_{m_{0}} \in  \mathrm{exp}_{m_{0}} $,
the composite $\mathrm{exp}_{m_{0}} \circ \eta_{\epsilon}^{-1}$ is an embedding
$\bR^n \to M$, so we can identify $\tilde{C}_{\epsilon}(M)_1$ with a subspace of
$\mathrm{Emb}(\bR^n, M)$. Similarly, we can include as subspace:
\begin{equation*}
  \begin{array}[h]{ccc}
    \tilde{C}_{\epsilon}(M) & \to & \mathrm{Emb}_M(S^0)\\
    (\delta,e_1,\cdots,e_k) & \mapsto & (e_1\circ \eta_{\delta}^{-1},\cdots,e_k\circ \eta_{\delta}^{-1})
  \end{array}
\end{equation*}

In McDuff's first step, let us define $\phi_{\epsilon}$ and compare it to the map $\gamma^+$ locally.
Put a Riemannian metric on $M$.
The input for $\phi_{\epsilon}$ are the exponential maps in $\tilde{C}_{\epsilon}(M)_1$.
Define
\begin{equation*}
\phi_{\epsilon}(\mathrm{exp}_{m_{0}})(m) = \begin{cases}
  * & \text{ if }\mathrm{dist}(m,m_0) > \epsilon; \\
  \frac{\mathrm{dist}(m,m_0)}{\epsilon} \cdot t(m,m_0) & \text{ if }\mathrm{dist}(m,m_0) \leq \epsilon.
\end{cases}
\end{equation*}
 Here, the values are vectors in $D_1(\mathrm{T}_mM)$; $t(m,m_0)$ is the unit tangent at $m$ of the minimal
 geodesic from $m_0$ to $m$; $\mathrm{dist}(m,m_0)$ is the distance between $m$ and $m_0$.
Now, it can be easily verified that
\begin{equation*}
\gamma^+( \mathrm{exp}_{m_{0}} \circ \eta_{\epsilon}^{-1} ) = \eta_{1} \circ \phi_{\epsilon}(\mathrm{exp}_{m_{0}}).
\end{equation*}
We can work the same way to extend $\phi_{\epsilon}$ to $\tilde{C}_{\epsilon}(M)$ and we have the
commutativity part of \autoref{eq:defnsMD}:
$$\gamma^+(e_1\circ \eta_{\delta}^{-1},\cdots,e_k\circ \eta_{\delta}^{-1}) = \eta_{1} \circ \phi_{\epsilon}(\delta,e_1,\cdots,e_k).$$

In McDuff's second step, we describe the fattening map in \autoref{eq:defnsMD}.
We can take a continuous positive function $\epsilon$ on $M$
such that for any $m_{0} \in M$, the exponential map $\mathrm{exp}_{m_{0}}: \mathrm{T}_{m_{0}}M \to M$
is always a diffeomorphism on the $\epsilon(m_{0})$-ball.
(We note the change here: $\epsilon(m_0)$ is going to serve as
the $\epsilon$ in the first step.
It does not harm to think as if $\epsilon(m_0) = \epsilon$ for all $m_0$.)
Then, as is easily checked, we can choose a continuous positive function $\bar{\epsilon}$ on
$F_M(S^0)$ such that at any $p=(m_1, \cdots, m_k) \in \mathscr{F}_M(k)$,
\begin{enumerate}[(i)]
\item\label{item:McDuff1} for all $i = 1, \cdots, k$, $\bar{\epsilon}(p) \leq
 \epsilon(m_{i})$ ;
\item\label{item:McDuff2} the $m_i$'s are at least $2 \bar{\epsilon}(p)$ apart
 from each other.
\end{enumerate}
The fattening map in \autoref{eq:defnsMD} sends
$p=(m_1, \cdots, m_k) \in \mathscr{F}_M(k)$ to $(\bar{\epsilon}(p), \mathrm{exp}_{m_1} , \cdots,
\mathrm{exp}_{m_k} ) \in \tilde{C}_{\epsilon}(M)$.
 The continuity of $\tilde{s}^{\textrm{MD}}$ follows from the continuity of $\bar{\epsilon}$.

 \begin{rem}
   \label{rem:sigma0}
  The composite 
\begin{equation*}
  \begin{tikzcd}
 F_M(S^0) \ar[r,"\mathrm{fatten}"] & \tilde{C}_{\epsilon}(M) \ar[r,"\mathrm{include}"] & \mathrm{Emb}_M(S^0)
  \end{tikzcd}
\end{equation*}
in \autoref{eq:defnsMD} is up to homotopy the $\sigma_0$ in \autoref{equ:sectionexp-at0}.
 \end{rem}
 
Equivariantly, we can take all of the Riemanian metric, $\epsilon$ and $\bar{\epsilon}$ to be $G$-invariant
because $G$ is finite: for example, replacing $\epsilon$ by $\Sigma_{g \in G} \epsilon(g-)/|G|$ will
do. Then $\tilde{s}^{\textrm{MD}}$ defined by \autoref{eq:defnsMD} is $G$-equivariant.
We can fiberwise smash with labels to get 
\begin{equation*}
\tilde{s}^{\textrm{MD}}_{X}:F_M(X)  \to  \mathrm{Sect}_c(M, \mathrm{Sph}(\mathrm{T}M) \wedge_M \tau_{X}).
\end{equation*}
We note that there is no $V$ involved in $\tilde{s}^{\textrm{MD}}_{X}$.
When $M$ is $V$-framed, we can compose it with the $V$-framing on $M$ to get
\begin{equation*}
s^{\textrm{MD}}_{X}:F_M(X)  \to  \mathrm{Map}_c(M, \Sigma^VX).
\end{equation*}

This scanning map $s^{\textrm{MD}}_{X}$ is good only for studying the configuration spaces, possibly
with labels.
It depends on the fattening-up radius $\bar{\epsilon}$,
which is not recorded explicitly in the data. The choice does not matter because a different choice of the fattening-up will give a homotopic scanning map.
But for the purpose of a scanning map out of ``configuration spaces with summable
labels'' or the factorization homology, remembering the radius is important to sum the labels.

We have seen three scanning maps so far: $s_X$ in~\autoref{eq:defnsX},
$s_X^{\mathrm{MT}}$ in~\autoref{eq:scanningMT} and $s_X^{\mathrm{MD}}$ in~\autoref{eq:defnsMD}. We
have shown that $s_X$ and $s_X^{\mathrm{MT}}$ are $G$-homotopic in \autoref{prop:compareScanning}.
We compare $s_X^{\mathrm{MD}}$ and $s_X^{\mathrm{MT}}$ in the following proposition.
\begin{prop}
  \label{thm:htpyOfScanning}
 The following diagram is $G$-homotopy commutative:
\begin{equation*}
  \begin{tikzcd}
     \mathrm{D}^{\mathrm{fr}_V}_{\myM} X \ar[r,"s_{X}^{\mathrm{MT}}"]&   \mathrm{Map}_c(M, \Sigma^V X)  \\
     F_M X \ar[from = u, "ev_0"] \ar[ur,"s^{\mathrm{MD}}_{X}"'] &
  \end{tikzcd}
\end{equation*}
\end{prop}

\begin{proof}
  Recall that $s_{X}^{\mathrm{MT}}$ is the composite of the forgetting map and $\gamma^+_V$:
  \begin{equation*}
   s_{X}^{\mathrm{MT}}: \mathrm{D}^{\mathrm{fr}_V}_{\myM} X \to   \mathrm{Emb}_{\myM}(X) \overset{\gamma^+_V}{ \to }  \mathrm{Map}_c(M, \Sigma^V X).
\end{equation*}
  By \autoref{eq:defnsMD} and \autoref{rem:sigma0}, 
  we have a homotopy commutative diagram:
\begin{equation*}
  \begin{tikzcd}    
     \mathrm{Emb}_{\myM}(X) \ar[r,"\gamma^+_V"]
  &  \mathrm{Map}_c(M, \Sigma^V X)  \\
     F_M(X) \ar[u,"\sigma_0"] \ar[ur,"s^{\mathrm{\textrm{MD}}}_X"'] & 
  \end{tikzcd}
\end{equation*}
By \autoref{cor:conf}\autoref{item:corconf2}, $\sigma_0 \circ ev_0$ is $G$-homotopic to the
forgetting map $ \mathrm{D}^{\mathrm{fr}_V}_{\myM} X \to   \mathrm{Emb}_{\myM}(X) $. So the claim follows.
\end{proof}



\subsection{Scanning equivalence}
We are interested in when the scanning map is an equivalence.
In this subsection, we list  Rourke--Sanderson's results from \cite{RS00}.
Their work is based on McDuff's scanning map. The $C_MX$ in their paper is our $(F_MX)^{G}$.

\begin{thm}
  \label{thm:RS}
  The scanning map $s_{X}^{\mathrm{MD}}: F_MX \to \mathrm{Map}_c(M, \Sigma^V X)$ is:
\begin{enumerate}
\item \label{item:RS1} a weak $G$-equivalence if  $X$ is $G$-connected,
\item \label{item:RS2} or a weak group completion if  $V \cong W \oplus \bR$ and $M \cong N \times
  \mathbb{R}$. Here, $W$ is a $(n-1)$-dimensional $G$-representation and $N$ is a $W$-framed 
  $G$-manifold, so that $N \times \bR$ is $V$-framed.
\end{enumerate}
\end{thm}
\begin{proof}
\autoref{item:RS1} is \cite[Theorem 5]{RS00}.
For \autoref{item:RS2}, we first note that when $M \cong N \times \mathbb{R}$, the map
$s^{\mathrm{\textrm{MD}}}_{X}$ factors in steps as:
\begin{eqnarray}
  \label{eq:8}
  & F_MX = F_{\bR}(F_NX)  &\to  \mathrm{Map}_{c}(\mathbb{R},\Sigma F_N(X)) \\
  \label{eq:28}& & \to \mathrm{Map}_{c}(\mathbb{R},F_N(\Sigma X)) \\
  \label{eq:17}&  & \to \mathrm{Map}_{c}(\mathbb{R},\mathrm{Map}_{c}(N, \Sigma^{1+W} X)).
\end{eqnarray}
Here, \autoref{eq:8} and \autoref{eq:17} are scanning maps for manifolds $\bR$ and $N$;
\autoref{eq:28} sends an element $p \sm t$ for a configuration $p$ on $N$ with labels in $X$ and $t
\in S^1$ to the same configuration on $N$ with labels suspended all by $t$ in $\Sigma X$.
 All spaces presented have $A_{\infty}$-structures from the factor $\bR$ in $M$: for any space $Y$,
 both the labeled configuration space $F_{\bR}Y$ and the mapping space
 $\mathrm{Map}_{c}(\mathbb{R},Y) \simeq \Omega Y$ have obvious $A_{\infty}$-structures.
 
The map \autoref{eq:17} is a weak $G$-equivalence by applying part~\autoref{item:RS1} with $M$
replaced by $N$ and $X$ replaced by $\Sigma X$, which is $G$-connected.
It suffices to show the composite of
\autoref{eq:8} and \autoref{eq:28}, denoted as $j$, is a weak group completion.

\cite[Theorem 3]{RS00} constructed a homotopy equivalence
$$q: \mathrm{B}\big((F_MX)^G\big) \simeq \big(F_N(\Sigma X)\big)^G.$$
Moreover, in Page 548, they established a homotopy commutative diagram:
\begin{equation*}
    \begin{tikzcd}
      (F_MX)^G \ar[r,"j^G"] \ar[d] & \mathrm{Map}_{c}(\mathbb{R}, \big(F_N(\Sigma X))\big)^G \ar[d, equal] \\
      \mathrm{Map}_{c}(\mathbb{R}, \mathrm{B}\big((F_MX)^G\big)) \ar[r,"\Omega q"] &
      \mathrm{Map}_{c}(\mathbb{R}, \big(F_N(\Sigma X)\big)^G)
    \end{tikzcd}
\end{equation*}
The left column is the group completion map for the $A_{\infty}$-space $(F_MX)^G$. Since $q$ is a
homotopy equivalence, $j^G$ is a weak group completion.
This remains true for any subgroup $H \subgroup  G$ replacing $G$. Therefore, $j$ is a weak group
completion.
\end{proof}

\begin{rem}
 \cite{RS00} does not assume the manifold $M$ to be framed. Without
  the framing on $M$, \autoref{thm:RS} is true in the following form:

  The scanning map $\tilde{s}_{X}^{\mathrm{MD}}: F_MX \to \mathrm{Sect}_c(M, \mathrm{Sph}(\mathrm{T}M)
  \sm_M \tau_X)$ is
\begin{enumerate}
\item \label{item:RS1} a weak $G$-equivalence if  $X$ is $G$-connected,
\item \label{item:RS2} or a weak group completion if $M \cong N \times
  \mathbb{R}$.
\end{enumerate}
\end{rem}

\section{A comparison of $\theta$-framed morphisms}
\label{chap:homspace}
In \autoref{sec:tangential}, we defined the $\theta$-framed embedding space
of $\theta$-framed bundles using paths in the $\theta$-framing. 
In this appendix, we compare this approach to an alternative definition
following Ayala--Francis \cite[Definition 2.7]{AF15} in  \autoref{prop:framed-embedding2}.
With this alternative definition, we identify the automorphism
$G$-space $\mathrm{Emb}^{\theta}(V,V)$ of $V$ in $\mathrm{Mfld}^{\theta}_{G,n}$ in \autoref{thm:autoV};
the special case $\theta = \mathrm{fr}_V$ has been treated directly in \autoref{sec:embeddingspace}.

\subsection{The $\theta$-framed maps}
\label{sec:embedding-space}

The classification theorem says that isomorphism classes of vector bundles are in bijection to
homotopy classes of maps to a classifying space.
Passing to the classification maps seems to lose the information about morphisms between bundles,
but it turns out not to. We show that the space of morphisms between bundles is equivalent to the space
of homotopies between their classifying maps in \autoref{cor:HomVSMapoverB}.
To this end, we first define a suitable ``over category up to homotopy''.

Let $B$ be a $G$-space. A typical example is to take $B=B_GO(n)$.
Then we have a $\mathrm{Top}$-enriched over category $\Gtop/_B$: the objects are $G$-spaces over
$B$, and the morphisms are $G$-maps over $B$. Explicitly, for $G$-spaces over $B$ given by $G$-maps
$\phi_M: M \to B$ and $\phi_N: N \to B$, the space $\mathrm{Hom}_{\Gtop/_B}(M,N)$ is the pullback
displayed in the following diagram:  (note that we have $\mathrm{Hom}_{\Gtop} = \mathrm{Map}_G$)
\begin{equation}
  \label{eq:23}
  \begin{tikzcd}
    \mathrm{Hom}_{\Gtop/_{B}}(M,N) \ar[r] \ar[d] & \mathrm{Map}_G(M,N) \ar[d,"{\phi_N\circ -}"] \\
    * \ar[r,"\{\phi_M\}"'] & \mathrm{Map}_G(M,B)
  \end{tikzcd}
\end{equation}
Now we want to work with $G$-spaces over $B$ up to homotopy. We modify the
morphism space by taking the homotopy pullback in \autoref{eq:23}. Just like the difference between
$\Gtop$ and $\topG$, we have two versions: the $\mathrm{Top}$-enriched $\Gtop\overB$ and the
$\Gtop$-enriched $\topG\overB$. That is, we have homotopy pullback diagrams of spaces in
\autoref{eq:24} and of $G$-spaces in \autoref{eq:25}:
\begin{equation}
  \label{eq:24}
  \begin{tikzcd}
    \mathrm{Hom}_{\Gtop\overB}(M,N) \ar[r] \ar[d] & \mathrm{Map}_G(M,N) \ar[d,"{\phi_N\circ -}"] \\
    * \ar[r,"\{\phi_M\}"'] & \mathrm{Map}_G(M,B)
  \end{tikzcd}
\end{equation}
\begin{equation}
  \label{eq:25}
  \begin{tikzcd}
    \mathrm{Hom}_{\topG\overB}(M,N) \ar[r] \ar[d] & \mathrm{Map}(M,N) \ar[d,"{\phi_N\circ -}"] \\
    * \ar[r,"\{\phi_M\}"'] & \mathrm{Map}(M,B)
  \end{tikzcd}
\end{equation}
Using the Moore path space model for the homotopy fiber as given in the following definition, one
can define unital and associative compositions to make $\Gtop\overB$ and
$\topG\overB$ categories.
\begin{defn}
  \label{defn:mapOverB}
  For $\phi_M : M \to B$ and $\phi_N: N \to B$, the space $\mathrm{Hom}_{\Gtop\overB}(M,N)$ and the
  $G$-space $\mathrm{Hom}_{\topG\overB}(M,N)$ are given by:
  \begin{align*}
    \mathrm{Hom}_{\Gtop\overB}(M,N)
      =\{(f, \alpha, l)| & f \in \mathrm{Map}_G(M,N), \alpha \in \mathrm{Map}({\bR_{\geq 0}}, \mathrm{Map}_G(M,B)), \\
                         &l \in \mathrm{Map}({\mathrm{Map}_G(M,N)}, {\bR_{\geq0}}) \text{ such that }  \\
    & l \text{ is locally constant}, \\
                         & \alpha(0) = \phi_M, \alpha(t) = \phi_N \circ f \text{ for } t \geq l(f)\}. \\
   \mathrm{Hom}_{\topG\overB}(M,N)
      =\{(f, \alpha, l)| & f \in \mathrm{Map}(M,N), \alpha \in \mathrm{Map}({\bR_{\geq 0}}, \mathrm{Map}(M,B)), \\
                         &l \in \mathrm{Map}({\mathrm{Map}(M,N)}, {\bR_{\geq0}}) \text{ such that }  \\
    & l \text{ is locally constant}, \\
    & \alpha(0) = \phi_M, \alpha(t) = \phi_N \circ f \text{ for } t \geq l(f)\}. \\
\end{align*}
\end{defn}


\begin{rem}
  Roughly speaking, a point in the morphism space $\Gtop\overB$ is a $G$-map $f \in
  \mathrm{Map}_G(M,N)$ and a $G$-homotopy from $\phi_M$ to $\phi_N \circ f $ in the following diagram:
  \begin{equation*}
  \begin{tikzcd}
    & N \ar[d, "\phi_N"] \\
  M \ar[r, "\phi_M"'] \ar[ur, dotted, "f"] & B
\end{tikzcd}
\end{equation*}
  A point in the morphism space $\topG\overB$ is a map $f \in \mathrm{Map}(M,N)$ and a 
  homotopy from $\phi_M$ to $\phi_N \circ f $; the map $f$ is not necessarily a $G$-map, but we do
  require $\phi_M$ and $\phi_N$ to be $G$-maps. And we have
\begin{equation*}
 \mathrm{Hom}_{\Gtop\overB}(M,N) \cong ( \mathrm{Hom}_{\topG\overB}(M,N))^G.
\end{equation*}
\end{rem}

\medskip
The category $\topG\overB$ models $\theta$-framed bundles:
\begin{prop}
  \label{prop:mapbeta}
  For $i=1,2$, let $E_i \to B_i$ be $G$-$n$-vector bundles with $\theta$-framings ${\phi_i: E_i \to
  \theta^{*}\universal}$.
  We have the following equivalences of $G$-spaces
  that are natural with respect to the two variables as well as the tangential structure:
  \begin{equation*}
    \beta: \mathrm{Hom}^{\mathrm{\theta}}(E_1, E_2) \overset{\sim}{ \longrightarrow }
    \mathrm{Hom}_{\topG\overB}(B_1,B_{2}).
  \end{equation*}
\end{prop}




\begin{proof}

One can restrict bundle maps to get maps on the base spaces. We denote this restriction map
by $\pi$. From our definition of
$\mathrm{Hom}^{\theta}$ in \autoref{defn:theta-bundle-map} and $\mathrm{Hom}_{\topG\overB}$ in
\autoref{defn:mapOverB}, $\pi$ induces the map $\beta$ and they fit in the following commutative
 diagram of $G$-spaces:
 \begin{equation}
   \label{eq:bundlecompare}
  \begin{tikzcd}
  \mathrm{Hom}^{\theta}(E_1,E_2) \ar[r, "\beta", "\sim"'] \ar[d, two heads]
  & \mathrm{Hom}_{\topG\overB}(B_1, B_2)  \ar[d,
    two heads] \\
    \mathrm{Hom}(E_1, E_2) \ar[r,"\pi"]  \ar[d, "\phi_2 \circ-"'] \ar[rd, phantom, "\lrcorner", very near
    start ]& \mathrm{Map}(B_1, B_2) \ar[d, "\phi_2 \circ-"] \\
    \mathrm{Hom}(E_1,\theta^{*}\universal) \ar[r,"\pi"] & \mathrm{Map}(B_{1},B)
  \end{tikzcd}
\end{equation}
We claim that the bottom square is a pullback. Since each column is a homotopy fiber sequence,
this implies immediately that $\beta$ is a $G$-equivalence.

To show the claim, first we note that the isomorphism $\phi_2: E_2 \cong \phi_2^{*}\theta^{*}\universal$
establishes $E_2$ as a pullback of
$\theta^{*}\universal$ over $\phi_{2}$. So a bundle map $E_1 \to E_2$ is determined by a map on the base
$f: B_1 \to B_2$ and a bundle map $(\bar{\varphi},\varphi): (E_1,B_1) \to (\universal,B)$
satisfying $\varphi = \phi_2 f $.
\begin{equation*}
  \begin{tikzcd}
    E_1 \ar[r,dotted] \ar[rr, bend left, dotted, "\bar{\varphi}"] \ar[d] & E_2 \ar[r] \ar[d]
    \ar[rd, phantom, "\lrcorner", very near start ]
    & \theta^{*}\universal \ar[d] \\
    B_1 \ar[r, "f"'] & B_2 \ar[r, "\phi_2"'] & B
  \end{tikzcd} \qedhere
\end{equation*}
\end{proof}

We remark that in \autoref{prop:mapbeta},
$\pi$ is not a homotopy equivalence to its image.
In other words, a vector bundle map is not just a map on the bases.
In contrast, a $\theta$-framed vector bundle map can be seen as a
map on the bases as $\beta$ is an equivalence. 

\begin{lem}(\cite[Lemma 3.18]{ZouBundle})
  \label{lem:contractible} Let  $p: P \to B$ be any principal $G$-$\Pi$-bundle and
  $\mathrm{Hom}(P,E_G\Pi)$ be  the space of (non-equivariant) principal  $\Pi$-bundle
  morphisms with $G$ acting by conjugation. $\mathrm{Hom}(P,E_G\Pi)$ is $G$-contractible.
\end{lem}

The ``classical'' bundle maps are the $\theta$-framed bundle maps for the
tangential structure $\theta = \mathrm{id}: B_GO(n) \to B_GO(n)$:
\begin{lem}
  \label{prop:mapalpha}
  For $G$-vector bundles $E_i \to B_i,\ i=1,2$, we have an equivalence of $G$-spaces:
  \begin{equation*}
    \alpha: \mathrm{Hom}^{\mathrm{id}}(E_1, E_2)  \overset{\sim}{ \longrightarrow } \mathrm{Hom}(E_1, E_2).
  \end{equation*}
\end{lem}

\begin{proof}
  By definition, $\mathrm{Hom}^{\mathrm{id}}(E_1,E_2)$ is the homotopy fiber of $\phi_2 \circ -$, so
  we have a homotopy fiber sequence of $G$-spaces:
\begin{equation*}
  \begin{tikzcd}
  \mathrm{Hom}^{\mathrm{id}}(E_1,E_2) \ar[r, two heads, "\alpha"] &
    \mathrm{Hom}(E_1, E_2) \ar[r, "\phi_2 \circ-"] &
    \mathrm{Hom}(E_1,\universal)
  \end{tikzcd}.
\end{equation*}
By \autoref{lem:contractible}, we know $\mathrm{Hom}(E_1,\universal)$ is $G$-contractible.
So $\alpha$ is a $G$-equivalence.
\end{proof}

\begin{cor}
  \label{cor:HomVSMapoverB}
  For $G$-vector bundles $E_i \to B_i,\ i=1,2$, we have an equivalence of $G$-spaces:
  \begin{equation*}
    \mathrm{Hom}(E_1, E_2) \simeq \mathrm{Hom}_{\topG\overB[B_GO(n)]}(B_1, B_2).
  \end{equation*}
\end{cor}
\begin{proof}
  This follows from \autoref{prop:mapbeta} and \autoref{prop:mapalpha}.
\end{proof}

\medskip
\begin{prop} \label{prop:framed-embedding2}
  The $G$-space $\mathrm{Emb}^{\theta}(M,N)$ as defined in \autoref{def:embedding} is the homotopy 
  pullback displayed in the following diagram of $G$-spaces:
\begin{equation}
    \label{eq:emb-space2}
  \begin{tikzcd}
    \mathrm{Emb}^{\theta}(M,N) \ar[r] \ar[d] & \mathrm{Hom}_{\topG\overB}(M,N) \ar[d] \\
    \mathrm{Emb}(M,N) \ar[r] & \mathrm{Hom}_{\topG\overB[B_{G}O(n)]}(M,N)
  \end{tikzcd}
\end{equation}
\end{prop}
\begin{proof}
  The lower horizontal map in \autoref{eq:emb-space2} is neither obvious nor canonical.
  We take it as the composite in the following commutative diagram with a chosen $G$-homotopy
  inverse to $\alpha$. The maps $\alpha$ and $\beta$ are $G$-equivalences by \autoref{prop:mapbeta} and \autoref{prop:mapalpha}.
 \begin{equation}
   \label{eq:comparison}
  \begin{tikzcd}
    \mathrm{Emb}^{\theta}(M,N) \ar[r] \ar[dd] & \mathrm{Hom}^{\mathrm{\theta}}(\mathrm{T}M,
    \mathrm{T}N) \ar[r, "\sim"', "\beta" ] \ar[d] & \mathrm{Hom}_{\topG\overB}(M,N) \ar[d] \\
    & \mathrm{Hom}^{\mathrm{id}}(\mathrm{T}M, \mathrm{T}N)\ar[r, "\sim"', "\beta"] \ar[d, "\sim", "\alpha"'] &
    \mathrm{Hom}_{\topG\overB[B_{G}O(n)]}(M,N) \\
    \mathrm{Emb}(M,N) \ar[r, "d"] & \mathrm{Hom}(\mathrm{T}M, \mathrm{T}N) &
  \end{tikzcd}
\end{equation}
As defined in \autoref{def:embedding}, $\mathrm{Emb}^{\theta}(M,N)$ is the pullback in the left square.
It is clear that it is also equivalent to the homotopy pullback of the whole square.
\end{proof}

We can take \autoref{eq:emb-space2} as an alternative definition to \autoref{eq:emb-space1}.
In practice, \autoref{eq:emb-space1} is easier to deal with.
First, the right vertical map in the square is a fibration so the diagram is an
actual pullback. Second, the map $d$ is easy to describe. On the other hand, \autoref{eq:emb-space2}
has a conceptual advantage.
It can be viewed as a comparison of the $\theta$-framing to the trivial framing $\mathrm{id}:
{B_{G}O(n)} \to B_GO(n)$.

\subsection{Automorphism space of $(V,\phi)$}
\label{sec:auto}
With this alternative description of $\theta$-framed mapping spaces in
\autoref{sec:embedding-space}, we can identify the automorphism
$G$-space $\mathrm{Emb}^{\theta}(V,V)$ of $V$ in $\mathrm{Mfld}^{\theta}_{G,n}$ by
first identifying of the automorphism $G$-space $\mathrm{Hom}^{\theta}(\mathrm{T}V,
\mathrm{T}V)$ of $\mathrm{T}V$ in $\mathrm{Vec}^{\theta}_{G,n}$.

\begin{notn}
As $\phi$ is an equivariant map, $\phi(0)$ for the origin $0 \in V$ is a $G$-fixed point in $B$.
We denote by $\moore_{\phi}B$ the Moore loop space of $B$ at the base point $\phi(0)$.
\end{notn}

\begin{thm}
  \label{thm:autoV} We have the following: 
\begin{enumerate}
\item  \label{lem:LoopB}
There is an equivalence of monoids in $G$-spaces
\begin{equation*} 
\mathrm{Hom}^{\theta}(\mathrm{T}V, \mathrm{T}V) \overset{\sim}{ \to } \moore_{\phi} B,
\end{equation*}
which is natural with respect to tangential structures $\theta: B \to B_GO(n)$.
Here, the group $G$ acts on both sides by conjugation.

\item \label{item:autoV}
  The automorphism $G$-space $\mathrm{Emb}^{\theta}(V,V)$ of $(V,\phi)$ in $\mathrm{Mfld}^{\theta}_{G,n}$ fits in the following homotopy
pullback diagram of $G$-spaces:
\begin{equation*}
  \begin{tikzcd}
    \mathrm{Emb}^{\theta}(V,V) \ar[d] \ar[r] & \moore_{\phi} B \ar[d]\\
    \mathrm{Emb}(V,V) \ar[r,"d_0"'] & O(V)
  \end{tikzcd}
\end{equation*}
Consequently, $\mathrm{Emb}^{\theta}(V,V)\simeq \moore_{\phi} B$.
\end{enumerate}

\end{thm}

\begin{proof}
  \autoref{lem:LoopB}
  We have $ \mathrm{Hom}_{\topG\overB}(V,V)$ from \autoref{defn:mapOverB} and showed in
  \autoref{prop:mapbeta} that restriction-to-the-base gives a natural $G$-equivalence:
\begin{equation*} 
\beta: \mathrm{Hom}^{\theta}(\mathrm{T}V, \mathrm{T}V) \overset{\sim}{\to} \mathrm{Hom}_{\topG\overB}(V,V).
\end{equation*}
  Let $*$ be the $G$-space over $B$ given by $\phi(0): * \to B$.
  We claim that the two maps $\mathrm{inc}: 0 \to V$ and $\mathrm{proj}: V \to
  *$ can be lifted to give equivalences of $V \simeq *$ in $\topG\overB$.
  If so, pre-composing with $\mathrm{inc}$ and post-composing with $\mathrm{proj}$ give
\begin{equation*}
\mathrm{Hom}_{\topG\overB}(V,V) \overset{\sim}{ \to } \mathrm{Hom}_{\topG\overB}(*, *) \cong \moore_{\phi} B.
\end{equation*}

It remains to verify the claim, which is a routine job.
We choose the lifts of $\mathrm{inc}$ and $\mathrm{proj}$ given by 
$$I=(\mathrm{inc}, \alpha_1, 0) \in \mathrm{Hom}_{\topG\overB}(*,V),
\text{ where }\alpha_1(t) = \phi(0) \text{ for all } t \geq 0.$$
$$P=(\mathrm{proj}, \alpha_2, 1) \in \mathrm{Hom}_{\topG\overB}(V,*),
\text{ where } \alpha_2(t) =
\begin{cases}
  \phi \circ h_t, & 0 \leq t < 1; \\
  \phi(0), & t \geq 1;
\end{cases}$$ where $h_t:V \to V$ is any chosen homotopy from $h_0 = \mathrm{id}$ to $h_1 = \mathrm{proj}$. 
Then we have an obvious homotopy:
\begin{equation*}
P \circ I = (\mathrm{id}, \mathrm{const}_{\phi(0)}, 1) \simeq (\mathrm{id}, \mathrm{const}_{\phi(0)}, 0) =
\mathrm{id}_{*}
\end{equation*}
and using the contraction $h_t$, we can also construct a homotopy:
\begin{equation*}
I \circ P = (\mathrm{proj}, \alpha_2, 1) \simeq (\mathrm{id}, \mathrm{const}_{\phi}, 0) =
\mathrm{id}_{V}. \qedhere
\end{equation*}

\autoref{item:autoV}
  This is an assembly of part~\autoref{lem:LoopB}, \autoref{prop:framed-embedding2} and
  \autoref{cor:monoidmap}. However, we note that the map $\moore_{\phi}B \to O(V)$
  is only a non-canonical $G$-equivalence. The author does not know how to upgrade it to a map of
  $G$-monoids.
  So although all spaces displayed in the pullback diagram are $G$-monoids, it is not obvious
  whether one can write $\mathrm{Emb}^{\theta}(V,V)$ as a pullback of $G$-monoids.

  To be more precise, we show how the quoted results assemble. We have the following large
  commutative diagram \autoref{eq:large} expanding \autoref{eq:comparison}.
  Note that this is a commutative diagram of $G$-monoids.
\begin{equation}\small
  \label{eq:large}
  \begin{tikzcd}
    \mathrm{Emb}^{\theta}(V,V) \ar[r] \ar[dd] &
    \mathrm{Hom}^{\theta}(\mathrm{T}V, \mathrm{T}V)\ar[r, "\sim"', "\beta"] \ar[d]
    \ar[rd,phantom,"\textcircled{1}"] & \mathrm{Hom}_{\topG\overB}(V,V) \ar[rd,"\sim"] \ar[d] & \\
    &
    \mathrm{Hom}^{\mathrm{id}}(\mathrm{T}V, \mathrm{T}V)\ar[r, "\sim"', "\beta"] \ar[d, "\sim",
    "\alpha"'] \ar[rd,"\sim"] &
    \mathrm{Hom}_{\topG\overB[B_{G}O(n)]}(V,V) \ar[rd,"\sim"]
    \ar[r,phantom, "\textcircled{2}"] \ar[d,phantom,"\textcircled{3}"]
    & \moore_{\phi}B \ar[d]\\
    \mathrm{Emb}(V,V) \ar[r] &
    \mathrm{Hom}(\mathrm{T}V, \mathrm{T}V) \ar[rd, "\sim"] \ar[r, phantom, "\textcircled{4}"]
    & \mathrm{Hom}^{\mathrm{id}}(V, V)\ar[r, "\sim"', "\beta"] \ar[d, "\sim", "\alpha"']
    & \moore_{\phi}B_{G}O(n) \\
    & & \mathrm{Hom}(V,V) = O(V) & \\
  \end{tikzcd}
\end{equation}
The map $\alpha$ is studied in \autoref{prop:mapalpha}. The map $\beta$ and the square \textcircled{1} are
in \autoref{prop:mapbeta}. The diagonal unlabeled maps are all induced by the inclusion
$V \to \mathrm{T}V$ and the projection $\mathrm{T}V \to V$. In particularly, the parallelogram
\textcircled{2} is in part~\autoref{lem:LoopB}. Naturality of $\alpha$ and $\beta$ gives the commutativity of
\textcircled{3} and \textcircled{4}.
Now, $d_0$ in the theorem is the composite
\begin{equation*}
  \begin{tikzcd}
    \mathrm{Emb}(V,V) \ar[r, "d"] & \mathrm{Hom}(\mathrm{T}V, \mathrm{T}V) \ar[r,"\sim"] & \mathrm{Hom}(V,V).
  \end{tikzcd}
\end{equation*}
It can be seen that the vertical map in the theorem involves choosing an inverse of the $\beta$
displayed in the third line.
\end{proof}

\begin{rem}
  \label{rem:zigzag}
The equivalence \autoref{cor:monoidmap} is  hidden in the following part of \autoref{eq:large}:
\begin{equation*}
  \begin{tikzcd}
     \moore_{\phi}B_{G}O(n) & \mathrm{Hom}^{\mathrm{id}}(V, V)\ar[l, "\sim", "\beta"'] \ar[r, "\sim"', "\alpha"]
     & \mathrm{Hom}(V,V) = O(V).
  \end{tikzcd}
\end{equation*}
(See also \cite[4.4.12, 5.3.4]{mythesis}.)
\end{rem}
\bibliographystyle{alpha}
\bibliography{../project/factorization}

\end{document}